\documentclass[11pt]{amsart}
\usepackage{latexsym}
\usepackage{amsmath}
\usepackage{amssymb}

\newcommand{\eproof}{\mbox{\ }\hfill $\Box$ \par \vskip 10pt}

\newtheorem{Theorem}{Theorem}[section]
\newtheorem{lemma}[Theorem]{Lemma}
\newtheorem{prop}[Theorem]{Proposition}
\newtheorem{rem}[Theorem]{Remark}

\usepackage[usenames,dvipsnames]{xcolor}
\numberwithin{equation}{section}

\def\C{{\mathbb C}}
\def\R{{\mathbb R}}
\def\Re{{\rm Re}\:}
\def\Im{{\rm Im}\:}

\def\cal{\mathcal}

\newcommand*{\defeq}{\mathrel{\vcenter{\baselineskip0.5ex \lineskiplimit0pt

                     \hbox{\scriptsize.}\hbox{\scriptsize.}}}%
                     =}

\begin{document}

\title[Exponential local energy decay]{Exponential local energy decay of solutions to the wave equation with $L^\infty$ electric and
magnetic potentials}

\author[A. Larra\'in-Hubach, J. Shapiro, G. Vodev]{Andr\'es Larra\'in-Hubach, Jacob Shapiro and Georgi Vodev}

\address {Department of Mathematics, University of Dayton, Dayton, OH 45469-2316, USA}
\email{alarrainhubach1@udayton.edu}

\address {Department of Mathematics, University of Dayton, Dayton, OH 45469-2316, USA}
\email{jshapiro1@udayton.edu}

\address {Universit\'e de Nantes, Laboratoire de Math\'ematiques Jean Leray, 2 rue de la Houssini\`ere, BP 92208, 44322 Nantes Cedex 03, France}
\email{Georgi.Vodev@univ-nantes.fr}

\date{}

\begin{abstract} We prove sharp resolvent estimates for the magnetic Schr\"odinger operator in $\R^d$, $d\ge 3$,
with $L^\infty$ short-range electric and magnetic potentials. We also show that these resolvent estimates still hold for the
Dirichlet self-adjoint realization of the Schr\"odinger operator in the exterior of a non-trapping obstacle in $\R^d$, $d\ge 2$,
provided the magnetic potential is supposed identically zero. As an application of the resolvent estimates, we obtain an 
exponential decay of the local energy of solutions to the wave equation with $L^\infty$ electric and
magnetic potentials which decay exponentially at infinity,
in all odd and even dimensions, provided the low frequencies are cut off in a suitable way.
We also show that in odd dimensions there is no need to cut off the low frequencies in order to get an exponential local energy decay,
provided zero is neither an eigenvalue nor a resonance. 


\end{abstract} 

\maketitle

\setcounter{section}{0}

\section{Introduction} \label{s1}
Let $\cal{O}\subseteq\mathbb{R}^d$, $d\ge 2$, be a possibly empty, bounded domain with smooth boundary such that 
$\Omega=\mathbb{R}^d\setminus\cal{O}$ is connected. In this paper we investigate the magnetic Schr\"odinger operator
\begin{equation} \label{eq:1.1}
P=(i\nabla+b(x))^2+V(x) : L^2(\Omega) \to L^2(\Omega)
\end{equation}
from the viewpoint of resolvent estimates. The magnetic potential $b : \R^d \to \R^d$ and electric potential $V : \R^d \to \R$ are assumed to have $L^\infty$ regularity. Leveraging these resolvent estimates, our primary goal is to establish conditions for exponential weighted energy decay for solutions to the associated wave equation
\begin{equation} \label{eq:1.2}
\begin{cases}
(\partial^2_t + P)u(t,x) = 0 & \text{in } \R \times \Omega, \\
u(t,x) = 0 & \text{on } \R \times \partial \Omega, \\
u(0,x) = f_1(x), \, \partial_tu(0,x) = f_2(x) & \text{in } \Omega.  
\end{cases}
\end{equation}

To set the stage, let us recall the classical results for the free wave equation, where both $b$ and $V$ are identically zero. In this simpler setting, if $\mathcal{O} = \emptyset$ (so $\Omega = \R^d$), Huygen's principle implies that when $d \ge 3$ is odd, the energy of the solution to \eqref{eq:1.2} within any fixed compact set decays to zero in finite time.  On the other hand, when $\Omega \neq \emptyset$ (i.e, $\Omega$ is an exterior domain, with $b$ and $V$ still vanishing), the decay of local energy for solutions to \eqref{eq:1.2} is related to the dynamics of the underlying Hamiltonian flow. The non-trapping condition, where all geodesics escape to infinity, is well known to be related to rapid energy decay.  This condition is known to hold for specific geometries, such as convex obstacles and more generally for obstacles where an escape function can be constructed. Foundational works by Lax, Morawetz, Phillips, Ralston and Strauss \cite{kn:lmp, kn:M1, kn:M2, kn:MRS} address such scenarios and the resulting decay, utilizing multiplier methods and associated properties of the resolvent of the Laplacian. More broadly, for non-trapping geometries, the resolvent satisfies a characteristic high frequency bound \cite[Theorem 4.43]{kn:DZ}. Our assumption for obstacles, stated as \eqref{eq:1.4} below, is a resolvent estimate of this type.

When non-zero potentials $b$ and $V$ are present, results on energy decay draw from the work of Vainberg \cite{kn:Va} and Melrose-Sj\"ostrand \cite{kn:MS1, kn:MS2}. For example, in the case of smooth potentials of compact support, the local energy is known to decay like $O(e^{-Ct})$ for some $C > 0$ when $d \ge 3$ is odd, and like $O(t^{-2d})$ when $d \ge 2$ is even. Vainberg \cite{kn:Va} showed these decay rates apply to compactly supported perturbations of the Laplacian satisfying the Generalized Huygens Principle (as defined in \cite{kn:V2a}). Work by Melrose and Sj\"ostrand on the propagation of singularities \cite{kn:MS1, kn:MS2} further implies that this principle is satisfied by a broad class of smooth non-trapping perturbations of the Laplacian, including those with smooth compactly supported potentials.

We now specify the assumptions for our analysis. In what follows, $\|\cdot\|$ and
$\|\cdot\|_1$ denote the operator norms $L^2(\Omega)\to L^2(\Omega)$ and $H^1(\Omega)\to L^2(\Omega)$, respectively.  We consider two main scenarios:\\
a) $\cal{O} = \emptyset$ (so that $\Omega = \R^d$), $d\ge 3$,\\
b) $\cal{O} \neq \emptyset$ and $b\equiv 0$. \\
For both cases we assume the potentials satisfy:
\begin{equation}\label{eq:1.3}
|V(x)|+|b(x)|\le Ce^{-c\langle x\rangle},
\end{equation}
where $\langle x \rangle \defeq (|x|^2+1)^{1/2}$ and $C,c>0$ are some constants. In case b) (exterior domain, $V$ only), we impose a non-trapping condition on the obstacle $\mathcal{O}$ via a high frequency resolvent estimate for the Dirichlet Laplacian $\widetilde P =-\Delta$ on $L^2(\Omega)$. Let $\chi\in C^\infty(\mathbb{R}^d; [0,1])$ be of compact support such that $\chi =1$ near $\overline{\mathcal{O}}$. Define multiplication by $\chi$ on $L^2(\Omega)$ by $u \mapsto \chi \rvert_\Omega u$. The operator $\widetilde P$ can be viewed as a \textit{black box Hamiltonian} in the sense of Sj\"ostrand and Zworski \cite{kn:SZ91}, as defined in \cite[Definition 4.1]{kn:DZ}. By the analytic Fredholm theorem \cite[Theorem 4.4]{kn:DZ}, the cutoff resolvent $\chi (\widetilde P - \lambda^2)^{-1} \chi: L^2(\Omega) \to D(\widetilde{P})$ continues meromorphically from $\{\text{Im} \lambda < 0\}$ to the whole complex plane $\mathbb{C}$ if $d$ is odd, and to the Riemann surface of the logarithm if $d$ is even. The poles of this continuation are called \textit{resonances}. Our non-trapping condition for case b) is the high-frequency bound:
\begin{equation}\label{eq:1.4}
\left\|\chi(\widetilde P-\lambda^2)^{-1}\chi\right\|\le C\lambda^{-1},\quad \lambda  \ge\lambda_0,
\end{equation}
for some constants $C,\lambda_0>0$. We note that the norm in \eqref{eq:1.4} is bounded by a constant also for $0\le\lambda \le\lambda_0$ for arbitrary obstacles, as proved in \cite[Appendix B]{kn:B1}. No separate non-trapping condition is imposed for case a) beyond the condition \eqref{eq:1.3}.

Hereafter, $P$ denotes the self-adjoint realization of the operator $(i\nabla+b)^2+V$ on the Hilbert space
$L^2(\mathbb{R}^d)$ in the case a), and the Dirichlet self-adjoint realization of $-\Delta+V$ in the case b). For the case a), Appendix \ref{self-adjointness mag Schro appendix} details the construction of $P$ via a quadratic form on $H^1(\R^d)$. For case b), we recall from \cite[Section 6.1.2]{kn:Bo} that the domain of $P$ is the intersection $H^1_0(\Omega) \cap H^2(\Omega)$ of Sobolev spaces (we define $H^1_0(\Omega)$ as the closure in $H^1$-norm of smooth and compactly supported functions on $\Omega$). In case b) we will at times make use of the Green formula
\begin{equation*}
\langle \nabla u, \nabla v \rangle_{L^2(\Omega)} = \langle u, -\Delta v\rangle_{L^2(\Omega)}, \qquad u, \, v \in H^1_0(\Omega) \cap H^2(\Omega). 
\end{equation*} 

In both scenarios, we assume $P\ge 0$, for which a sufficient condition is $V\ge 0$. The domain of the square root of a nonnegative self-adjoint operator coincides with its quadratic form domain \cite[Section 3.1, (3.52) and (3.53)]{kn:Te}. Consequently, in case a), the form domain is $H^1(\R^d)$, while in case b) it is $H^1_0(\Omega)$. 

The solution to the wave equation (\ref{eq:1.2}) is given by
\begin{equation}\label{eq:1.5}
u(t)=\cos\left(t\sqrt{P}\right)f_1+P^{-1/2}\sin\left(t\sqrt{P}\right)f_2.
\end{equation}
We define the weight $\mu(x)=e^{-c\langle x\rangle/2}$. For some of our results it is important to suppose that zero is neither an eigenvalue nor a resonance of $P$.
More precisely, we require the following low frequency resolvent bound: there exist constants $C> 0$ and $0<\delta_0, \, \varepsilon_0 \ll 1$ such that
\begin{equation}\label{eq:1.6}
\sum_{\ell=0}^1\left\|\mu\nabla^\ell(P-\lambda^2\pm i\varepsilon)^{-1}\mu\right\|\le C,\qquad 0<\lambda\le\delta_0, \, 0 < \varepsilon \le \varepsilon_0.
\end{equation}
The condition \eqref{eq:1.6} is established in Section 5 in dimensions $d\ge 5$, under suitable 
short range conditions on $b$ and $V$, and provided $V \ge 0$ (see \eqref{eq:5.1}). Our main result is then:

\begin{Theorem} \label{1.1}
Assume the conditions \eqref{eq:1.3} and \eqref{eq:1.4} fulfilled. Then, given $t\gg 1$ and any 
$0<\delta\ll 1$ (independent of $t$), there exists a real-valued function
 $\psi_{\delta,t}\in C^\infty(\mathbb{R})$, $0\le\psi_{\delta,t}\le 1$, $\psi_{\delta,t}(\sigma)=0$ for $\sigma\le\delta$,
 $\psi_{\delta,t}(\sigma)=1$ for $\sigma\ge \delta\ln t+C_\delta$, for some constant $C_\delta>0$ independent of $t$, so that the estimates 
 \begin{equation}\label{eq:1.7}
\left\|\mu\cos(t\sqrt{P})\psi_{\delta,t}(P^{1/2})\mu\right\|+\sum_{\ell=0}^1
\left\|\mu\nabla^\ell P^{-1/2}\sin(t\sqrt{P})\psi_{\delta,t}(P^{1/2})\mu\right\|\le C_1e^{-c_1t},
\end{equation}
\begin{equation}\label{eq:1.8}
\left\|\mu P^{1/2}\sin(t\sqrt{P})\psi_{\delta,t}(P^{1/2})\mu\right\|_1+\sum_{\ell=0}^1
\left\|\mu\nabla^\ell \cos(t\sqrt{P})\psi_{\delta,t}(P^{1/2})\mu\right\|_1\le C_1e^{-c_1t}
\end{equation}
 hold with constants $C_1,c_1>0$ depending on $\delta$ but independent of $t$. 
 
 Furthermore, there exists a real-valued function
 $\psi_{\delta,t}^\sharp\in C^\infty(\mathbb{R})$, $0\le\psi_{\delta,t}^\sharp\le 1$, $\psi_{\delta,t}^\sharp(\sigma)=0$ for $\sigma\le\delta$,
 $\psi_{\delta,t}^\sharp(\sigma)=1$ for $\sigma\ge 2\delta$, so that the estimates (\ref{eq:1.7}) and (\ref{eq:1.8}) 
 hold with $\psi_{\delta,t}$ and $e^{-c_1t}$ replaced by $\psi_{\delta,t}^\sharp$ and $e^{-c_1t/\ln t}$, respectively. 
 
 If the dimension $d$ is odd and the condition
 (\ref{eq:1.6}) is assumed, then the estimates (\ref{eq:1.7}) and (\ref{eq:1.8}) hold with $\psi_{\delta,t}$ replaced by $1$. 
 \end{Theorem}
 
 \begin{rem} \label{1.2}
As an immediate consequence of Theorem \ref{1.1} and the formula (\ref{eq:1.5}), we obtain in odd dimensions,
under the condition (\ref{eq:1.6}), an exponential decay
of the local energy of the solution of the wave equation (\ref{eq:1.2}) with initial data $f_1$ and $f_2$ such that 
$\mu^{-1}f_1\in H^1$ and $\mu^{-1}f_2\in L^2$. 
\end{rem}

The first part of this theorem shows that exponential decay holds, regardless of dimension, when the low frequencies are suitably cut off by a function $\psi_{\delta,t}$ depending on $t$, and to the best of our knowledge this is the first result of this type. In general, one cannot expect decay faster than exponential, since a resonance contributes an exponentially decaying mode to the resonance expansion of a solution to the wave equation; see \cite[Sections 2.3 and 3.2.2]{kn:DZ}. A related result is due to Christiansen, Datchev, and Yang \cite[Theorem 2.1]{kn:CDY}, who obtain wave expansions with an exponentially decaying remainder under the assumptions of a low frequency resolvent expansion together with suitable resolvent bounds away from zero.

The cut-off functions $\psi_{\delta,t}$ and
$\psi_{\delta,t}^\sharp$ cannot be chosen
independent of $t$ if one wants to keep the same exponential decay in the right-hand side. This is due to the well-known fact that
there is no non-trivial analytic function $\psi(\sigma)$ vanishing for $\sigma\le\delta$.
However, there exists a function $\psi_s$ belonging to the Gevrey class $G^s$, $0<s<1$ being arbitrary, such that
$\psi_s(\sigma)=0$ for $\sigma\le\delta$, $\psi_s(\sigma)=1$ for $\sigma\ge 2\delta$. 
Therefore, one can see from the proof of Theorem \ref{1.1} in Section 6 that the estimates (\ref{eq:1.7}) and (\ref{eq:1.8}) hold with $\psi_{\delta,t}$ replaced by the cut-off function $\psi_s\in G^s$, depending on $\delta$ and independent of $t$, but with the weaker decay
$e^{-c_1t^s}$ in the right-hand sides.

This paper approaches the proof of Theorem \ref{1.1} through resolvent estimates derived from Carleman estimates and perturbation arguments. This strategy has precedents in related areas, including low frequency resolvent estimates and expansions for blackbox, short range, and nontrapping perturbations \cite{kn:V1, kn:B2, kn:V1a, kn:ZA, kn:BB, kn:CDY}, high frequency resolvent bounds for magnetic Schr\"odinger operators \cite{kn:V3, kn:MPS, kn:V5}, and Strichartz and smoothing estimates for magnetic Schr\"odinger operators \cite{kn:EGS, kn:CDL}. Note that the key resolvent bound (\ref{eq:3.2}) at high frequency was
first proved in \cite{kn:V3} for $d\ge 3$ and later in \cite{kn:MPS} for $d \ge 2$. Classical positive energy limiting absorption results in dimensions $d\geq 2$ include the work of Ikebe and Sait\={o} \cite{kn:IS} and H\"ormander \cite[Chapter XXX]{kn:H4}. These results assume continuity of the magnetic potential. In particular, Ikebe and Sait\={o} obtained weighted $L^2$ resolvent bounds locally uniformly away from zero and proved absolute continuity of the positive spectrum, assuming $C^1$ regularity of the magnetic vector potential, short range decay of the magnetic field, and suitable long and short range conditions on the electric potential.
 
Absence of positive eigenvalues is known under weaker regularity assumptions. Koch and Tataru \cite[Section 6]{kn:KT} treat rough short range electric and first order potentials under integrability, decay, and low frequency assumptions, permitting local singularities. Their first-order result applies for $d\geq 3$, with a modified hypothesis in dimension two, and imposes conditions directly on the vector potential. Avramska-Lukarska, Hundertmark, and Kova\v{r}'ik \cite{kn:AHK} obtain gauge invariant criteria in dimensions $d\geq 2$, formulated in terms of the magnetic field and quadratic forms, and allowing local singularities. 

The novelty of our method lies in its suitability for handling $L^\infty$ coefficients and it uniform treatment of the medium and high frequency regimes, which is needed for obtaining the decay described in Theorem \ref{1.1}. In Section \ref{Carleman est section}, we establish Carleman estimates for the free Laplacian on $\R^d$ that, for any fixed $\delta > 0$, hold uniformly for frequencies $\lambda \ge \delta$. (Propositions \ref{2.2} and \ref{2.4}). The Carleman estimates facilitate the derivation of limiting absorption resolvent bounds at medium and high frequencies for cases a) and b), which are Theorems \ref{3.1} and \ref{4.1}, respectively. For these resolvent bounds, it is enough to suppose short range conditions weaker than exponential decay. For case b), we employ a resolvent remainder argument, see \eqref{eq:4.7} and \eqref{eq:4.8}, to transfer the high frequency nontrapping bound \eqref{eq:1.4} to the perturbed resolvent. The smallness of the remainder at high frequency, captured by \eqref{eq:4.8}, does not apply in the case of a first order perturbation. That is why we assume $b$ vanishes when $\mathcal{O} \neq \emptyset$.

Subsequently, in Sections \ref{resolv bds mag Schro section} and \ref{resolv bds nontrap obs section}, under the assumption \eqref{eq:1.3}, we utilize resolvent identities to extend these limiting absorption bounds to the meromorphic continuation of the weighted resolvent $\mu(P -\lambda^2)^{-1}\mu$ (Theorems \ref{3.3} and \ref{4.2}). These identities allow us to leverage the meromorphic continuation of the free resolvent (see Appendix \ref{free resolv appendix} for a review of its properties). This is the step where the exponential decay of the coefficients plays a key role. Finally, Section \ref{time decay section} demonstrates how these resolvent estimates lead to the exponential decay rates presented in \eqref{eq:1.7} and \eqref{eq:1.8}. Furthermore, we show $\psi_{\delta, t}$ can be replaced by $1$ when the dimension is odd and condition \eqref{eq:1.6} is met. The arguments in this section draw inspiration from \cite[Section 3]{kn:V2}, which established polynomial-in-time decay for wave equation solutions on unbounded Riemannian manifolds with general smooth, nontrapping metrics. However, modifications to these arguments are introduced in our setting to exploit the meromorphic continuation of the resolvent, enabling us to obtain exponential decay.

\textbf{Future directions:} We do not treat the case $\mathcal{O}=\emptyset$ and $d=2$ in the presence of a magnetic potential. The obstruction comes from the Carleman estimate in Proposition \ref{2.2}, which in two dimensions has a loss near the origin. This loss is related to the effective potential arising after separation of variables in polar coordinates, which has a negative singularity at the origin when $d =2$. In case b), we avoid this by using the local Carleman estimate of \cite{kn:V4}; however, that estimate applies to electric potentials only and for Dirichlet boundary condition. As mentioned above, when $d=2$ and $\mathcal{O}=\emptyset$, the necessary Carleman estimate is known to hold for sufficiently large $\lambda$ \cite[Theorem 1.1]{kn:MPS}.

 We anticipate that in even dimensions, under condition \eqref{eq:1.6}, it is also possible replace $\psi_{\delta, t}$ 
 by $1$ in \eqref{eq:1.7} and \eqref{eq:1.8}, provided the right-hand sides are modified to $Ct^{-d}$. To achieve this it seems necessary to demonstrate control on the derivatives of the resolvent all the way down to zero frequency, as we have in odd dimensions (see Theorems \ref{3.3} and \ref{4.2}).

It would also be interesting to investigate time decay for short-range $L^\infty$ potentials--those not necessarily exhibiting exponential decay. In such scenarios, the weighted resolvent is unlikely to possess a meromorphic continuation. Consequently, control over the derivatives of the resolvent would require different arguments.

\medskip
\noindent{\textsc{Acknowledgements:}} We thank Kiril Datchev for helpful discussions. J. S. and A. L-H. gratefully acknowledge support from NSF DMS-2204322.\\

\section{Carleman estimates for the Euclidean Laplacian} \label{Carleman est section}

Let $r=|x|$ be the radial variable and define a Lipschitz function $\omega$ by
\begin{equation*}
\omega(r)=
\left\{
\begin{array}{lll}
 (r+1)^{2\ell}&\mbox{for}& 0\le r\le A,\\
 (A+1)^{2\ell}\left(1+(A+1)^{-2s+1}-(r+1)^{-2s+1}\right)&\mbox{for}& r\ge A,
\end{array}
\right.
\end{equation*}
with parameters $A\gg 1$ and satisfying
\begin{equation*}
0<s-\frac{1}{2}<\ell<\frac{2s}{3}<\frac{2}{3}.
\end{equation*}
This condition can be satisfied whenever $1/2 < s < 1$ since $s - 1/2 < 2s/3$ if an only if $s < 3/2$. The first derivative of $\omega$ is given by
\begin{equation*}
\omega'(r)=
\left\{
\begin{array}{lll}
 2\ell(r+1)^{2\ell-1}&\mbox{for}& 0\le r<A,\\
 (2s-1)(A+1)^{2\ell}(r+1)^{-2s}&\mbox{for}& r>A.
\end{array}
\right.
\end{equation*}
We also define a function 
$\varphi\in C^1([0,+\infty))$ such that $\varphi(0)=0$ and its first derivative is a Lipschitz
function given by
\begin{equation*}
\varphi'(r)=
\left\{
\begin{array}{lll}
 (r+1)^{-\ell}(2-(r+1)^{-\kappa})&\mbox{for}& 0\le r\le A,\\
 K_A(r+1)^{-2s}&\mbox{for}& r\ge A,
\end{array}
\right.
\end{equation*}
where 
\begin{equation*}
0<\kappa<2s-1,\quad \kappa<1-\ell,
\end{equation*}
 and 
\begin{equation*}
K_A=(A+1)^{2s-\ell}(2-(A+1)^{-\kappa}) \sim A^{2s-\ell}.
\end{equation*}
The main properties of the functions $\omega$ and $\varphi$ are given in the next lemma.

\begin{lemma} \label{2.1}
For $0<r<A$ we have  
\begin{equation}\label{eq:2.1}
2r^{-1}\omega(r)-\omega'(r)\ge2(1-\ell)(r+1)^{2\ell-1},
\end{equation}
\begin{equation}\label{eq:2.2}
\left(\omega(\varphi')^2\right)'(r)\ge 2\kappa(r+1)^{-1-\kappa}.
\end{equation}
For all $r>A$ we have 
\begin{equation}\label{eq:2.3}
2r^{-1}\omega(r)-\omega'(r)\ge CA^{2\ell}(r+1)^{-1},
\end{equation}
\begin{equation}\label{eq:2.4}
\left(\omega(\varphi')^2\right)'(r)\ge -CA^{-1-2\ell+2s}\omega'(r), 
\end{equation}
with some constant $C>0$. 
\end{lemma}

\begin{proof}
For $0< r < A$,
\begin{equation*}
2r^{-1}\omega(r) - \omega'(r) = 2r^{-1}(r + 1)^{2 \ell} - 2\ell (r + 1)^{2\ell -1} \ge (2- 2\ell)(r + 1)^{2\ell -1} > 0, 
\end{equation*}
\begin{equation*}
(\omega(\varphi')^2)' = ((2 - (r + 1)^{-\kappa})^2)' = 2\kappa (r + 1)^{-\kappa - 1}(2 - (r + 1)^{-\kappa})\ge 2\kappa (r + 1)^{-\kappa - 1}. 
\end{equation*}
On the other hand, for $r > A$, since $s < 1$,
\begin{equation*}
\begin{split}
2 r^{-1} \omega(r) - \omega'(r) &\ge (A+1)^{2\ell}(r^{-1} - (2s-1)(r + 1)^{-2s})\\
& \ge(A+1)^{2\ell}(2-2s)(r + 1)^{-1} \ge 0,
\end{split}
\end{equation*}
which is \eqref{eq:2.3}. Finally,
\begin{equation*}
\begin{split}
(\omega (\varphi')^2)' &= \omega'(\varphi')^2 + 2\omega \varphi'\varphi'' \\
& \ge -2 \frac{\omega}{\omega'} \varphi' |\varphi''| \omega' \\
& = -\frac{4s}{2s - 1} (1 + (A+1)^{-2s + 1} - (r + 1)^{-2s + 1})K_A^{2} (r + 1)^{-2s - 1}\omega' \\ 
&\ge -CK_A^{2}A^{-2s-1}\omega'(r)\\ 
&\ge -CA^{-1-2\ell + 2s} \omega'(r),
\end{split}
\end{equation*}
with some constant $C>0$, confirming \eqref{eq:2.4}.
\end{proof}

Set 
\begin{equation*}
P_{0,\varphi}(\tau)=-e^{\tau\varphi}\Delta e^{-\tau\varphi},
\end{equation*}
where $\tau\gg 1$ is a large parameter. Given a parameter $0<h\le 1$ we will denote by $H^1_h(\R^d)$ the
Sobolev space $H^1(\mathbb{R}^d)$ equipped with the norm $\|\cdot\|_{H^1_h}$ defined by
\begin{equation*}
\|f\|_{H^1_h}^2 \defeq \|f\|_{L^2}^2+h^{2}\|\nabla f\|_{L^2}^2.
\end{equation*}
Furthermore, $H^{-1}_h$ will denote the dual space of $H^1_h$ with respect to the scalar product 
$\langle\cdot,\cdot\rangle_{L^2}$ with the norm
\begin{equation*}
\|f\|_{H^{-1}_h} \defeq \sup_{0\neq g\in H^1_h}\frac{|\langle f,g\rangle_{L^2}|}{\|g\|_{H^1_h}}.
\end{equation*}
Let the function $\psi\in C_0^\infty(\mathbb{R}^d)$ be such that $\psi(x)=1$ for $|x|\le 1$. 
We first prove the following 

\begin{prop} \label{2.2} 
Let $d\ge 2$. 
Given any $\delta>0$, there are positive constants $C$, $A_0$ and $\tau_0$ such that if $A=A_0\tau^{2/(1+2\ell-2s)}$, for all $\tau\ge \tau_0$, $\lambda\ge\delta$, 
$0<\varepsilon\le 1$, and for all functions
$f\in H^2(\mathbb{R}^d)$ satisfying 
\begin{equation*}
\langle x\rangle^{s}(P_{0,\varphi}(\tau)-\lambda^2\pm i\varepsilon)(1-\psi)f\in L^2(\mathbb{R}^d),
\end{equation*}
 we have the estimate 
 \begin{equation}\label{eq:2.5}
\|\langle x\rangle^{-s}(1-\psi)f\|_{H^1_h}\le Ch\tau^{-1/2}\|\langle x\rangle^{s}
(P_{0,\varphi}(\tau)-\lambda^2\pm i\varepsilon)(1-\psi)f\|_{L^2}+CA^\ell(\varepsilon h)^{1/2}\|f\|_{L^2},
\end{equation}
where $h=(\lambda+\tau)^{-1}$. If $d\ge 3$, for all functions
$f\in H^2(\mathbb{R}^d)$ satisfying 
\begin{equation*}
\langle x\rangle^{s}(P_{0,\varphi}(\tau)-\lambda^2\pm i\varepsilon)f\in L^2(\mathbb{R}^d),
\end{equation*}
 we have the estimate 
 \begin{equation}\label{eq:2.6}
\|\langle x\rangle^{-s}f\|_{H^1_h}\le Ch\tau^{-1/2}\|\langle x\rangle^{s}
(P_{0,\varphi}(\tau)-\lambda^2\pm i\varepsilon)f\|_{L^2}+CA^\ell(\varepsilon h)^{1/2}\|f\|_{L^2}.
\end{equation} 
\end{prop}

{\it Proof.} We will write $P_{0,\varphi}(\tau)$ in the polar coordinates 
$(r,w)\in\mathbb{R}^+\times\mathbb{S}^{d-1}$, $r=|x|$, $w=x/|x|$. Recall that  
$L^2(\mathbb{R}^d)=L^2(\mathbb{R}^+\times\mathbb{S}^{d-1}, r^{d-1}drdw)$. 
In what follows we denote by $\|\cdot\|_0$ and $\langle\cdot,\cdot\rangle_0$
the norm and the scalar product in $L^2(\mathbb{S}^{d-1})$. We take complex conjugation to occur in the first argument of $\langle\cdot,\cdot\rangle_0$. We make use of the identity
\begin{equation}\label{eq:2.7}
 r^{(d-1)/2}\Delta  r^{-(d-1)/2}=\partial_r^2+r^{-2}\Delta_w-(d-1)(d-3)(2r)^{-2},
\end{equation}
where $\Delta_w$ denotes the negative Laplace-Beltrami operator
on $\mathbb{S}^{d-1}$.  
Using (\ref{eq:2.7}), we write the operator 
\begin{equation*}
\mathcal{P}_\varphi(\tau)=r^{(d-1)/2}(P_{0,\varphi}(\tau)-\lambda^2)r^{-(d-1)/2}
\end{equation*}
in the form
\begin{equation*}
{\cal P}_\varphi(\tau)=D_r^2+r^{-2}(\Lambda+c_d)-\lambda^2-2i\tau\varphi'D_r+V_1+V_2,
\end{equation*}
where $D_r=i\partial_r$, $\Lambda=-\Delta_w$, and 
\begin{align*}
V_1&=-\tau^{2}(\varphi')^2, &&V_2=(d-1)(d-3)(2r)^{-2}+\tau\varphi'', &&c_d=0, &\text{ if }  d=2, \\
V_1&=-\tau^{2}(\varphi')^2, &&V_2=\tau\varphi'', &&c_d=(d-1)(d-3)/4, &\text{ if } d\ge 3.
\end{align*}
Set $u(r,w)=r^{(d-1)/2}(1-\psi(rw))f(rw)$ and, for $r>0$, $r\neq A$, 
\begin{equation*}
E(r)=-\left\langle (r^{-2}(\Lambda+c_d) -\lambda^2+V_1)u(r,\cdot),u(r,\cdot)\right\rangle_0+\|D_ru(r,\cdot)\|_0^2.
\end{equation*}
For the first derivative of $E$, we get in the sense of distributions on $(0, \infty)$,
\begin{equation*}
\begin{split}
E'(r) &=\frac{2}{r}\left\langle r^{-2}(\Lambda+c_d)u,u\right\rangle_0-\left\langle V_1'u,u\right\rangle_0 +4\tau\varphi'\|D_ru\|_0^2\\
&-2{\rm Im}\,\left\langle\mathcal{P}_\varphi(\tau)u,D_ru\right\rangle_0 + 2{\rm Im}\,\left\langle V_2u,D_ru\right\rangle_0.
\end{split}
\end{equation*}
If $\omega$ is as above, then
\begin{equation*}
\begin{split}
(\omega E)' &=\omega'E+\omega E' \\
&= (2r^{-1}\omega-\omega')\left\langle r^{-2}(\Lambda+c_d)u,u\right\rangle_0+\left\langle(\lambda^2\omega'-(\omega V_1)')u,u\right\rangle_0 \\
&+(\omega'+4\tau\varphi'\omega)\|D_ru\|_0^2 +2\omega{\rm Im}\,\left\langle V_2u,D_ru\right\rangle_0 \\
&-2\omega{\rm Im}\,\left\langle(\mathcal{P}_\varphi(\tau)\pm i\varepsilon)u,D_ru\right\rangle_0 \mp 2\varepsilon\omega{\rm Im}\,\left\langle u,D_ru\right\rangle_0.
\end{split}
\end{equation*}
For $x \in {\rm supp}\, u$ and $r = |x|$, we have
\begin{equation}\label{eq:2.8}
|V_2(r)|\lesssim 
\begin{cases}
 \tau(r+1)^{-1-\ell} &\mbox{for}\quad   r<A,\\
 \tau A^{2s-\ell}(r+1)^{-1-2s}+(r+1)^{-2} &\mbox{for}\quad r>A.
\end{cases}
\end{equation}
In what follows $C>0$ will be a constant which may depend on $\delta$ but is independent of $h$ and $\lambda$. 
Its precise value may change from line to line. 
We have the lower bound 
\begin{equation*}
\begin{split}
(\omega E)'(r) &\ge (2r^{-1}\omega-\omega')\left\langle r^{-2}(\Lambda + c_d) u,u\right\rangle_0+(\lambda^2\omega'-(\omega V_1)')\left\|u\right\|_0^2 \\
&+\left( \frac{\omega'}{2}+ \frac{7\tau\omega\varphi'}{2} \right)\|D_ru\|_0^2-C\omega^2|V_2|^2(\omega'+\tau\omega\varphi')^{-1}\left\|u\right\|_0^2 \\
&-4\omega^2(\omega'+\tau\omega\varphi')^{-1}\left\|(\mathcal{P}_\varphi(\tau)\pm i\varepsilon)u\right\|_0^2 -2\varepsilon\omega\left\|u\|_0\|D_ru\right\|_0 \\
&\ge(2r^{-1}\omega-\omega')\left\langle r^{-2}(\Lambda + c_d) u,u\right\rangle_0+n(r)\left\|u\right\|_0^2+\tau\omega\varphi'\|D_ru\|_0^2 \\
&-C\tau^{-1}\omega(\varphi')^{-1}\left\|(\mathcal{P}_\varphi(\tau)\pm i\varepsilon)u\right\|_0^2 -2\varepsilon\omega\left\|u\|_0\|D_ru\right\|_0,
\end{split}
\end{equation*}
where
\begin{equation*}
\begin{split}
n(r) &=\lambda^2\omega'-(\omega V_1)'-C\omega^2|V_2|^2(\tau\omega\varphi')^{-1} \\
&= \lambda^2\omega'+\tau^{2}(\omega(\varphi')^2)'-C\tau^{-1}|V_2|^2\omega(\varphi')^{-1}.
\end{split}
\end{equation*}
When $r < A$, $\omega (\varphi')^{-1}\lesssim (r + 1)^{3\ell}$. Thus in view of \eqref{eq:2.2} and \eqref{eq:2.8}, 
since $0 < \kappa < 2s-1$ and $\kappa < 1 - \ell$, we have for $r<A$, 
\begin{equation*}
\begin{split}
 n(r) &\ge \lambda^2\omega'+2\kappa\tau^{2}(r+1)^{-1-\kappa}-C\tau(r+1)^{-2+\ell} \\
 &= \lambda^2\omega'+\kappa\tau (r+1)^{-1-\kappa} ( 2\tau - C\kappa^{-1}(r+1)^{\kappa - (1- \ell)}) \\
 &\ge\lambda^2\omega'+\kappa\tau^{2}(r+1)^{-1-\kappa}\\
 &\ge\lambda^2\omega'+\kappa \tau^{2}(r+1)^{-2s},
 \end{split}
 \end{equation*}
provided $\tau$ is large enough.
To bound $n(r)$ from below for $r>A$ observe that, in view of (\ref{eq:2.8}),
in this case we have
\begin{equation*}
\begin{split}
 \frac{|V_2(r)|^2\omega(r)}{\omega'(r)\varphi'(r)}& \lesssim A^{\ell-2s}(r+1)^{4s}|V_2(r)|^2 \\
 &\lesssim \tau^2A^{2s-\ell}(r+1)^{-2}+ A^{\ell-2s}(r+1)^{4s-4} \\
 &\lesssim \tau^2A^{2s-\ell -2}+ A^{2s + \ell -4}\\
 &\lesssim \tau^2A^{2s -2 \ell-1}.
 \end{split}
 \end{equation*}
To get the last inequality we used $2s - \ell - 2 , \, 2s +\ell -4 < 2s - 1 - 2\ell$. 
From this and (\ref{eq:2.4}), for $r>A$, we get
\begin{equation*}
\begin{split}
 n(r) &\ge \omega'\left(\lambda^2-C\tau^{2}A^{2s-2\ell -1}\right) \\
 &=\omega'\left(\lambda^2-CA_0^{2s-2\ell-1}\right)\\
 &\ge 2\lambda^2\omega'/3, 
  \end{split}
 \end{equation*}
provided $A_0$ is taken large enough (recall that $s - 1/2 < \ell$). Combining this with $\lambda \ge \delta$, $4\ell/(1+2\ell-2s)\ge2$ (since $s > 1/2$), and
\begin{equation*}
\begin{split}
 \omega'(r) &= (2s-1)(A+1)^{2\ell}(r+1)^{-2s} \\
& \ge (2s -1) A_0 \tau^{4\ell/(1 + 2\ell -2s)} (r + 1)^{-2s} \\
 &\ge (2s -1) A_0 \tau^{2} (r + 1)^{-2s}
 \end{split}
\end{equation*}
for $r > A$, we get, taking $A_0$ larger if necessary,
 \begin{equation*}
 n(r) \ge \lambda^2\omega'/2+\tau^2(r+1)^{-2s}, \qquad r > A.
 \end{equation*}
 From the above inequalities,
 \begin{equation*}
\begin{split}
 (\omega E)'(r)&\ge(2r^{-1}\omega-\omega')\left\langle r^{-2}(\Lambda +c_d) u,u\right\rangle_0+2^{-1}(\omega'\lambda^2+\kappa \tau^2(r+1)^{-2s})\left\|u\right\|_0^2\\
&+\tau\omega\varphi'\|D_ru\|_0^2-C\tau^{-1}\omega(\varphi')^{-1}\left\|(\mathcal{P}_\varphi(\tau)\pm i\varepsilon)u\right\|_0^2
-2\varepsilon\omega\left\|u\|_0\|D_ru\right\|_0.
\end{split}
 \end{equation*}
Integrating this inequality and using that
$$\int_0^\infty (\omega E)'(r)dr=0,$$
we obtain 
\begin{equation}\label{eq:2.9}
\begin{split}
&\int_0^\infty(2r^{-1}\omega-\omega')\left\langle r^{-2}(\Lambda u + c_d),u\right\rangle_0dr+\int_0^\infty(\omega'\lambda^2
+\tau^2(r+1)^{-2s}) \left\|u\right\|_0^2dr\\
&+\tau\int_0^\infty \omega\varphi'\|D_ru\|_0^2dr\\
&\lesssim \tau^{-1}\int_0^\infty \omega(\varphi')^{-1}\left\|(\mathcal{P}_\varphi(\tau)\pm i\varepsilon)u\right\|_0^2dr
+\varepsilon\int_0^\infty\omega\|u\|_0\|D_ru\|_0dr.
\end{split}
\end{equation}
Observe now that (since $\ell < 2s/3$),
\begin{equation*}
\omega(r)\varphi'(r)^{-1}\lesssim (r+1)^{2s},\quad \omega(r)\lesssim A^{2\ell}.
\end{equation*}
In view of Lemma \ref{2.1} we also have
\begin{equation*}
\omega'(r)\gtrsim (r+1)^{-2s},\quad \omega(r)\varphi'(r)\gtrsim (r+1)^{-2s},\quad 2r^{-1}\omega(r)-\omega'(r)\gtrsim (r+1)^{-2s}.
\end{equation*}
Therefore (\ref{eq:2.9}) implies
\begin{equation}\label{eq:2.10}
\begin{split}
&\int_0^\infty(r+1)^{-2s} \langle r^{-2}(\Lambda + c_d)u, u \rangle_0dr+(\lambda^2+\tau^2)\int_0^\infty(r+1)^{-2s}\|u\|_0^2dr\\
&+\tau\int_0^\infty(r+1)^{-2s}\|D_ru\|_0^2dr\\
&\lesssim \tau^{-1}\int_0^\infty(r+1)^{2s}\left\|(\mathcal{P}_\varphi(\tau)\pm i\varepsilon)u\right\|_0^2dr+A^{2\ell}\varepsilon\int_0^\infty\left(\gamma\|D_ru\|_0^2+\gamma^{-1}\|u\|_0^2\right)dr,
\end{split}
\end{equation}
for every $\gamma>0$. 
On the other hand, in view of the identity
\begin{equation*}
{\rm Re}\,\int_0^\infty\langle 2i\varphi' D_ru,u\rangle_0 dr=\int_0^\infty \varphi''\|u\|_0^2dr,
\end{equation*}
we obtain
\begin{equation*}
\begin{split}
{\rm Re}\,\int_0^\infty\langle(\mathcal{P}_\varphi(\tau)\pm i\varepsilon)u,u\rangle_0 dr 
&=\int_0^\infty\|D_ru\|_0^2dr
+\int_0^\infty \langle r^{-2}(\Lambda+c_d)u,u\rangle_0 dr\\
&-\int_0^\infty(\lambda^2+\tau^{2}(\varphi')^2+\widetilde c_dr^{-2})\|u\|_0^2dr\\
&\ge \int_0^\infty\|D_ru\|_0^2dr-O(\lambda^2+\tau^{2})\int_0^\infty\|u\|_0^2dr,
\end{split}
 \end{equation*}
where $\widetilde c_d=1/4$ if $d=2$ and $\widetilde c_d=0$ if $d\ge 3$.
This implies
\begin{equation}\label{eq:2.11}
\int_0^\infty\|D_ru\|_0^2dr\lesssim (\lambda^2+\tau^{2})\int_0^\infty\|u\|_0^2dr
+\tau^{-2}\int_0^\infty\|(\mathcal{P}_\varphi(\tau)\pm i\varepsilon)u\|_0^2dr.
\end{equation}
By (\ref{eq:2.10}) with $\gamma=(\lambda+\tau)^{-1}$ and (\ref{eq:2.11}), 
\begin{equation}\label{eq:2.12}
\begin{split}
&\int_0^\infty(r+1)^{-2s}\langle r^{-2}(\Lambda + c_d)u, u \rangle_0dr+(\lambda+\tau)^2\int_0^\infty(r+1)^{-2s}\left\|u\right\|_0^2dr\\
&+\int_0^\infty(r+1)^{-2s}\|D_ru\|_0^2dr\\
&\lesssim\tau^{-1}\int_0^\infty(r+1)^{2s}\left\|(\mathcal{P}_\varphi(\tau)\pm i\varepsilon)u\right\|_0^2dr
+A^{2\ell}\varepsilon(\lambda+\tau)\int_0^\infty\|u\|_0^2dr.
\end{split}
\end{equation}
We will now show that (\ref{eq:2.12}) implies (\ref{eq:2.5}). Since 
\begin{equation*}
r^{-2}( \Lambda + c_d) = -\Delta + r^{-(d-1)/2} \partial^2_r r^{(d-1)/2} + \tilde{c}_d r^{-2},
\end{equation*} 
for any $0<\epsilon\ll  1$ independent of $\tau$ and $\lambda$,
\begin{equation*}
\begin{split}
\epsilon^2 \int_0^\infty &(r+1)^{-2s}  \langle r^{-2}(\Lambda + c_d) u, u \rangle_0 dr \\
&= \epsilon^2 \int_0^\infty(r+1)^{-2s} r^{d-1}  \langle r^{-2}(\Lambda + c_d) (1 - \psi) f, (1 - \psi) f \rangle_0 dr \\
&\ge \epsilon^2 \int_0^\infty(r+1)^{-2s} r^{d-1} \langle -\Delta (1 - \psi) f, (1 - \psi) f\rangle_0 dr  +   \epsilon^2 \int_0^\infty(r+1)^{-2s} \langle  \partial^2_r u, u \rangle_0 dr \\
& \gtrsim \epsilon^2 \langle  -\Delta(1 -\psi) f, (1- \psi) \langle x \rangle^{-2s} f \rangle_{L^2}  +   \epsilon^2 \int_0^\infty(r+1)^{-2s}  \langle  \partial^2_r u, u \rangle_0 dr \\
& \gtrsim \epsilon^2  \| \langle x \rangle^{-s}\nabla(1 -\psi) f \|^2_{L^2}  - O(\epsilon^2) \| \langle x \rangle^{-s} (1 - \psi) f \|^2_{L^2} +   \epsilon^2 \int_0^\infty(r+1)^{-2s}  \langle  \partial^2_r u, u \rangle_0 dr.  
\end{split}
\end{equation*}
Now integrate by parts
\begin{equation} \label{eq:2.13}
\begin{split}
\int_0^\infty(r+1)^{-2s}  \langle  \partial^2_r u, u \rangle_0 dr = &-\int_0^\infty(r+1)^{-2s}  \|D_r u \|^2_0 dr \\
 &+\int_0^\infty 2s (r+1)^{-2s-1}   \| u \|^2_0 dr.
\end{split}
\end{equation}
Therefore (\ref{eq:2.12}) implies
\begin{equation}\label{eq:2.14}
\begin{split}
&(\lambda+\tau)\|\langle x\rangle^{-s}(1-\psi)f\|_{L^2}
+\epsilon\|\langle x\rangle^{-s}\nabla((1-\psi)f)\|_{L^2}\\ 
&\lesssim \epsilon\|\langle x\rangle^{-s}(1-\psi)f\|_{L^2}+\tau^{-1/2}\|\langle x\rangle^{s}
(P_{0,\varphi}(\tau)-\lambda^2\pm i\varepsilon)(1-\psi)f\|_{L^2}\\
&+A^{\ell}\varepsilon^{1/2}(\lambda+\tau)^{1/2}\|f\|_{L^2}.
\end{split}
\end{equation}
If $\epsilon\ll\delta$ we can absorb the first term in the right-hand side of (\ref{eq:2.14}) by the first term in the left-hand side
and obtain (\ref{eq:2.5}). 

Finally, we explain why (\ref{eq:2.6}) holds when $d\ge 3$. That is, when $d \ge 3$, we may take $\psi \equiv 0$ and thus $u(r,w)=r^{(d-1)/2}f(rw)$. In this case, \eqref{eq:2.8} holds for all $x$ with $r = |x|\neq A$. From this, one shows we again have \eqref{eq:2.12}. The subsequent estimates follow as before. In particular, if $d \ge 3$, the Poincar\'e inequality (Lemma \ref{D.1}) ensures convergence of the right side of \eqref{eq:2.13}.

\eproof

In what follows we improve the estimate \eqref{eq:2.6} with $\psi\equiv 0$ when $d\ge 3$, showing that it still holds with the norm $\|\cdot\|_{L^2}$ in the first term in the right-hand side replaced 
by the smaller Sobolev norm $\|\cdot\|_{H_h^{-1}}$, where still $h = (\lambda + \tau)^{-1}$ and $\lambda \ge \delta$, $\tau \ge \tau_0$ are as in the statement of Proposition \ref{2.2}. To this end we again utilize the operator $P_{0, \varphi}(\tau) = -e^{\tau \varphi} \Delta e^{-\tau \varphi}$ and its generalization
\begin{equation*}
P_{0,\varphi_p}(\tau) \defeq \langle x\rangle^{p}P_{0,\varphi}(\tau)\langle x\rangle^{-p} =-\Delta+\mathcal{Q}_p, \quad \mathcal{Q}_p=2\tau\nabla\varphi_p\cdot\nabla-\tau^{2}|\nabla\varphi_p|^2+\tau\Delta\varphi_p, \qquad p \in \R,
\end{equation*}
where
\begin{equation*}
 \varphi_p(r)=\varphi(r)+\frac{p}{2}\tau^{-1}\log(r^2+1).
\end{equation*}
By integration by parts 
\begin{equation*}
\begin{gathered}
\langle \mathcal{Q}_p f, g \rangle_{L^2} = \langle f, \mathcal{Q}^*_p g \rangle_{L^2}, \qquad f,g \in H^1(\R^d),  \\
\mathcal{Q}^*_p \defeq -2\tau\nabla\varphi_p\cdot\nabla-\tau^{2}|\nabla\varphi_p|^2-\tau\Delta\varphi_p.
\end{gathered}
\end{equation*}

It is easy to see that
\begin{equation*}
|\nabla\varphi_p|\lesssim |\varphi'(r)|+\tau^{-1}\lesssim 1,
\end{equation*}
where the constants implicit in the estimate depend on $p$ but are independent of $\tau$. Furthermore, since
\begin{equation*}
\Delta\varphi_p=\varphi_p''(r)+\frac{d-1}{r}\varphi'_p(r),\quad r\neq 0, \, A,
\end{equation*}
we also have
\begin{equation*}
|\Delta\varphi_p|\lesssim 1+r^{-1},\quad r \neq 0,\,A.
\end{equation*}
 Therefore, by Poincar\'e's inequality (see \eqref{eq:D.1} in Appendix \ref{Poincare appendix}), $P_{0, \varphi_p}(\tau)$ maps boundedly $H^2(\R^d) \to L^2(\R^d)$. Additionally,
 \begin{equation*}
\begin{split}
\|h^{2}\mathcal{Q}_pf\|_{L^2}&\lesssim (h\tau+h^{2}\tau^{2})\|f\|_{H_h^1}+h^2\tau\|r^{-1}f\|_{L^2}\\
&\lesssim (h\tau+h^{2}\tau^{2})\|f\|_{H_h^1}+h^2\tau\|\nabla f\|_{L^2}\\
&\lesssim (h\tau+h^{2}\tau^{2})\|f\|_{H_h^1}, \qquad f \in H^1(\R^d),
\end{split}
 \end{equation*}
and similarly for $\mathcal{Q}^*_p$, where we have again used Poincar\'e's inequality. Hence
\begin{equation}\label{eq:2.15}
\|h^{2}\mathcal{Q}_p\|_{H_h^1\to L^2} \lesssim 
 h\tau+h^{2}\tau^{2}\lesssim 1.
\end{equation}

\begin{lemma} \label{2.3} 
Let $p\in\mathbb{R}$ and let $\delta$ and $\tau_0$ be as in the statement of Proposition \ref{2.2}.  Then, there exist $C > 0$ and $\theta_0 > 0$ independent of $\lambda$ and $\tau$, such that for all $\lambda \ge \delta$, $\tau \ge \tau_0$, and $\theta \ge \theta_0$,
\begin{equation}\label{eq:2.16}
\left\|\langle x\rangle^{p}\left(h^{2}P_{0,\varphi}(\tau)\pm i\theta^2\right)^{-1}\langle x\rangle^{-p}
\right\|_{H_h^{-1}\to H_h^1} \le C,
\end{equation}
\begin{equation}\label{eq:2.17}
\left\|\langle x\rangle^{p}\left(h^{2}P_{0,\varphi}(\tau)\pm i\theta^2\right)^{-1}\langle x\rangle^{-p}
\right\|_{H_h^{-1}\to L^2}\le C \theta^{-1},
\end{equation}
\begin{equation}\label{eq:2.18}
\left\|\langle x\rangle^{p}\left(h^{2}P_{0,\varphi}(\tau)\pm i\theta^2\right)^{-1}\langle x\rangle^{-p}
\right\|_{L^2\to H_h^1}\le C \theta^{-1},
\end{equation}
\begin{equation}\label{eq:2.19}
\left\|\langle x\rangle^{p}\left(h^{2}P_{0,\varphi}(\tau)\pm i\theta^2\right)^{-1}\langle x\rangle^{-p}
\right\|_{L^2\to L^2}\le C \theta^{-2},
\end{equation}
where $h=(\lambda+\tau)^{-1}$. 
\end{lemma}

{\it Proof.} Recall that $\|f\|_{H_h^{s}}\sim\|(1-h^{2}\Delta)^{s/2}f\|_{L^2}$, $s=-1,1$.
Using this it is easy to see that the above bounds hold for $p=0$ and $P_{0,\varphi}(\tau)$ replaced by $-\Delta$.

To prove \eqref{eq:2.16} through \eqref{eq:2.19}, begin by using \eqref{eq:2.15} in combination with \eqref{eq:2.16} in the case $p =0$ and $P_{0, \varphi}(\tau)$ replaced by $-\Delta$. We get that for $\theta \gg 1$, 
\begin{equation*}
\| h^2 \mathcal{Q}_p (-h^2 \Delta \pm i \theta^2)^{-1}\|_{L^2 \to L^2} \le \| h^2 \mathcal{Q}_p\|_{H^1_h \to L^2} \| (-h^2 \Delta \pm i \theta^2)^{-1}\|_{L^2 \to H^1_h} \lesssim \theta^{-1} \le 1/2,
\end{equation*} 
whence 
\begin{equation*}
I + h^2 \mathcal{Q}_p(-h^2 \Delta \pm i \theta^2)^{-1}:L^2(\R^d) \to L^2(\R^d)
\end{equation*} 
 is invertible by a Neumann series. 
It is then checked by direct computation that the inverse of  
\begin{equation*}
h^2 P_{0, \varphi_p}( \tau) \pm i \theta^2 = -h^2 \Delta + h^2 \mathcal{Q}_p \pm i\theta^2 : H^2(\R^d) \to L^2(\R^d)
\end{equation*}
is 
\begin{equation} \label{eq:2.20}
(-h^2 \Delta + h^2 \mathcal{Q}_p \pm i\theta^2)^{-1} = (-h^2 \Delta \pm i \theta^2)^{-1} (I + h^2 \mathcal{Q}_p (-h^2 \Delta \pm i \theta^2)^{-1})^{-1}, \qquad \theta \gg 1.
\end{equation}
From this we also conclude the identity
\begin{equation} \label{eq:2.21}
\begin{split}
&\left(-h^{2}\Delta+h^{2}\mathcal{Q}_p\pm i\theta^2\right)^{-1}-
\left(-h^{2}\Delta\pm i\theta^2\right)^{-1}\\
&=-\left(-h^{2}\Delta\pm i\theta^2\right)^{-1}h^2\mathcal{Q}_p
\left(-h^{2}\Delta+h^{2}\mathcal{Q}_p\pm i\theta^2\right)^{-1}.
\end{split}
 \end{equation}
 Using this strategy we can also establish bounded invertibility of 
 \begin{equation*}
 -h^2 \Delta + h^2 \mathcal{Q}^*_p \pm i\theta^2 : H^2(\R^d) \to L^2(\R^d)
 \end{equation*}
  for $\theta$ large.
We further show that for $\theta$ big enough
\begin{equation} \label{eq:2.22}
\langle x\rangle^{p}\left(h^{2}P_{0,\varphi}(\tau)\pm i\theta^2\right)^{-1}\langle x\rangle^{-p}  = (-h^2 \Delta + h^2 \mathcal{Q}_p \pm i \theta^2)^{-1}, \qquad p \in \R,
\end{equation}
where initially the left side is interpreted as an operator sending $C^\infty_0(\R^d)$ to $H^2_{\text{loc}}(\R^d)$. At first, let $p \ge 0$. Suppose $f, \, g \in C^\infty_0(\R^d)$ and choose a sequence $u_k \in C^\infty_0(\R^d)$ converging to $(-h^2 \Delta + h^2 \mathcal{Q}_{p} \pm i \theta^2)^{-1}f$ in the $H^2(\R^d)$-norm. Then
\begin{equation*}
\begin{split}
\int_{\R^d}  &\overline{g} \left(h^{2}P_{0,\varphi}(\tau)\pm i\theta^2\right)^{-1}\langle x\rangle^{-p} f \\
& = \int_{\R^d}  \overline{g} \left(-h^2 \Delta + h^2 \mathcal{Q}_0 \pm i\theta^2\right)^{-1}\langle x\rangle^{-p} f  \\
&= \lim_{k \to \infty} \int_{\R^d} \left[ \left(-h^2 \Delta + h^2 \mathcal{Q}^*_0 \mp i\theta^2\right)^{-1}   \overline{g}\right]\langle x\rangle^{-p} (-h^2 \Delta + h^2 \mathcal{Q}_{p} \pm i \theta^2)u_k \\
&= \lim_{k \to \infty} \int_{\R^d} \left[ \left(-h^2 \Delta + h^2 \mathcal{Q}^*_0 \mp i\theta^2\right)^{-1}   \overline{g}  \right] (-h^2 \Delta + h^2 \mathcal{Q}_0 \pm i \theta^2) \langle x\rangle^{-p}  u_k \\
& = \int_{\R^d}  \overline{g}  \langle x\rangle^{-p}   (-h^2 \Delta + h^2 \mathcal{Q}_p \pm i \theta^2)^{-1}f,
\end{split}
\end{equation*}
This confirms \eqref{eq:2.22} for $p \ge 0$. To see \eqref{eq:2.22} for $p < 0$, we give a calculation analogous to the previous one but replace $P_{0,\varphi}(\tau)$ with its adjoint $-h^2 \Delta + h^2 \mathcal{Q}^*_0$. This yields
\begin{equation*}
\langle x\rangle^{p}\left(-h^2 \Delta + h^2 \mathcal{Q}^*_0 \mp i\theta^2\right)^{-1}\langle x\rangle^{-p}  = (-h^2 \Delta + h^2 \mathcal{Q}^*_p \mp i \theta^2)^{-1}
\end{equation*}
 for $p \ge 0$, which implies \eqref{eq:2.22} for $p < 0$ by duality.

Now we are in a position to show \eqref{eq:2.16}. By \eqref{eq:2.15}, \eqref{eq:2.21}, and since \eqref{eq:2.22} holds for any $p \in \R$:
\begin{equation*}
\begin{split}
& \left \| \langle x\rangle^{p}\left(h^{2}P_{0,\varphi}(\tau)\pm i\theta^2\right)^{-1}\langle x\rangle^{-p}\right\|_{H_h^{-1}\to H_h^1} = \left\|\left(-h^2 \Delta + h^2 \mathcal{Q}_p \pm i \theta^2 \right)^{-1}
\right\|_{H_h^{-1}\to H_h^1}\\
&\le \left\|\left(-h^{2}
\Delta\pm i\theta^2\right)^{-1}\right\|_{H_h^{-1}\to H_h^1}\\
& +\left\|\left(-h^{2}\Delta\pm i\theta^2\right)^{-1}
\right\|_{L^2\to H_h^1}\|h^{2}\mathcal{Q}_p\|_{H_h^1\to L^2}
\left\|\left(-h^2 \Delta + h^2 \mathcal{Q}_p \pm i \theta^2 \right)^{-1}
\right\|_{H_h^{-1}\to H_h^1}\\
&\lesssim 1+\theta^{-1}\left\|\left(-h^2 \Delta + h^2 \mathcal{Q}_p \pm i \theta^2 \right)^{-1}
\right\|_{H_h^{-1}\to H_h^1}
\end{split}
 \end{equation*}
which implies \eqref{eq:2.16} if $\theta$ is taken large enough. Clearly, \eqref{eq:2.17} can be obtained in the same way. On the other hand, we obtain \eqref{eq:2.18} and \eqref{eq:2.19} from \eqref{eq:2.20} and \eqref{eq:2.22}. 
\eproof

We derive from Proposition \ref{2.2} and Lemma \ref{2.3} the following.

\begin{prop} \label{2.4} 
Let $d\ge 3$. 
Given any $\delta>0$, there are positive constants $C$, $A_0$ and $\tau_0$ such that if $A=A_0\tau^{2/(1+2\ell-2s)}$, for all $\tau\ge \tau_0$, $\lambda\ge\delta$, 
$0<\varepsilon\le 1$, and for all functions
$f\in H^1(\mathbb{R}^d)$ satisfying 
\begin{equation*}
\langle x\rangle^{s}(P_{0,\varphi}(\tau)-\lambda^2\pm i\varepsilon)f\in H^{-1}(\mathbb{R}^d),
\end{equation*}
 we have 
 \begin{equation}\label{eq:2.23}
\|\langle x\rangle^{-s}f\|_{H^1_h}\le Ch\tau^{-1/2}\|\langle x\rangle^{s}
(P_{0,\varphi}(\tau)-\lambda^2\pm i\varepsilon)f\|_{H_h^{-1}}+CA^{\ell}(\varepsilon h)^{1/2}\|f\|_{L^2},
\end{equation}
where $h=(\lambda+\tau)^{-1}$.
\end{prop}

{\it Proof.} We use the identity
\begin{equation*}
\begin{split}
f&=h^{2}\left(\mp i(\varepsilon+(\theta/h)^2)+\lambda^2\right)\left(h^{2}P_{0,\varphi}(\tau)\mp i\theta^2\right)^{-1}f\\
&+h^{2}\left(h^{2}P_{0,\varphi}(\tau)\mp i\theta^2\right)^{-1}(P_{0,\varphi}(\tau)-\lambda^2\pm i\varepsilon)f.
\end{split}
 \end{equation*}
 Set
\begin{equation*}
g=\left(h^{2}P_{0,\varphi}(\tau)\mp i\theta^2\right)^{-1}f.
\end{equation*}
By Lemma \ref{2.3}, for $\theta$ large enough,
\begin{equation*}
\begin{split}
\|\langle x\rangle^{-s}f\|_{L^2}\lesssim &
\left\|\langle x\rangle^{-s}g\right\|_{L^2}+h^2\left\|\left(h^{2}P_{0,\varphi}(\tau)\mp i\theta^2\right)^{-1}\right\|_{H^{-1}_h\to L^2}
\left\|(P_{0,\varphi}(\tau)-\lambda^2\pm i\varepsilon)f\right\|_{H^{-1}_h}\\
& \lesssim 
\left\|\langle x\rangle^{-s}g\right\|_{L^2}+h^2
\left\|(P_{0,\varphi}(\tau)-\lambda^2\pm i\varepsilon)f\right\|_{H^{-1}_h}.
\end{split}
 \end{equation*}
Here and later in the proof the implicit constants depend on $\theta$ but are independent of $\lambda$ and $\tau$. We now apply (\ref{eq:2.6}) to the function $g$. Note that $g$ satisfies the required hypothesis of Proposition \ref{2.2} because by Lemma \ref{2.3},
\begin{equation*}
\begin{split}
\langle x \rangle^s (&P_{0, \varphi}(\tau) - \lambda^2 \pm i \varepsilon)g \\
&=\langle x \rangle^s (P_{0, \varphi}(\tau) - \lambda^2 \pm i \varepsilon) \left(h^{2}P_{0,\varphi}(\tau)\mp i\theta^2\right)^{-1}f \\
&=\left(\langle x \rangle^s \left(h^{2}P_{0,\varphi}(\tau)\mp i\theta^2\right)^{-1} \langle x \rangle^{-s} \right) \langle x \rangle^{s} (P_{0, \varphi}(\tau) - \lambda^2 \pm i \varepsilon) f \in L^2(\R^d). 
\end{split}
\end{equation*}
Therefore, combining \eqref{eq:2.6} with Lemma \ref{2.3},
\begin{equation*}
\begin{split}
&\|\langle x\rangle^{-s}g\|_{H^1_h}\lesssim h\tau^{-1/2}\|\langle x\rangle^{s}
(P_{0,\varphi}(\tau)-\lambda^2\pm i\varepsilon)g\|_{L^2}+A^{\ell}(\varepsilon h)^{1/2}\|g\|_{L^2}\\
 &\lesssim h\tau^{-1/2}\left\|\langle x\rangle^{s}\left(h^{2}P_{0,\varphi}(\tau)\mp i\theta^2\right)^{-1}
\langle x\rangle^{-s}\right\|_{H^{-1}_h\to L^2}
\|\langle x\rangle^{s}(P_{0,\varphi}(\tau)-\lambda^2\pm i\varepsilon)f\|_{H_h^{-1}}\\
&+A^{\ell}(\varepsilon h)^{1/2}
\left\|\left(h^{2}P_{0,\varphi}(\tau)\mp i\theta^2\right)^{-1}\right\|_{L^2\to L^2}\|f\|_{L^2}\\
 &\lesssim h\tau^{-1/2}
\|\langle x\rangle^{s}(P_{0,\varphi}(\tau)-\lambda^2\pm i\varepsilon)f\|_{H_h^{-1}}+A^{\ell}(\varepsilon h)^{1/2}\|f\|_{L^2}.
\end{split}
 \end{equation*}
Thus we obtain
\begin{equation*}
\|\langle x\rangle^{-s}f\|_{H^1_h}\lesssim h\left(h+\tau^{-1/2}\right)
\|\langle x\rangle^{s}(P_{0,\varphi}(\tau)-\lambda^2\pm i\varepsilon)f\|_{H_h^{-1}}+A^{\ell}(\varepsilon h)^{1/2}\|f\|_{L^2},
\end{equation*}
which implies \eqref{eq:2.23} since $h<\tau^{-1}$.
\eproof

\section{Resolvent bounds for the magnetic Schr\"odinger operator} \label{resolv bds mag Schro section}

Consider in $\mathbb{R}^d$, $d\ge 3$, the operator
\begin{equation*}
P=(i\nabla+b(x))^2+V(x),
\end{equation*}
where the electric potential $V\in L^\infty(\mathbb{R}^d,\mathbb{R})$
and the magnetic potential 
$b\in L^\infty(\mathbb{R}^d,\mathbb{R}^d)$ satisfy
\begin{equation}\label{eq:3.1}
|V(x)|+|b(x)|\le C\langle x\rangle^{-\rho}, \quad C>0,\,\rho>1.
\end{equation}
In this section we prove weighted resolvent bounds for the self-adjoint realization of the above operator (which
again will be denoted by $P$) on the Hilbert space $L^2(\R^d)$. 
We have

\begin{Theorem} \label{3.1}
Assume the condition \eqref{eq:3.1} fulfilled. Then, given any $s>1/2$ and $\delta>0$ there is a constant $C>0$ such that 
\begin{equation}\label{eq:3.2}
\left\|\langle x\rangle^{-s}\partial_x^\alpha(P-\lambda^2\pm i\varepsilon)^{-1}\partial_x^\beta\langle x\rangle^{-s}\right\|
\le C\lambda^{|\alpha|+|\beta|-1},\quad\lambda\ge\delta,\,0<\varepsilon<1,
\end{equation}
 where 
$\alpha$ and $\beta$ are multi-indices such that $|\alpha|\le 1$ and $|\beta|\le 1$.  
\end{Theorem}

{\it Proof.} We prove \eqref{eq:3.2} using the Carleman estimate \eqref{eq:2.23}.
 We keep the same notations as in the previous section.  
Clearly, it suffices to prove \eqref{eq:3.2} for $0<s-\frac{1}{2}\ll 1$, since this would imply the estimate
for all $s>\frac{1}{2}$. 
In Appendix \ref{self-adjointness mag Schro appendix} we show that, in the sense of distributions on $\R^d$, 
the operator $P$ acts on $u$ in the domain $D(P) \subseteq H^1(\R^d)$ by
\begin{equation*}
Pu =-\Delta u+i\nabla\cdot (bu)+ib\cdot\nabla u+\widetilde Vu,
\end{equation*}
where $\widetilde V=V+|b|^2$. Here, $\nabla \cdot (bu)$ is defined distributionally by 
$(\nabla \cdot (bu), v) \defeq -(u, b \cdot \nabla  v)$, where $(\cdot, \cdot)$ denotes distributional pairing. 
We note that  $u \mapsto \nabla \cdot (bu)$ is a bounded mapping from $L^2(\R^d)$ to $H_h^{-1}(\R^d)$. Given $g \in C^\infty_0(\mathbb{R}^d)$, set
\begin{equation*}
f=(P-\lambda^2\pm i\varepsilon)^{-1}g \in D(P) \cap H^1(\R^d),\qquad f_1=e^{\tau\varphi}f \in H^1(\R^d).
\end{equation*}
Both $P$ and $P_{0, \varphi} = -e^{\tau \varphi} \Delta e^{-\tau \varphi}:H_h^1(\R^d) \to H_h^{-1}(\R^d)$ are bounded. As members of $H^{-1}_h(\R^d)$,
\begin{equation*}
\begin{split}
P_{0, \varphi} f_1 - e^{\tau \varphi} Pf &= -e^{\tau \varphi} (i \nabla \cdot (bf) + ib \cdot \nabla f + \tilde{V} f) \\
&= i\nabla\cdot (bf_1) -ib\cdot\nabla f_1 -\widetilde V f_1 +2i\tau\nabla\varphi\cdot b f_1.
\end{split}
\end{equation*}
By the definition of $\varphi$, we have $ \nabla \varphi = O(\langle r \rangle^{-2s})$, and if we take $s > \tfrac{1}{2}$ 
small enough so that $2s < \rho$ with $\rho$ as in \eqref{3.1}, then
\begin{equation}\label{eq:3.3}
\left\|\langle x\rangle^{s}\left(P_{0, \varphi} f_1 - e^{\tau \varphi} Pf\right)
\right\|_{H_h^{-1}}\lesssim h^{-1}\|\langle x\rangle^{-s}f_1\|_{H_h^1}.
\end{equation}
We are going to use the estimate \eqref{eq:2.23} with $f$ replaced by $f_1$. Note that $f$ satisfies the required hypothesis of Proposition \ref{2.4} because
\begin{equation*}
\begin{split}
\langle x \rangle^{s}(P_{0, \varphi} &- \lambda^2 \pm i \varepsilon) e^{\tau \varphi} (P - \lambda^2 \pm i\varepsilon)^{-1} g\\
&= \langle x \rangle^{s} e^{\tau \varphi }( -\Delta - \lambda^2 \pm i \varepsilon) (P - \lambda^2 \pm i\varepsilon)^{-1} g\\
&= \langle x \rangle^{s} e^{\tau \varphi} g + \langle x \rangle^{s} e^{\tau \varphi} ( i \nabla \cdot b + ib \cdot \nabla + \widetilde V ) (P- \lambda^2 \pm i\varepsilon)^{-1}g \in H^{-1}(\R^d).
\end{split}
\end{equation*}
By (\ref{eq:2.23}) and (\ref{eq:3.3}) we get
\begin{equation*}
\begin{split}
\|\langle x\rangle^{-s}f_1\|_{H_h^1}&\lesssim h\tau^{-1/2}\|\langle x\rangle^{s}
e^{\tau \varphi}(P-\lambda^2\pm i\varepsilon)f\|_{H_h^{-1}}\\
&+\tau^{-1/2}\|\langle x\rangle^{-s}f_1\|_{L^2}
+A^{\ell}(\varepsilon h)^{1/2}\|f_1\|_{L^2}.
\end{split}
 \end{equation*}
We can absorb the second term in the right-hand side of the above inequality by taking $\tau$ large enough independent of
$\lambda$. Since $h<\lambda^{-1}$, this leads to
\begin{equation*}
\|\langle x\rangle^{-s}f_1\|_{H_h^1}\lesssim \lambda^{-1}\|\langle x\rangle^{s}
e^{\tau \varphi}(P-\lambda^2\pm i\varepsilon)f\|_{H_h^{-1}}
+\varepsilon^{1/2}\lambda^{-1/2}\|f_1\|_{L^2},
 \end{equation*}
which in turn implies
\begin{equation}\label{eq:3.4}
\|\langle x\rangle^{-s}f\|_{H_h^1}\lesssim \lambda^{-1}\|\langle x\rangle^{s}
(P-\lambda^2\pm i\varepsilon)f\|_{H_h^{-1}}
+\varepsilon^{1/2}\lambda^{-1/2}\|f\|_{L^2},
\end{equation}
where the implicit constant depends on $\tau$, which is now fixed, but is indepedent of $\lambda$. 
On the other hand, the symmetry of the operator $P$ on the Hilbert space $L^2(\mathbb{R}^d)$ gives
\begin{equation}\label{eq:3.5} 
\begin{split}
\varepsilon\|f\|_{L^2}^2&=\left| {\rm Im}\,\left\langle(P-\lambda^2\pm i\varepsilon)f,f\right\rangle_{L^2} \right| \\
&\le\left|\left\langle \langle x\rangle^{s}g,\langle x\rangle^{-s}f\right\rangle_{L^2}\right|\\
&\le \gamma\lambda\|\langle x\rangle^{-s}f\|_{H_h^1}^2+
\gamma^{-1}\lambda^{-1}\|\langle x\rangle^{s}g\|_{H_h^{-1}}^2
\end{split}
\end{equation}
for every $\gamma>0$. Combining \eqref{eq:3.4}, \eqref{eq:3.5} and taking $\gamma$ small enough independent of $\lambda$
 we obtain
\begin{equation}\label{eq:3.6}
\|\langle x\rangle^{-s}f\|_{H_h^1}\lesssim\lambda^{-1}
\|\langle x\rangle^{s}g\|_{H_h^{-1}}.
\end{equation}
It is easy to see that \eqref{eq:3.6} is equivalent to \eqref{eq:3.2}.
\eproof

Denote $\mathbb{C}^-:=\{\lambda\in \mathbb{C}:{\rm Im}\,\lambda<0\}$ and 
$\mathcal{L}=\mathbb{C}$ if $d$ is odd, while
\begin{equation*}
\mathcal{L}=\left\{\lambda\in\mathbb{C}:-\frac{3\pi}{2}<\arg(\lambda)<\frac{\pi}{2}\right\}
\end{equation*}
 if $d$ is even. Also, given a parameter
$\gamma>0$, set $\mathcal{L}_\gamma=\{\lambda\in\mathcal{L}:{\rm Im}\,\lambda<\gamma\}$. 
In Proposition \ref{3.2} below, we combine \eqref{eq:3.2} with estimates for the free resolvent 
(reviewed in Appendix \ref{free resolv appendix}) to  construct an analytic continuation of the operator valued function 
$\mu (P - \lambda^2)^{-1}\mu : L^2(\R^d) \to L^2(\R^d)$ from $\C^{-}$ into $\mathcal{L}_{\gamma}$, for $\gamma$ small enough. 

\begin{prop} \label{3.2}
 Suppose (\ref{eq:1.3}) is fulfilled. There is a constant $\gamma>0$ such that the operator-valued function 
\begin{equation*}
\mu\nabla^{\ell}(P-\lambda^2)^{-1}\mu:L^2(\R^d) \to L^2(\R^d),\quad \ell=0,1,
\end{equation*}
extends analytically from $\mathbb{C}^-$ to $\mathcal{L}_\gamma$ and satisfies the bound 
\begin{equation}\label{eq:3.7}
\left\|\mu\nabla^{\ell}(P-\lambda^2)^{-1}\mu\right\|\le C(|\lambda|+1)^{\ell-1}
\end{equation}
for $\lambda\in\mathcal{L}_\gamma$, $|\lambda|\ge\delta$, $\delta>0$ being arbitrary, 
 with a constant $C>0$ which may depend on $\delta$.
 Moreover, if $d$ is odd and the condition \eqref{eq:1.6} is assumed, the bound (\ref{eq:3.7}) holds for all $\lambda\in\mathcal{L}_\gamma$.
\end{prop}

From \eqref{eq:3.7} and Lemma \ref{B.1}, we obtain the following bounds on the $\lambda$-derivatives of $\mu(P - \lambda^2)^{-1} \mu$, which are key to our proof of the wave decay in Section  \ref{time decay section}.

\begin{Theorem} \label{3.3}
Assume the condition \eqref{eq:1.3} is fulfilled. Then, given any $\delta>0$ and any integer $k\ge 0$, the bound
\begin{equation}\label{eq:3.8}
\left\|\frac{d^k}{d\lambda^k}\left(\mu\nabla^{\ell}(P-\lambda^2)^{-1}\mu\right)\right\|
\le C^{k+1}k!(|\lambda|+1)^{\ell-1}
\end{equation}
holds for all $\lambda\in\R$, $|\lambda|\ge\delta$, with a constant $C=C_\delta>0$, where $\ell\in\{0,1\}$.
If $d$ is odd and the condition \eqref{eq:1.6} is assumed, the bound \eqref{eq:3.8} holds for all $\lambda\in\R$.   
\end{Theorem}

{\it Proof of Proposition \ref{3.2}.} Denote by $P_0$ the self-adjoint realization of $-\Delta$
on $L^2(\mathbb{R}^d)$. 
Let $\lambda\in\mathbb{C}^-$ and denote by $I$ the identity operator. We begin from two resolvent identities,
\begin{align*}
& (P- \lambda^2)^{-1} (\widetilde{V} +  i\nabla\cdot b+ib\cdot\nabla) =  
I - (P- \lambda^2)^{-1} (P_0 - \lambda^2) \quad \text{on}\quad  H^2(\R^d), \\
& (P_0- \lambda^2)^{-1} (\widetilde{V} +  i\nabla\cdot b+ib\cdot\nabla) =  
-I + (P_0- \lambda^2)^{-1} (P - \lambda^2) \quad \text{on}\quad D(P), 
\end{align*}
the first of which we prove in detail in Appendix \ref{self-adjointness mag Schro appendix}. These yield
\begin{equation}\label{eq:3.9}
\begin{split}
&(P-\lambda^2)^{-1}-(P_0-\lambda^2)^{-1}\\
&=-(P_0-\lambda^2)^{-1}(\widetilde V+i\nabla\cdot b+ib\cdot\nabla)(P-\lambda^2)^{-1}\\
&=-(P-\lambda^2)^{-1}(\widetilde V+i\nabla\cdot b+ib\cdot\nabla)(P_0-\lambda^2)^{-1}.
\end{split}
\end{equation}
Let $z\in\mathbb{C}^-$. By \eqref{eq:3.9}, we get
\begin{equation}\label{eq:3.10}
\begin{split}
&(P-\lambda^2)^{-1}-(P-z^2)^{-1}\\
&=(\lambda^2-z^2)(P-z^2)^{-1}(P-\lambda^2)^{-1}\\
&=L^\sharp(z)((P_0-\lambda^2)^{-1}-(P_0-z^2)^{-1})L^\flat(\lambda),
\end{split}
\end{equation}
where
\begin{align*}
& L^\sharp=I-(P-z^2)^{-1}(\widetilde V+i\nabla\cdot b+ib\cdot\nabla),\\
 & L^\flat=I-(\widetilde V+i\nabla\cdot b+ib\cdot\nabla)(P-\lambda^2)^{-1}.
 \end{align*}
Multiplying both sides of (\ref{eq:3.10}) by $\mu$ we get
\begin{equation}\label{eq:3.11}
\begin{split}
&\mu(P-\lambda^2)^{-1}\mu-\mu(P-z^2)^{-1}\mu\\
&=\sum_{\ell_1=0}^1\sum_{\ell_2=0}^1L^\sharp_{\ell_1}(z)\mu^{1-\ell_1}(-i\mu^{-1}b\cdot\nabla)^{\ell_1}((P_0-\lambda^2)^{-1}-(P_0-z^2)^{-1})(-i\nabla\cdot b\mu^{-1})^{\ell_2}\mu^{1-\ell_2}L^\flat_{\ell_2}(\lambda),
\end{split}
\end{equation}
where
\begin{align*}
& L_0^\sharp=I-\mu(P-z^2)^{-1}(\widetilde V+i\nabla\cdot b)\mu^{-1},\\
& L_1^\sharp=\mu(P-z^2)^{-1}\mu,\\
& L_0^\flat=I-\mu^{-1}(\widetilde V+ib\cdot\nabla)(P-\lambda^2)^{-1}\mu,\\
& L_1^\flat=\mu(P-\lambda^2)^{-1}\mu,
\end{align*}
are bounded operators on $L^2(\mathbb{R}^d)$. 
We now let the operator $\mu^{-1}ib\cdot\nabla$ act on the left side of (\ref{eq:3.10}) and multiply the right side by $\mu$.
We get
\begin{equation}\label{eq:3.12}
\mu^{-1}ib\cdot\nabla(P-\lambda^2)^{-1}\mu=T_1(\lambda, z)+T_2(\lambda,z)\mu(P-\lambda^2)^{-1}\mu+T_3(\lambda,z)
\mu^{-1}ib\cdot\nabla(P-\lambda^2)^{-1}\mu,
\end{equation}
where
\begin{align*}
& T_1=\mu^{-1}ib\cdot\nabla(P-z^2)^{-1}\mu -\sum_{\ell_1=0}^1 \widetilde L^\sharp_{\ell_1}(z)\mu^{1-\ell_1}(i\mu^{-1}b\cdot\nabla)^{\ell_1}((P_0 - \lambda^2)^{-1} - (P_0 - z^2)^{-1})\mu ,\\
& T_2=\sum_{\ell_1=0}^1\sum_{\ell_2=0}^1\widetilde L^\sharp_{\ell_1}(z)\mu^{1-\ell_1}(i\mu^{-1}b\cdot\nabla)^{\ell_1}((P_0-\lambda^2)^{-1}-(P_0-z^2)^{-1})
(\widetilde V\mu^{-1})^{1-\ell_2}(i\nabla\cdot b\mu^{-1})^{\ell_2},\\
& T_3=\sum_{\ell_1=0}^1\widetilde L^\sharp_{\ell_1}(z)\mu^{1-\ell_1}
(i\mu^{-1}b\cdot\nabla)^{\ell_1}((P_0-\lambda^2)^{-1}-(P_0-z^2)^{-1})\mu,\\
&\widetilde L_0^\sharp=\mu^{-1}ib\cdot\nabla(P-z^2)^{-1}(\widetilde V+i\nabla\cdot b)\mu^{-1},\\
&\widetilde L_1^\sharp=-I+\mu^{-1}ib\cdot\nabla(P-z^2)^{-1}\mu.
\end{align*}
Fix $z \in \C^-$ and consider the above operators as functions of $\lambda$. Due to the exponential decay in \eqref{eq:1.3}, the operators 
$\widetilde L_0^\sharp$
and $\widetilde L_1^\sharp$ are bounded on $L^2(\R^d)$. Furthermore, the operators 
$(i\mu^{-1}b\cdot\nabla)^{\ell_1}(P_0-\lambda^2)^{-1}\mu$, $\ell_1=0,1$ are compact and, in view of Lemma \ref{C.2}, extend holomorphically
to $\mathcal{L}_{\gamma_0}$ for some constant $\gamma_0>0$. Hence $T_3(\lambda,z)$ is a family of compact operators, analytic in 
$\mathcal{L}_{\gamma_0}$. Therefore, since $T_3(z,z)\equiv 0$, by the Fredholm theorem we conclude $(I-T_3(\lambda,z))^{-1}$ exists as a
meromorphic in $\mathcal{L}_{\gamma_0}$ operator-valued function. Thus by (\ref{eq:3.12}), still for $\lambda\in\mathbb{C}^-$, we get
\begin{equation}\label{eq:3.13}
\mu^{-1}ib\cdot\nabla(P-\lambda^2)^{-1}\mu=(I-T_3)^{-1}T_1+(I-T_3)^{-1}T_2\mu(P-\lambda^2)^{-1}\mu.
\end{equation}
By (\ref{eq:3.11}) and (\ref{eq:3.13}),
\begin{equation}\label{eq:3.14}
\mu(P-\lambda^2)^{-1}\mu=F_1(\lambda,z)+F_2(\lambda,z)\mu(P-\lambda^2)^{-1}\mu,
\end{equation}
where
\begin{align*}
 F_1&=\mu(P-z^2)^{-1}\mu + \sum_{\ell_1 = 0}^1 L^\sharp_{\ell_1}(z)\mu^{1-\ell_1}(-i\mu^{-1}b\cdot\nabla)^{\ell_1} ((P_0-\lambda^2)^{-1}-(P_0-z^2)^{-1})\mu \\
 &-\sum_{\ell_1=0}^1L^\sharp_{\ell_1}(z)\mu^{1-\ell_1}(-i\mu^{-1}b\cdot\nabla)^{\ell_1}
((P_0-\lambda^2)^{-1}-(P_0-z^2)^{-1})\mu(I-T_3)^{-1}T_1,\\
 F_2&=\sum_{\ell_1=0}^1\sum_{\ell_2=0}^1L^\sharp_{\ell_1}(z)\mu^{1-\ell_1}(-i\mu^{-1}b\cdot\nabla)^{\ell_1} \\
 &\cdot ((P_0-\lambda^2)^{-1}-(P_0-z^2)^{-1})(-i\nabla\cdot b\mu^{-1})^{\ell_2}(-\widetilde V\mu^{-1})^{1-\ell_2}\\
&-\sum_{\ell_1=0}^1L^\sharp_{\ell_1}(z)\mu^{1-\ell_1}(-i\mu^{-1}b\cdot\nabla)^{\ell_1}
((P_0-\lambda^2)^{-1}-(P_0-z^2)^{-1})\mu(I-T_3)^{-1}T_2.
\end{align*}
It is easy to see that the operator $F_2$ sends $L^2(\mathbb{R}^d)$ into $H^1(\mathbb{R}^d)$. 
Therefore $F_2$ is a meromorphic (in $\lambda\in\mathcal{L}_{\gamma_0}$) family of compact operators on $L^2(\R^d)$. Since $F_2(z,z)\equiv 0$, 
this implies that $(I-F_2)^{-1}$ and $F_1$ 
are meromorphic operator-valued functions in $\mathcal{L}_{\gamma_0}$ and, by (\ref{eq:3.14}), we have
\begin{equation}\label{eq:3.15}
\mu(P-\lambda^2)^{-1}\mu=(I-F_2(\lambda,z))^{-1}F_1(\lambda,z).
\end{equation}
Thus we conclude that 
\begin{equation*}
\mu(P-\lambda^2)^{-1}\mu:L^2(\R^d) \to L^2(\R^d)
\end{equation*}
extends meromorphically from $\mathbb{C}^-$ to $\mathcal{L}_{\gamma_0}$. Note also that in view of the resolvent estimate \eqref{eq:3.2}, the identity (\ref{eq:3.15}) extends to all $z\in \mathbb{R}$, $z\neq 0$. 

Let now $0<{\rm Im}\,\lambda<\gamma_0$, $z={\rm Re}\,\lambda$, $|z|\ge\delta$, $0<\delta\ll 1$ being arbitrary. It follows from the resolvent estimate (\ref{eq:3.2}) that
\begin{equation*}
\|\widetilde L_\ell^\sharp(z)\|\lesssim |z|^{1-\ell},\quad \ell=0,1,
\end{equation*}
which together with (\ref{eq:C.6}) imply
\begin{equation}\label{eq:3.16}
\|T_3(\lambda,z)\|\lesssim {\rm Im}\,\lambda\le 1/2,
\end{equation}
if ${\rm Im}\,\lambda\le\gamma_1$ with some constant $0<\gamma_1<\gamma_0$. By (\ref{eq:3.2}) and (\ref{eq:C.5}) we also have
\begin{equation}\label{eq:3.17}
\|T_j(\lambda,z)\|\lesssim |z|^{j-1},\quad j=1,2,
\end{equation}
\begin{equation}\label{eq:3.18}
\|L_\ell^\sharp(z)\|\lesssim |z|^{-\ell},\quad \ell=0,1.
\end{equation}
By (\ref{eq:3.16}), (\ref{eq:3.17}) and (\ref{eq:3.18}) together with (\ref{eq:C.6}),
\begin{equation}\label{eq:3.19}
\|F_1(\lambda,z)\|\lesssim |z|^{-1},
\end{equation}
\begin{equation}\label{eq:3.20}
\|F_2(\lambda,z)\|\lesssim {\rm Im}\,\lambda\le 1/2,
\end{equation}
if ${\rm Im}\,\lambda\le\gamma_2$ with some constant $0<\gamma_2<\gamma_1$. By (\ref{eq:3.15}) and (\ref{eq:3.20}) we conclude 
 $\mu(P-\lambda^2)^{-1}\mu$
is analytic in $\{\lambda\in\mathcal{L}_{\gamma_2},\,|{\rm Re}\,\lambda|\ge\delta\}$. In odd dimensions, if \eqref{eq:1.6} holds, then 
 $\mu(P-\lambda^2)^{-1}\mu$
is analytic in $\mathcal{L}_{\gamma_2}$ since \eqref{eq:1.6} implies $\lambda = 0$ is not a pole.
The estimate (\ref{eq:3.7}) with $\ell=0$, $\gamma=\gamma_2$, follows from (\ref{eq:3.15}), (\ref{eq:3.19}) and (\ref{eq:3.20}). 
The estimate (\ref{eq:3.7}) with $\ell=1$ is obtained by combining
\eqref{eq:3.7} with $\ell=0$, the first identity in
\eqref{eq:3.9}, \eqref{eq:3.13}, and \eqref{eq:C.5}.
\eproof

\section{Resolvent bounds in the exterior of a non-trapping obstacle} \label{resolv bds nontrap obs section}

Let $\mathcal{O}\subset\mathbb{R}^d$, $d\ge 2$, be a bounded domain with smooth boundary such that 
$\Omega=\mathbb{R}^d\setminus\mathcal{O}$ is connected. 
Denote by $P$ the Dirichlet self-adjoint realization of $-\Delta+V$   
on the Hilbert space $L^2(\Omega)$, where $V\in L^\infty(\Omega)$ is a real-valued potential satisfying
\begin{equation}\label{eq:4.1}
|V(x)|\le C\langle x\rangle^{-\rho}, \quad C>0,\rho>1.
\end{equation}
We have

\begin{Theorem} \label{4.1}
Under the conditions (\ref{eq:1.4}) and (\ref{eq:4.1}), given any $s>1/2$ and $\delta>0$ there is a constant $C>0$ such that 
\begin{equation}\label{eq:4.2}
\left\|\langle x\rangle^{-s}\partial_x^\alpha(P-\lambda^2\pm i\varepsilon)^{-1}\partial_x^\beta\langle x\rangle^{-s}\right\|
\le C\lambda^{|\alpha|+|\beta|-1},\quad\lambda\ge\delta,\,0<\varepsilon<1,
\end{equation}
 where $\alpha$ and $\beta$ are multi-indices such that $|\alpha|\le 1$ and $|\beta|\le 1$.  
\end{Theorem}

{\it Proof.} In view of the coercivity of the operator $\widetilde P$, the bound (\ref{eq:1.4}) implies
\begin{equation}\label{eq:4.3}
\left\|\chi\partial_x^\alpha(\widetilde P-\lambda^2\pm i\varepsilon)^{-1}
\partial_x^\beta\chi\right\|
\lesssim\lambda^{|\alpha|+|\beta|-1},\quad\lambda\ge\lambda_0,
\end{equation}
for all multi-indices $\alpha$ and $\beta$ such that $|\alpha|\le 1$ and $|\beta|\le 1$. Let us see that (\ref{eq:4.3}) implies the weighted
resolvent bounds
\begin{equation}\label{eq:4.4}
\left\|\langle x\rangle^{-s}\partial_x^\alpha(\widetilde P-\lambda^2\pm i\varepsilon)^{-1}
\partial_x^\beta\langle x\rangle^{-s}\right\|
\lesssim\lambda^{|\alpha|+|\beta|-1},\quad\lambda\ge\lambda_0,
\end{equation}
for every $s>1/2$ and all multi-indices $\alpha$ and $\beta$ such that $|\alpha|\le 1$ and $|\beta|\le 1$. 
To this end, we use the fact that (\ref{eq:4.4}) holds for the 
operator $P_0$, the self-adjoint realization of $-\Delta$
on $L^2(\mathbb{R}^d)$ (see Lemma \ref{C.1}). 
Let $\eta, \, \chi \in C^\infty(\mathbb{R}^d)$ be of compact support such that
$\eta =1$ on $\mathcal{O}$ and $\chi=1$ on supp$\,\eta$.  We have the identity
\begin{equation*}
(P_0-\lambda^2\pm i\varepsilon)(1-\eta)(\widetilde P-\lambda^2\pm i\varepsilon)^{-1}=
[\Delta,\eta](\widetilde P-\lambda^2\pm i\varepsilon)^{-1}+1-\eta,
\end{equation*}
which implies
\begin{equation}\label{eq:4.5}
(1-\eta)(\widetilde P-\lambda^2\pm i\varepsilon)^{-1}=(P_0-\lambda^2\pm i\varepsilon)^{-1}[\Delta,\eta](\widetilde P-\lambda^2\pm i\varepsilon)^{-1}+(P_0-\lambda^2\pm i\varepsilon)^{-1}(1-\eta).
\end{equation}
Similarly,
\begin{equation}\label{eq:4.6}
(\widetilde P-\lambda^2\pm i\varepsilon)^{-1}(1-\eta)=(\widetilde P-\lambda^2\pm i\varepsilon)^{-1}[\Delta,\eta](P_0-\lambda^2\pm i\varepsilon)^{-1}+(1-\eta)(P_0-\lambda^2\pm i\varepsilon)^{-1}.
\end{equation}
By \eqref{eq:4.3}, (\ref{eq:4.5}), and (\ref{eq:4.6}),
\begin{equation*}
\begin{split}
&\left\|\langle x\rangle^{-s}\partial_x^\alpha(\widetilde P-\lambda^2\pm i\varepsilon)^{-1}
\partial_x^\beta\langle x\rangle^{-s}\right\|\le \left\|\chi\partial_x^\alpha(\widetilde P-\lambda^2\pm i\varepsilon)^{-1}
\partial_x^\beta\langle x\rangle^{-s}\right\|\\
&+\left\|\langle x\rangle^{-s}(1-\eta)\partial_x^\alpha(\widetilde P-\lambda^2\pm i\varepsilon)^{-1}
\partial_x^\beta\langle x\rangle^{-s}\right\|\\
&\lesssim \left\|\chi\partial_x^\alpha(\widetilde P-\lambda^2\pm i\varepsilon)^{-1}
\partial_x^\beta\langle x\rangle^{-s}\right\|+\lambda^{|\alpha|}\left\|\chi(\widetilde P-\lambda^2\pm i\varepsilon)^{-1}
\partial_x^\beta\langle x\rangle^{-s}\right\|+\lambda^{|\alpha|+|\beta|-1}\\
&\lesssim \left\|\chi\partial_x^\alpha(\widetilde P-\lambda^2\pm i\varepsilon)^{-1}
\partial_x^\beta\chi\right\|+\lambda^{|\alpha|}\left\|\chi(\widetilde P-\lambda^2\pm i\varepsilon)^{-1}
\partial_x^\beta\chi\right\|+\lambda^{|\alpha|+|\beta|-1}\\
&+\left\|\chi\partial_x^\alpha(\widetilde P-\lambda^2\pm i\varepsilon)^{-1}
\partial_x^\beta(1-\eta)\langle x\rangle^{-s}\right\|+\lambda^{|\alpha|}\left\|\chi(\widetilde P-\lambda^2\pm i\varepsilon)^{-1}
\partial_x^\beta(1-\eta)\langle x\rangle^{-s}\right\|+\lambda^{|\alpha|+|\beta|-1}\\
&\lesssim \left\|\chi\partial_x^\alpha(\widetilde P-\lambda^2\pm i\varepsilon)^{-1}
\partial_x^\beta\chi\right\|+\lambda^{|\beta|}\left\|\chi\partial_x^\alpha(\widetilde P-\lambda^2\pm i\varepsilon)^{-1}
\chi\right\|\\
&+\lambda^{|\alpha|+|\beta|}\left\|\chi(\widetilde P-\lambda^2\pm i\varepsilon)^{-1}
\chi\right\|+\lambda^{|\alpha|+|\beta|-1}\lesssim \lambda^{|\alpha|+|\beta|-1}.
\end{split}
 \end{equation*}

We now derive (\ref{eq:4.2}) from (\ref{eq:4.4}) for large $\lambda$. 
To this end we use the resolvent identity
\begin{equation}\label{eq:4.7}
(I+K(\lambda))\langle x\rangle^{-s}(P-\lambda^2\pm i\varepsilon)^{-1}\langle x\rangle^{-s}
=\langle x\rangle^{-s}(\widetilde P-\lambda^2\pm i\varepsilon)^{-1}\langle x\rangle^{-s},
\end{equation}
where
 \begin{equation*}
K(\lambda)=\langle x\rangle^{-s}(\widetilde P-\lambda^2\pm i\varepsilon)^{-1}\langle x\rangle^{s}V.
 \end{equation*}
If $1/2<s\le\rho/2$, by (\ref{eq:4.4}) we get
\begin{equation}\label{eq:4.8}
\|K(\lambda)\|\le C\lambda^{-1}\le 1/2,
\end{equation}
for $\lambda\gg 1$. It follows from (\ref{eq:4.4}), (\ref{eq:4.7}) 
and (\ref{eq:4.8}) that there is a constant $\lambda_1>\lambda_0$ such that 
(\ref{eq:4.2}) with $\alpha=\beta=0$ holds for $\lambda\ge\lambda_1$.
In the general case (\ref{eq:4.2}) follows from the identities
\begin{equation}\label{eq:4.9}
\begin{split}
&\langle x\rangle^{-s}\partial_x^\alpha(P-\lambda^2\pm i\varepsilon)^{-1}\partial_x^\beta\langle x\rangle^{-s}-
\langle x\rangle^{-s}\partial_x^\alpha(\widetilde P-\lambda^2\pm i\varepsilon)^{-1}\partial_x^\beta\langle x\rangle^{-s}\\
&=-\langle x\rangle^{-s}\partial_x^\alpha(\widetilde P-\lambda^2\pm i\varepsilon)^{-1}
V(P-\lambda^2\pm i\varepsilon)^{-1}\partial_x^\beta\langle x\rangle^{-s}\\
&=-\langle x\rangle^{-s}\partial_x^\alpha(P-\lambda^2\pm i\varepsilon)^{-1}
V(\widetilde P-\lambda^2\pm i\varepsilon)^{-1}\partial_x^\beta\langle x\rangle^{-s},
\end{split}
\end{equation}
together with (\ref{eq:4.4}) and (\ref{eq:4.2}) with $\alpha=\beta=0$. 

Next we prove (\ref{eq:4.2}) as well as (\ref{eq:4.4}) for $\delta\le\lambda\le\lambda_1$, $0<\delta\ll 1$ being arbitrary, by using the Carleman estimate (\ref{eq:2.5}). We keep the same notations as in Section 2. 
Given a function $g$ such that $\langle x\rangle^{s}g\in L^2(\Omega)$, set
 \begin{equation*}
f=(P-\lambda^2\pm i\varepsilon)^{-1}g.
 \end{equation*}
Clearly, $f|_{\partial\Omega}=0$. Let $a\gg 1$ be such that $\mathcal{O}\subset B_a:=\{x\in \mathbb{R}^d:|x|\le a\}$. 
Choose  functions $\psi_a,\widetilde\psi_a\in C_0^\infty(\mathbb{R}^d)$
such that $\widetilde\psi_a(x)=1$ for $|x|\le a+1$, $\widetilde\psi_a(x)=0$ for $|x|\ge a+2$, $\psi_a(x)=1$ for $|x|\le a+3$, $\psi_a(x)=0$ for $|x|\ge a+4$. Now Theorem 2.1 of \cite{kn:V4} applied to the function $\psi_a f$ (with $h=1$) leads to the estimate
\begin{equation}\label{eq:4.10}
\begin{split}
\|\psi_a f\|_{H^1(\Omega)}&\lesssim \|(P-\lambda^2\pm i\varepsilon)(\psi_a f)\|_{L^2(\Omega)}\\
&\lesssim \|\psi_a g\|_{L^2(\Omega)}+\|[\Delta,\psi_a]f\|_{L^2(\Omega)}\lesssim \|\psi_a g\|_{L^2(\Omega)}
+\|f\|_{H^1(B_{a+4}\setminus B_{a+3})}.
\end{split}
\end{equation}
Let $1/2<s<\min\{1,\rho/2\}$. We now use the estimate
(\ref{eq:2.5}) with $f$ replaced by $e^{\tau\varphi}(1-\widetilde\psi_a)f$. 
Since $h^{-1}\le \tau+\lambda_1$, we obtain the estimate
\begin{equation}\label{eq:4.11}
\begin{split}
&\|\langle x\rangle^{-s}e^{\tau\varphi}(1-\widetilde\psi_a)f\|_{H^1(\mathbb{R}^d)}\le
Ch^{-1}\|\langle x\rangle^{-s}e^{\tau\varphi}(1-\widetilde\psi_a)f\|_{H^1_h(\mathbb{R}^d)}\\
&\le C\tau^{-1/2}\|\langle x\rangle^{s}e^{\tau\varphi}
(-\Delta-\lambda^2\pm i\varepsilon)(1-\widetilde\psi_a)f\|_{L^2(\mathbb{R}^d)}
+C_\tau\varepsilon^{1/2}\|e^{\tau\varphi}(1-\widetilde\psi_a)f\|_{L^2(\mathbb{R}^d)}\\
&\le C\tau^{-1/2}\|\langle x\rangle^{s}e^{\tau\varphi}
(-\Delta+V-\lambda^2\pm i\varepsilon)(1-\widetilde\psi_a)f\|_{L^2(\mathbb{R}^d)}\\
&+C\tau^{-1/2}\|\langle x\rangle^{-s}e^{\tau\varphi}(1-\widetilde\psi_a)f\|_{L^2(\mathbb{R}^d)}+
C_\tau\varepsilon^{1/2}\|e^{\tau\varphi}(1-\widetilde\psi_a)f\|_{L^2(\mathbb{R}^d)}.
\end{split}
\end{equation}
Hereafter $C>0$ denotes a constant, independent of $\tau$, which may change from line to line,
while $C_\tau>0$ denotes a constant, depending on $\tau$, 
which may change from line to line and whose precise value is not important in the analysis
that follows. 
Taking $\tau$ large enough we can absorb the second term in the right-hand side of (\ref{eq:4.11}) to obtain
\begin{equation}\label{eq:4.12}
\begin{split}
&\|\langle x\rangle^{-s}e^{\tau\varphi}(1-\widetilde\psi_a)f\|_{H^1(\mathbb{R}^d)}\le
 C\|\langle x\rangle^{s}e^{\tau\varphi}(1-\widetilde\psi_a)g\|_{L^2(\mathbb{R}^d)}\\
&+C\|\langle x\rangle^{s}e^{\tau\varphi}[\Delta,\widetilde\psi_a]f\|_{L^2(\mathbb{R}^d)}
+C_\tau\varepsilon^{1/2}\|e^{\tau\varphi}(1-\widetilde\psi_a)f\|_{L^2(\mathbb{R}^d)}\\
&\le C\|\langle x\rangle^{s}g\|_{L^2(\Omega)}+Ce^{\tau\varphi(a+2)}\|f\|_{H^1(B_{a+2}\setminus B_{a+1})}
+C_\tau \varepsilon^{1/2}\|f\|_{L^2(\Omega)}.
\end{split}
\end{equation}
In particular, (\ref{eq:4.12}) implies
\begin{equation}\label{eq:4.13}
\begin{split}
&e^{\tau\varphi(a+3)}\|f\|_{H^1(B_{a+4}\setminus B_{a+3})}\\
&\le C\|\langle x\rangle^{s}g\|_{L^2(\Omega)}+Ce^{\tau\varphi(a+2)}\|f\|_{H^1(B_{a+2}\setminus B_{a+1})}
+C_\tau\varepsilon^{1/2}\|f\|_{L^2(\Omega)}\\
&\le
 C\|\langle x\rangle^{s}g\|_{L^2(\Omega)}+Ce^{\tau\varphi(a+2)}\|f\|_{H^1(B_{a+4}\setminus B_{a+3})}
+C_\tau\varepsilon^{1/2}\|f\|_{L^2(\Omega)},
\end{split}
\end{equation}
where we have also used (\ref{eq:4.10}). Since $\varphi(a+3)-\varphi(a+2)>0$ is independent of $\tau$, we can absorb the second term in the right-hand side of 
(\ref{eq:4.13}) by taking $\tau$ large enough. We now fix $\tau$. Thus we obtain 
\begin{equation}\label{eq:4.14}
\|f\|_{H^1(B_{a+4}\setminus B_{a+3})}\lesssim
 \|\langle x\rangle^{s}g\|_{L^2(\Omega)}
+\varepsilon^{1/2}\|f\|_{L^2(\Omega)}.
\end{equation}
Combining (\ref{eq:4.10}), (\ref{eq:4.12}) and (\ref{eq:4.14}) leads to
\begin{equation}\label{eq:4.15}
\|\langle x\rangle^{-s}f\|_{H^1(\Omega)}\lesssim
\|\langle x\rangle^{s}g\|_{L^2(\Omega)}
+\varepsilon^{1/2}\|f\|_{L^2(\Omega)}.
\end{equation}
On the other hand, the symmetry of the operator $P$ on the Hilbert space $L^2(\Omega)$
gives
\begin{equation}\label{eq:4.16} 
\begin{split}
\varepsilon\|f\|_{L^2(\Omega)}^2&= \left | \Im \langle (P- \lambda^2 \pm i \varepsilon)f, f \rangle_{L^2(\Omega)} \right |\\
&\le \left|\left\langle \langle x\rangle^{s}g,\langle x\rangle^{-s}f\right\rangle_{L^2(\Omega)}\right|\\
&\le \gamma\|\langle x\rangle^{-s}f\|_{L^2(\Omega)}^2+
\gamma^{-1}\|\langle x\rangle^{s}g\|_{L^2(\Omega)}^2
\end{split}
\end{equation}
for every $\gamma>0$. Combining (\ref{eq:4.15}), (\ref{eq:4.16}) and taking $\gamma$ small enough we obtain the estimate
\begin{equation}\label{eq:4.17}
\|\langle x\rangle^{-s}f\|_{H^1(\Omega)}\lesssim
\|\langle x\rangle^{s}g\|_{L^2(\Omega)},
\end{equation}
which implies (\ref{eq:4.2}) as well as (\ref{eq:4.4}) for $\delta\le\lambda\le\lambda_1$ and $|\alpha|\le 1$, $\beta=0$. 
For $|\alpha|\le 1$, $|\beta|\le 1$ the estimate (\ref{eq:4.4}) follows from the coercivity of the operator $\widetilde P$,
while (\ref{eq:4.2}) follows from (\ref{eq:4.4}) and the identities (\ref{eq:4.9}).
\eproof

Like in the previous section, we develop the meromorphic continuation of the operator $\mu (P - \lambda^2)^{-1} \mu : L^2(\Omega) \to L^2(\Omega)$, and establish resolvent bounds crucial for obtaining the wave decay in Section \ref{time decay section}.

\begin{Theorem} \label{4.2}
Assume the conditions (\ref{eq:1.3}) and (\ref{eq:1.4}) fulfilled. Then, given any $\delta>0$ and any integer $k\ge 0$, the bound
\begin{equation}\label{eq:4.18}
\left\|\frac{d^k}{d\lambda^k}\left(\mu\nabla^{\ell}(P-\lambda^2)^{-1}\mu\right)\right\|
\le C^{k+1}k!(|\lambda|+1)^{\ell-1}
\end{equation}
holds for all $\lambda\in\R$, $|\lambda|\ge\delta$, with a constant $C=C_\delta>0$, where $\ell\in\{0,1\}$.
If $d$ is odd and the condition \eqref{eq:1.6} is assumed, the bound \eqref{eq:4.18} holds for all $\lambda\in\R$.   
\end{Theorem}

{\it Proof.} We follow the same strategy as in the proof of Proposition \ref{3.2}. 
Let $\eta\in C^\infty(\mathbb{R}^d)$ be of compact support such that
$\eta =1$ on $\mathcal{O}$.  For $\lambda\in\mathbb{C}^-$ we have
 \begin{equation*}
(P_0-\lambda^2)(1-\eta)(P-\lambda^2)^{-1}=([\Delta,\eta]-(1-\eta)V)(P-\lambda^2)^{-1}+1-\eta, \qquad \text{on } L^2(\Omega),
 \end{equation*}
which implies
\begin{equation}\label{eq:4.19}
(1-\eta)(P-\lambda^2)^{-1} =(P_0-\lambda^2)^{-1}([\Delta,\eta]-(1-\eta)V)(P-\lambda^2)^{-1}+(P_0-\lambda^2)^{-1}(1-\eta).
\end{equation}
Let $z\in\mathbb{C}^-$. Similarly,
\begin{equation}\label{eq:4.20}
\begin{split}
(P&-z^2)^{-1}(1-\eta)\\
&=(P-z^2)^{-1}([\Delta,\eta]-(1-\eta)V)(P_0-z^2)^{-1}+(1-\eta)(P_0-z^2)^{-1}, \qquad \text{on } L^2(\R^d).
\end{split}
\end{equation}
In view of (\ref{eq:4.19}) and (\ref{eq:4.20}),
 \begin{equation*}
 \begin{split}
(P-\lambda^2)^{-1}-(P-z^2)^{-1}&=(\lambda^2-z^2)(P-z^2)^{-1}(P-\lambda^2)^{-1}\\
&=(\lambda^2-z^2)(P-z^2)^{-1}\eta(2-\eta)(P-\lambda^2)^{-1}\\
&+(\lambda^2-z^2)(P-z^2)^{-1}(1-\eta)^2(P-\lambda^2)^{-1}\\
&=(\lambda^2-z^2)(P-z^2)^{-1}\eta(2-\eta)(P-\lambda^2)^{-1}\\
&+(1-\eta+(P-z^2)^{-1}([\Delta,\eta]-(1-\eta)V))((P_0-\lambda^2)^{-1}\\
&-(P_0-z^2)^{-1})(1-\eta+([\Delta,\eta]-(1-\eta)V)(P-\lambda^2)^{-1}).
\end{split}
 \end{equation*}
Multiplying both sides of this identity by $\mu$ we get
\begin{equation}\label{eq:4.21}
\begin{split}
\mu(P-\lambda^2)^{-1}\mu-\mu(P-z^2)^{-1}\mu&=(\lambda^2-z^2)\mu(P-z^2)^{-1}\eta(2-\eta)(P-\lambda^2)^{-1}\mu\\
 &+Q_1(z)(\mu(P_0-\lambda^2)^{-1}\mu-\mu(P_0-z^2)^{-1}\mu)Q_2(\lambda),
 \end{split}
 \end{equation}
 where 
 \begin{align*}
& Q_1(z)=1-\eta+\mu(P-z^2)^{-1}([\Delta,\eta]-(1-\eta)V)\mu^{-1},\\
& Q_2(\lambda)=1-\eta+\mu^{-1}([\Delta,\eta]-(1-\eta)V)(P-\lambda^2)^{-1}\mu.
 \end{align*}
 We rewrite (\ref{eq:4.21}) in the form
 \begin{equation}\label{eq:4.22}
 \begin{split}
 &(I-K(\lambda,z))\mu(P-\lambda^2)^{-1}\mu=\mu(P-z^2)^{-1}\mu\\
 &+Q_1(z)(\mu(P_0-\lambda^2)^{-1}\mu-\mu(P_0-z^2)^{-1}\mu)(1-\eta),
 \end{split}
 \end{equation}
 where the operator
 \begin{equation*}
 \begin{split}
 K(\lambda,z)&=(\lambda^2-z^2)\mu(P-z^2)^{-1}\eta(2-\eta)\mu^{-1}\\
 &+Q_1(z)(\mu(P_0-\lambda^2)^{-1}-\mu(P_0-z^2)^{-1})([\Delta,\eta]-(1-\eta)V)\mu^{-1}
 \end{split}
 \end{equation*}
 sends $L^2(\Omega)$ into $H^1(\Omega)$ and extends analytically in $\lambda\in \mathcal{L}_{\gamma_0}$ in view of Lemma \ref{C.2}.
 Therefore $K(\lambda,z)$ is a family of compact operators on $L^2(\Omega)$, analytic in $\mathcal{L}_{\gamma_0}$. Since
 $K(z,z)\equiv 0$, by the Analytic Fredholm theorem $(I-K(\lambda,z))^{-1}$ exists as a meromorphic in $\mathcal{L}_{\gamma_0}$
 operator-valued function.  
 By (\ref{eq:4.22}) we get that $\mu(P-\lambda^2)^{-1}\mu$
extends meromorphically from $\mathbb{C}^-$ to $\mathcal{L}_{\gamma_0}$. 
 Moreover, the identity (\ref{eq:4.22}) extends to all $\lambda\in \mathcal{L}_{\gamma_0}$ as well as
 to all $z\in\mathbb{R}$, $z\neq 0$.
Let now $0<{\rm Im}\,\lambda<\gamma_0$, $z={\rm Re}\,\lambda$, $|z|\ge\delta$, $0<\delta\ll 1$ being arbitrary. 
It follows from (\ref{eq:4.2}) that
\begin{equation}\label{eq:4.23}
\|Q_1(z)\|\lesssim 1.
\end{equation}
By (\ref{eq:4.2}), (\ref{eq:4.23}) and (\ref{eq:C.6}),
\begin{equation}\label{eq:4.24}
\|K(\lambda,z)\|\lesssim {\rm Im}\,\lambda\le 1/2
\end{equation}
for ${\rm Im}\,\lambda\le\gamma_1$ with some constant $0<\gamma_1<\gamma_0$. Thus, by (\ref{eq:4.22}) and (\ref{eq:4.24}) we obtain that 
$\mu(P-\lambda^2)^{-1}\mu$ extends analytically to 
$\{\lambda\in\mathcal{L}_{\gamma_1},\,|{\rm Re}\,\lambda|\ge\delta\}$. 
In odd dimensions 
 $\mu(P-\lambda^2)^{-1}\mu$
is analytic in $\mathcal{L}_{\gamma_1}$ since the condition (\ref{eq:1.6}) implies that $\lambda = 0$ is not a pole.
Also from \eqref{eq:4.22} it is easy to see that the analog of (\ref{eq:3.7}) is valid in this case, whence (\ref{eq:4.18}) follows from this fact and Lemma \ref{B.1}. 
\eproof

\section{Low-frequency resolvent bounds}
 
 Let $P: L^2(\R^d) \to L^2(\R^d)$ be the self-adjoint operator from Section 3. 
 In this section we will suppose that
 \begin{equation}\label{eq:5.1}
0\le V(x)\le C\langle x\rangle^{-\rho}, \quad |b(x)|\le C\langle x\rangle^{-\rho},
\end{equation}
 with constants $C>0$, $\rho>\max\{3,\frac{d}{2}\}$. We have the following
 
 \begin{Theorem} \label{5.1}
 Let $d\ge 5$ and assume the condition (\ref{eq:5.1}) is fulfilled. 
If $s>1$, we have the low-frequency estimate
\begin{equation}\label{eq:5.2}
\left\|\langle x\rangle^{-s}\nabla^{\ell}(P-\lambda^2\pm i\varepsilon)^{-1}\langle x\rangle^{-s}\right\|
\le C,\quad 0<\lambda\le\delta_0,\,0<\varepsilon \le \varepsilon_0,
\end{equation}
 with constants $0<\delta_0, \, \varepsilon_0 \ll 1$, $C>0$ independent of $\lambda$ and $\varepsilon$, where $\ell\in\{0,1\}$.
 \end{Theorem}
 
 {\it Proof.} Given a function $g \in L^2(\R^d)$ such that $\langle x\rangle^{s}g\in L^2(\R^d)$, set
\begin{equation*}
f=(P-\lambda^2\pm i\varepsilon)^{-1}g.
\end{equation*}
Let $a\gg 1$ be a parameter independent of $\lambda$ and choose $\chi_a\in C_0^\infty(\mathbb{R}^d; [0,1])$
such that $\chi_a(x)=1$ for $|x|\le 3a$, $\chi_a(x)=0$ for $|x|\ge 4a$, and $\partial_x^\alpha\chi_a(x)=O(a^{-|\alpha|})$.

For the rest of the proof $\|\cdot\|$ and $\langle\cdot,\cdot\rangle$ denote the norm and the scalar product in $L^2(\mathbb{R}^d)$. We have
\begin{equation*}
(P-\lambda^2\pm i\varepsilon)(\chi_af)=\chi_ag+[P,\chi_a]f.
\end{equation*}
Hence
\begin{equation*}
\begin{split}
{\rm Re}\,\left\langle \chi_ag+[P,\chi_a]f,\chi_af\right\rangle &={\rm Re}\,\left\langle P\chi_af,\chi_af\right\rangle
-\lambda^2\|\chi_af\|^2\\
&=\|(i\nabla+b)\chi_af\|^2+\left\langle V\chi_af,\chi_af\right\rangle-\lambda^2\|\chi_af\|^2\\
&\ge\|(i\nabla+b)\chi_af\|^2-\lambda^2\|\chi_af\|^2.
\end{split}
\end{equation*}
Thus we obtain
\begin{equation}\label{eq:5.3}
\|(i\nabla+b)\chi_af\|\le (\lambda+\gamma)\|\chi_af\|+\gamma^{-1}\|\chi_ag\|+\gamma^{-1}\|[P,\chi_a]f\|
\end{equation}
for every $\gamma>0$. On the other hand, by the Poincar\'e inequality (\ref{eq:D.1}) we have 
\begin{equation}\label{eq:5.4}
\|\chi_af\|\le Ca\|(i\nabla+b)\chi_af\|.
\end{equation}
We now combine (\ref{eq:5.3}) and (\ref{eq:5.4}). Choosing $\gamma=a^{-1}\gamma_0$ with $\gamma_0>0$ 
small enough independent of $a$, and $\lambda >0$ small enough depending on $a$, we arrive at
\begin{equation}\label{eq:5.5}
a^{-1}\|\chi_af\|+\|(i\nabla+b)\chi_af\|\le 
Ca\|\chi_ag\|+Ca\|[P,\chi_a]f\|.
\end{equation}
On the other hand, using the resolvent identity (\ref{eq:3.9}) we obtain
\begin{equation}\label{eq:5.6}
[P,\chi_a]f=[P,\chi_a](P_0-\lambda^2\pm i\varepsilon)^{-1}g-[P,\chi_a](P_0-\lambda^2\pm i\varepsilon)^{-1}
(\widetilde V+i\nabla\cdot b+ib\cdot\nabla)f.
\end{equation}
Observe now that $[P,\chi_a]$ is supported in $3a\le|x|\le 4a$ and 
\begin{equation*}
[P,\chi_a]=[-\Delta,\chi_a]+2ib\cdot\nabla\chi_a=-\Delta\chi_a-2\nabla\chi_a\cdot\nabla+2ib\cdot\nabla\chi_a
=O(a^{-1})\cdot\nabla+O(a^{-2}).
\end{equation*}
Hence, in view of Lemma \ref{C.1}, given $0<\epsilon\ll 1$ and $h \in L^2(\R^d)$ with $\langle x \rangle^{1 + \epsilon} h \in L^2(\R^d)$, we have (with $|\alpha| \le 1$),
\begin{equation}\label{eq:5.7}
\begin{split}
&\left\|[P,\chi_a](P_0-\lambda^2\pm i\varepsilon)^{-1} \partial_x^\alpha h\right\|\\
&\lesssim 
\sum_{j=0}^1a^{-1+j/2+\epsilon}\left\|\langle x\rangle^{j/2-1-\epsilon}\nabla^j(P_0-\lambda^2
\pm i\varepsilon)^{-1} \partial_x^\alpha \langle x\rangle^{j/2-1-\epsilon}\right\|\|\langle x\rangle^{-j/2+1+\epsilon}h\|\\
&\lesssim \sum_{j=0}^1a^{-1+j/2+\epsilon}\|\langle x\rangle^{-j/2+1+\epsilon}h\|.
\end{split}
\end{equation}
Choose a function $\widetilde\chi_a\in C_0^\infty(\mathbb{R}^d)$
such that $\widetilde\chi_a(x)=1$ for $|x|\le a$, $\widetilde\chi_a(x)=0$ for $|x|\ge 2a$, and $\partial_x^\alpha\widetilde\chi_a(x)=O(a^{-|\alpha|})$. We will now bound the norms of the functions
\begin{align*}
& f_1:=[P,\chi_a](P_0-\lambda^2\pm i\varepsilon)^{-1}
(\widetilde V+i\nabla\cdot b+ib\cdot\nabla)(1-\widetilde\chi_a)f,\\
& f_2:=[P,\chi_a](P_0-\lambda^2\pm i\varepsilon)^{-1}
(\widetilde V+i\nabla\cdot b+ib\cdot\nabla)\widetilde\chi_af.
\end{align*}
We bound the norm of $f_1$ by applying (\ref{eq:5.7}) with $\alpha = 0$ to the function $h = (\widetilde V+ib\cdot\nabla)(1-\widetilde\chi_a)f$, and applying (\ref{eq:5.7}) with $|\alpha| = 1$ to the entries of $b(1 -\tilde \chi_a)f$. We get
\begin{equation}\label{eq:5.8}
\|f_1\|\lesssim a^{1+3\epsilon-\rho}\sum_{j=0}^1\|\langle x\rangle^{-1-\epsilon}\nabla^j((1-\widetilde\chi_a)f)\|.
\end{equation}
To bound the norm of $f_2$ we will use that the kernel of the free resolvent 
$(P_0-\lambda^2\pm i\varepsilon)^{-1}$ is of the form $z^{d-2}E_d^\pm(z|x-y|)$, where $z^2=\lambda^2\mp i\varepsilon$, $\pm{\rm Im}\,z>0$,
and the function 
$E_d^\pm(\zeta)$ is given in terms of the Hankel functions by the formula
\begin{equation}\label{eq:5.9}
E_d^\pm(\zeta)=C_d\zeta^{-\frac{d-2}{2}}H^\pm_{\frac{d-2}{2}}(\zeta).
\end{equation}
It is well-known that
\begin{equation}\label{eq:5.10}
\left|\partial_\zeta^kE_d^\pm(\zeta)\right|\lesssim|\zeta|^{-d+2-k} \quad\mbox{for}\quad |\zeta|\le 1,\quad k=0,1,2.
\end{equation}
Observe now that if $x\in {\rm supp}\,[P,\chi_a]$, $y\in {\rm supp}\,\widetilde\chi_a$, then $a\le|x-y|\le 6a$.
Hence we can arrange $|z||x-y|\le 1$ by taking $\lambda$ smaller if necessary, and by taking $\varepsilon$ small, so that $|z|a = (\lambda^4 + \varepsilon^2)^{1/4} a \ll 1$. Therefore, for such $x$, $y$ and $z$ we derive from 
(\ref{eq:5.10}):
\begin{equation}\label{eq:5.11}
\left|z^{d-2}\partial_x^{j_1}\partial_y^{j_2}E^\pm_d(z|x-y|)\right|\lesssim a^{-d+2-j_1-j_2},
\end{equation}
where $j_1,j_2\in\{0,1\}$. By (\ref{eq:5.11}),
\begin{equation*}
|f_2|\lesssim a^{-d}\sum_{j=0}^1\|\langle x\rangle^{-\rho}\nabla^j(\widetilde\chi_af)\|_{L^1}\lesssim a^{-d}\sum_{j=0}^1\|\nabla^j(\widetilde\chi_af)\|,
\end{equation*}
where we have used that $\langle x\rangle^{-\rho}\in L^2$. Hence
\begin{equation}\label{eq:5.12}
\|f_2\|^2=\int_{3a\le|x|\le 4a}|f_2|^2dx
\lesssim a^{-d}\sum_{j=0}^1\|\nabla^j(\widetilde\chi_af)\|^2.
\end{equation}
By (\ref{eq:5.6}), (\ref{eq:5.7}), (\ref{eq:5.8}), and (\ref{eq:5.12}),
\begin{equation}\label{eq:5.13}
\begin{split}
& \|[P,\chi_a]f\|\lesssim a^{-1/2+\epsilon}\|\langle x\rangle^{1+\epsilon}g\|\\
&+a^{1+3\epsilon-\rho}\sum_{j=0}^1\|\langle x\rangle^{-1-\epsilon}\nabla^j((1-\widetilde\chi_a)f)\|+
a^{-d/2}\sum_{j=0}^1\|\nabla^j(\widetilde\chi_af)\|.
\end{split}
\end{equation}
By (\ref{eq:5.5}) and (\ref{eq:5.13}),
\begin{equation}\label{eq:5.14}
\begin{split}
& a^{-1}\|\chi_af\|+\|(i\nabla+b)\chi_af\|\lesssim a\|\chi_ag\|+a^{1/2+\epsilon}\|\langle x\rangle^{1+\epsilon}g\|\\  
& +a^{2+3\epsilon-\rho}\sum_{j=0}^1\|\langle x\rangle^{-1-\epsilon}\nabla^j((1-\widetilde\chi_a)f)\|+
a^{-d/2+1}\sum_{j=0}^1\|\nabla^j(\chi_af)\|.
\end{split}
\end{equation}
Since $d\ge 5$, we can arrange that $a^{-d/2+1}\ll a^{-1}$. Therefore, taking $a$
big enough we can absorb the last term in the right-hand side of (\ref{eq:5.14}) to obtain
\begin{equation}\label{eq:5.15}
a^{-1}\|\chi_af\|+\|(i\nabla+b)\chi_af\|\lesssim a\|\langle x\rangle^{1+\epsilon}g\|
+a^{2+3\epsilon-\rho}\sum_{j=0}^1\left\|\langle x\rangle^{-1-\epsilon}\nabla^j((1-\widetilde\chi_a)f)\right\|.
\end{equation}
Observe now that the identity (\ref{eq:5.6}) still holds with $[P,\chi_a]$ replaced by $1-\widetilde\chi_a$. Using this 
together with Lemma \ref{C.1}, we get 
\begin{equation*}
\begin{split}
&\sum_{j=0}^1\left\|\langle x\rangle^{-1-\epsilon}\nabla^j((1-\widetilde\chi_a)f)\right\|\lesssim\sum_{j=0}^1\left\|
\langle x\rangle^{-1-\epsilon}\nabla^j(P_0-\lambda^2\pm i\varepsilon)^{-1}\langle x\rangle^{-1-\epsilon}\right\|
\|\langle x\rangle^{1+\epsilon}g\|\\
&+a^{2+2\epsilon-\rho}\sum_{j=0}^1\sum_{\ell_1+\ell_2\le 1}\left\|
\langle x\rangle^{-1-\epsilon}\nabla^j(P_0-\lambda^2\pm i\varepsilon)^{-1}\nabla^{\ell_1}\langle x\rangle^{-1-\epsilon}
\right\|\left\|\langle x\rangle^{-1-\epsilon}\nabla^{\ell_2}((1-\widetilde\chi_a)f)\right\|\\
&+\sum_{j=0}^1\sum_{\ell_1+\ell_2\le 1}\left\|
\langle x\rangle^{-1-\epsilon}\nabla^j(P_0-\lambda^2\pm i\varepsilon)^{-1}\nabla^{\ell_1}\langle x\rangle^{-1-\epsilon}
\right\|\left\|\nabla^{\ell_2}(\widetilde\chi_af)\right\|\\
 &\lesssim\|\langle x\rangle^{1+\epsilon}g\|+a^{2+2\epsilon-\rho}\sum_{\ell=0}^1
\left\|\langle x\rangle^{-1-\epsilon}\nabla^{\ell}((1-\widetilde\chi_a)f)\right\|
+\sum_{\ell=0}^1\left\|\nabla^{\ell}(\widetilde\chi_af)\right\|.
\end{split}
\end{equation*}
Taking $a$ larger as needed we can absorb the second term in the right-hand side of the above inequality to obtain
\begin{equation}\label{eq:5.16}
\sum_{j=0}^1\left\|\langle x\rangle^{-1-\epsilon}\nabla^j((1-\widetilde\chi_a)f)\right\|\lesssim \|\langle x\rangle^{1+\epsilon}g\|+
\sum_{\ell=0}^1\left\|\nabla^{\ell}(\widetilde\chi_af)\right\|.
\end{equation}
By (\ref{eq:5.15}) and (\ref{eq:5.16}),
\begin{equation}\label{eq:5.17}
a^{-1}\|\chi_af\|+\|(i\nabla+b)\chi_af\|\lesssim a\|\langle x\rangle^{1+\epsilon}g\|+
a^{2+3\epsilon-\rho}\sum_{\ell=0}^1\left\|\nabla^{\ell}(\chi_af)\right\|.  
\end{equation}
If $\epsilon$ is small enough we have $a^{2+3\epsilon-\rho}\ll a^{-1}$. Therefore, taking $a$ even bigger if necessary we can absorb the last
term in the right-hand side of (\ref{eq:5.17}) to obtain
\begin{equation}\label{eq:5.18}
a^{-1}\|\chi_af\|+\|(i\nabla+b)\chi_af\|\lesssim a\|\langle x\rangle^{1+\epsilon}g\|.  
\end{equation}
With $a$ now fixed we combine (\ref{eq:5.16}) and (\ref{eq:5.18}) to conclude
\begin{equation}\label{eq:5.19}
\sum_{j=0}^1\left\|\langle x\rangle^{-1-\epsilon}\nabla^jf\right\|\lesssim \|\langle x\rangle^{1+\epsilon}g\|,
\end{equation}
 which clearly implies (\ref{eq:5.2}).
\eproof

Let now $\mathcal{O}\subset\mathbb{R}^d$ be a bounded domain with smooth boundary such that 
$\Omega=\mathbb{R}^d\setminus\mathcal{O}$ is connected. In what follows in this section we will prove the following

 \begin{Theorem} \label{5.2} 
 The conclusions of Theorem \ref{5.1} remain valid for the 
 Dirichlet self-adjoint realization (which again will be denoted by $P$) of the operator $-\Delta+V :L^2(\Omega) \to L^2(\Omega)$, where $V$ satisfies the condition (\ref{eq:5.1}) in $\Omega$.
\end{Theorem}
 
 {\it Proof.} We will adapt the proof of Theorem \ref{5.1} to this case and will keep the same notations. In this follows 
 $\|\cdot\|$ and $\langle\cdot,\cdot\rangle$ will denote the norm and the scalar product on $L^2(\Omega)$. We take the parameter
 $a$ big enough so that $\chi_a=1$ on $\mathcal{O}$. 
 In this case the function 
 \begin{equation*}
 f=(P-\lambda^2\pm i\varepsilon)^{-1}g
 \end{equation*}
 satisfies the equation 
\begin{equation*}
(P-\lambda^2\pm i\varepsilon)(\chi_af)=\chi_ag-[\Delta,\chi_a]f
\end{equation*}
in $\Omega$ and $\chi_a f|_{\partial\Omega}=0$. 
Hence, by the Green formula, 
\begin{equation*}
\begin{split}
{\rm Re}\,\left\langle \chi_ag-[\Delta,\chi_a]f,\chi_af\right\rangle &={\rm Re}\,\left\langle P\chi_af,\chi_af\right\rangle
-\lambda^2\|\chi_af\|^2\\
&=\|\nabla(\chi_af)\|^2+\left\langle V\chi_af,\chi_af\right\rangle-\lambda^2\|\chi_af\|^2\\
&\ge\|\nabla(\chi_af)\|^2_{L^2}-\lambda^2\|\chi_af\|^2.
\end{split}
\end{equation*}
Thus we obtain the inequality
\begin{equation}\label{eq:5.20}
\|\nabla(\chi_af)\|\le (\lambda+\gamma)\|\chi_af\|+\gamma^{-1}\|\chi_ag\|+\gamma^{-1}\|[\Delta,\chi_a]f\|
\end{equation}
for every $\gamma>0$. On the other hand, by the Poincar\'e inequality (\ref{eq:D.2}), we have
\begin{equation}\label{eq:5.21}
\|\chi_af\|\le Ca\|\nabla(\chi_af)\|.
\end{equation}
Choosing $\gamma=a^{-1}\gamma_0$ with $\gamma_0>0$ 
small enough independent of $a$ and 
$\lambda$ small enough, we obtain from the above inequalities the estimate
\begin{equation}\label{eq:5.22}
a^{-1}\|\chi_af\|+\|\nabla(\chi_af)\|\le 
Ca\|\chi_ag\|+Ca\|[\Delta,\chi_a]f\|.
\end{equation}
On the other hand, by the resolvent identity (\ref{eq:4.12}) we have
\begin{equation}\label{eq:5.23}
(1-\eta)f=(P_0-\lambda^2\pm i\varepsilon)^{-1}(1-\eta)g+(P_0-\lambda^2\pm i\varepsilon)^{-1}([\Delta,\eta]-(1-\eta)V)f.
\end{equation}
Hence, if $a$ is big enough, we have
\begin{equation}\label{eq:5.24}
[\Delta,\chi_a]f=[\Delta,\chi_a](P_0-\lambda^2\pm i\varepsilon)^{-1}(1-\eta)g+
[\Delta,\chi_a](P_0-\lambda^2\pm i\varepsilon)^{-1}([\Delta,\eta]-(1-\eta)V)f,
\end{equation}
\begin{equation}\label{eq:5.25}
(1-\widetilde\chi_a)f=(1-\widetilde\chi_a)(P_0-\lambda^2\pm i\varepsilon)^{-1}(1-\eta)g+(1-\widetilde\chi_a)(P_0-\lambda^2\pm i\varepsilon)^{-1}([\Delta,\eta]-(1-\eta)V)f.
\end{equation}
With these formulas in hands, the proof now is exactly the same as the proof of Theorem \ref{5.1}.
Therefore we omit the details.
\eproof

\section{Time decay estimates} \label{time decay section}

In this section we use Theorems \ref{3.3} and \ref{4.2} to prove Theorem \ref{1.1} for our self-adjoint operator
\begin{equation*}
P = (i\nabla + b)^2 + V : L^2(\Omega) \to L^2(\Omega).
\end{equation*}
Recall that we consider the two cases a) and b) as described in Section \ref{s1}. We will treat these cases separately as necessary. Throughout, we suppose $P \ge 0$, for which $V \ge 0$ suffices. Furthermore, we assume $b$ and $V$ obey (\ref{eq:1.3}) and that (\ref{eq:1.4}) holds for our domain $\Omega$.

In view of Lemma \ref{E.1}, given an integer $m\gg 1$
there is a real-valued function $\rho_m\in C_0^\infty(\R)$, $\rho_m\ge 0$, $\int_{-\infty}^\infty\rho_m(\sigma)d\sigma=1$, such that $\rho_m(\sigma)=0$ for $\sigma\le 1$ and $\sigma\ge 2+\ln m$, and 
 \begin{equation}\label{eq:6.1}
\left|\partial_\sigma^k\rho_m(\sigma)\right|\le C^{k+1}k!,\quad\forall\sigma\in \R,
\end{equation}
for all integers $0\le k\le m$  with a constant $C>0$ independent of $k$ and $m$. Given any $\delta>0$, set
\begin{equation*}
\psi_m(\lambda)=\int_{-\infty}^{\lambda/\delta} \rho_m(\sigma)d\sigma,
\end{equation*}
so we have $\partial_\lambda\psi_m(\lambda)=\delta^{-1}\rho_m(\lambda/\delta)$. Therefore, by (\ref{eq:6.1}),
\begin{equation}\label{eq:6.2}
\left|\partial_\lambda^k\psi_m(\lambda)\right|\le (C/\delta)^{k}k!,\quad\forall\lambda\in \R,
\end{equation}
for all integers $0\le k\le m$.
Clearly, we also have $0\le \psi_m(\lambda)\le 1$, and $\psi_m(\lambda)=0$ for $\lambda\le\delta$, 
$\psi_m(\lambda)=1$ for $\lambda\ge \delta_1:=2\delta+\delta\ln m$.
Define the function $\widetilde\Psi_m(\lambda,\lambda')$, $\lambda,\lambda'\in [0, \infty)$, by
\begin{equation*}
\widetilde\Psi_m(\lambda,\lambda')= 
\begin{cases}
\frac{\psi_m(\lambda)-\psi_m(\lambda')}{\lambda-\lambda'} & \lambda \neq\lambda', \\
\partial_\lambda\psi_m(\lambda) & \lambda = \lambda'.
\end{cases}
\end{equation*}
We note that an equivalent way to define $\widetilde \Psi_m$ is 
\begin{equation*}
\widetilde\Psi_m(\lambda,\lambda') = \delta^{-1}\int_0^1\rho_m(\lambda'(1-\sigma)/\delta+\lambda \sigma/\delta)d\sigma,
\end{equation*}
which follows from 
\begin{equation*}
\psi_m(\lambda)-\psi_m(\lambda')=\int_{\lambda'}^{\lambda}  \partial_\tau (\psi_m(\tau))d\tau = \delta^{-1} \int_{\lambda'}^{\lambda} \rho_m(\tau/\delta) d \tau
\end{equation*}
 followed by the substitution $\tau= \lambda'(1-\sigma)+\lambda \sigma$. Set
\begin{equation*}
\Psi_m(\lambda,\lambda')=(\lambda+\lambda')^{-1}\widetilde\Psi_m(\lambda,\lambda'), \qquad \lambda, \, \lambda' \in [0, \infty).
\end{equation*}
which is well defined since $\widetilde{\Psi}_m(\lambda, \lambda') = 0$ if $\lambda, \, \lambda' \le \delta$. We need the following

\begin{lemma} \label{6.1}
For large $m$, the function $\Psi_m(\cdot,\lambda')\in C^\infty(\R^+)$ satisfies the bounds
\begin{equation}\label{eq:6.3}
\left|\partial_\lambda^k\Psi_m(\lambda,\lambda')\right|\le C^{k+1}k!(\lambda+1)^{-1}\left((\lambda'+1)^{-1}\ln m\right)^j,\quad j=0,1,
\end{equation}
for all $\lambda, \lambda'\in \R^+$ and all integers $0\le k\le m$ with some constant $C>0$ depending on $\delta$.
\end{lemma}

\begin{proof} For $\lambda \pm \lambda'\neq 0$, we have
\begin{equation}\label{eq:6.4}
\left|\partial_\lambda^k(\lambda \pm \lambda')^{-1}\right|\le k!|\lambda \pm \lambda'|^{-k-1}
\end{equation}
for every integer $k\ge 0$. To prove (\ref{eq:6.3}), suppose first that
$0 \le \lambda < \delta/2$. Then
\begin{equation*}
\widetilde \Psi_m (\lambda, \lambda') =
\begin{cases}
 0 &  0 \le \lambda' < \delta, \\
   -\psi_m(\lambda')(\lambda - \lambda')^{-1} & \lambda' > 3\delta/4,
\end{cases}
\end{equation*}
and in the latter case, $|\lambda  -  \lambda'|  \gtrsim \lambda' + 1$, $\lambda + \lambda' \gtrsim \lambda + 1$. Here and for the rest of the proof, the implicit constants may depend on $\delta$ but are independent of $\delta_1$. Thus when $0 \le \lambda < \delta/2$ and $\lambda' > 3\delta/4$ we have, by \eqref{eq:6.4} and the Leibniz rule,
 \begin{equation*}
 \begin{split}
 \left| \partial_\lambda^k \Psi_m(\lambda, \lambda') \right| &\le \left| \partial^k_\lambda \frac{1}{(\lambda + \lambda')(\lambda - \lambda')} \right| \\
 &\le C^{k+1}k!(\lambda +1)^{-1} (\lambda' +1)^{-1}  \\
 &\le C^{k+1}k!(\lambda'+1)^{-1}\left((\lambda'+1)^{-1}\ln m\right)^j.
 \end{split}
 \end{equation*}
 
Next, assume $\lambda\ge 2\delta_1$. Then,
\begin{equation*}
\widetilde \Psi_m (\lambda, \lambda') =
\begin{cases}
 (1-\psi_m(\lambda'))(\lambda-\lambda')^{-1} &  0 \le \lambda'  < 3\delta_1 /2, \\
 0 & \lambda' > \delta_1,
\end{cases}
\end{equation*}
and in the former case $|\lambda - \lambda'| \gtrsim \lambda+1$, $\lambda + \lambda' \gtrsim \lambda' + 1
$. So again by \eqref{eq:6.4} and the Leibniz rule $ \left| \partial_\lambda^k \Psi_m(\lambda, \lambda') \right| \le C^{k+1}k!(\lambda+1)^{-1}\left((\lambda'+1)^{-1}\ln m\right)^j.$

Let now $\delta/3 < \lambda < 3\delta_1$. Then 
\begin{equation*}
\widetilde\Psi_m(\lambda,\lambda')= \begin{cases} 
 \psi_m(\lambda)(\lambda-\lambda')^{-1} & 0 \le \lambda' < \delta/4, \\
\delta^{-1}\int_0^1\rho_m(\lambda'(1-\sigma)/\delta+\lambda\sigma/\delta)d\sigma & \delta/5 < \lambda' < 5\delta_1, \\
(\psi_m(\lambda)- 1)(\lambda-\lambda')^{-1}& \lambda' > 4\delta_1.
\end{cases}
\end{equation*}
In the first case $|\lambda-\lambda'|\gtrsim \lambda+1$, $\lambda+\lambda'\gtrsim \lambda'+1$, and in the third case
 $|\lambda-\lambda'|\gtrsim \lambda'+1$, $\lambda+\lambda'\gtrsim \lambda+1$. Thus  (\ref{eq:6.3}) follows from (\ref{eq:6.2}), (\ref{eq:6.4}) and the Leibniz rule as above. In the second case $\lambda+\lambda'\gtrsim \lambda+1$. Moreover, differentiating under the under the integral and using \eqref{eq:6.1} gives
\begin{equation*}
\left| \partial_\lambda^\ell \widetilde{\Psi}_m(\lambda,\lambda') \right|
\le C^{\ell+1}\ell!,
\qquad 0\le \ell\le k\le m.
\end{equation*}
 In addition, $(\lambda'+1)^{-1}\ln m \ge (5\delta_1 + 1)^{-1} \ln m \gtrsim 1$. Therefore (\ref{eq:6.3}) follows in the second case from (\ref{eq:6.4}), the previous estimate, and the Leibniz rule.
\end{proof}

Next we extend the function $\psi_m(\lambda)$ to a smooth, even function on the whole $\R$ (recall that $\psi_m(\lambda) = 0$ for $\lambda \le \delta$) From now on our use of the notation $\psi_m$ is meant to refer to this extension. We then smoothly extend $\Psi_m(\lambda, \lambda') = (\lambda + \lambda')^{-1} \widetilde{\Psi}_m(\lambda, \lambda')$ to the whole space $\R_\lambda \times \R_{\lambda'}$ by
\begin{equation*}
\Psi_m(\lambda, \lambda') = 
\begin{cases}
(\lambda^2 - (\lambda')^2)^{-1} (\psi_m(\lambda) - \psi_m(\lambda'))& \lambda \neq \lambda', \\
(2 \lambda)^{-1} \partial_\lambda \psi_m(\lambda) & \lambda = \lambda'.
\end{cases}
\end{equation*}
Observe the smoothness of this extension is justified by noticing that $\Psi_m(\lambda, \lambda') = 0$ for $|\lambda|,\, |\lambda'| \le \delta$, while near the set $\{(\lambda, \lambda') : 0 \neq \lambda = \lambda'\}$, 
\begin{equation*}
\Psi_m(\lambda, \lambda') = (\lambda + \lambda')^{-1}\int_0^1 (\partial \psi_m)((1 - \sigma)\lambda' + \sigma \lambda)d\sigma.
\end{equation*}
 We also have $\Psi_m(\lambda, \lambda') = \Psi_m(|\lambda|, |\lambda'|)$.

Rearranging the expression for $\Psi_m(\lambda, \lambda')$ yields
\begin{equation*}
\psi_m(\lambda') - \psi_m(\lambda) = ((\lambda')^2 - \lambda^2) \Psi_m(\lambda, \lambda'), \qquad \lambda, \, \lambda' \in \R,
\end{equation*}
whence for all $\lambda \in \R$,
\begin{equation*}
\psi_m(P^{1/2})-\psi_m(\lambda)=(P-\lambda^2)\Psi_m(\lambda,P^{1/2}), \qquad \text{on } D(P).
\end{equation*}
This identity will be used at a later stage in the analysis.

In view of Lemma \ref{E.2}, given an integer $m\gg 1$
there is a real-valued function $\rho_m^\sharp\in C_0^\infty(\R)$, $\rho_m^\sharp\ge 0$, 
$\int_{-\infty}^\infty\rho_m^\sharp(\sigma)d\sigma=1$, such that $\rho_m^\sharp(\sigma)=0$ for $\sigma\le 1$ and $\sigma\ge 2$, and 
 \begin{equation*}
\left|\partial_\sigma^k\rho_m^\sharp(\sigma)\right|\le (C\ln m)^{k+1}k!,\quad\forall\sigma\in \R,
\end{equation*}
for all integers $0\le k\le m$  with a constant $C>0$ independent of $k$ and $m$. Define the functions $\psi_m^\sharp$ and
$\Psi_m^\sharp$ by replacing the function $\rho_m$ by $\rho_m^\sharp$ in the definition of $\psi_m$ and
$\Psi_m$, respectively. Then $\psi_m^\sharp(\lambda) = 0$ for $\lambda \le \delta$ and $\psi_m^\sharp(\lambda) = 1$ for $\lambda \ge 2 \delta$. Repeating the prior arguments of this Section, with $\delta_1$ replaced by $\delta$, shows that $\Psi_m^\sharp$ satisfies (\ref{eq:6.3}) with an additional factor $(\ln m)^{k+1}$ in the right-hand side.

Using Lemma \ref{6.1} we will prove

\begin{prop} \label{6.2}
For large $m$ and $t\gg 1$ we have the estimates
\begin{equation}\label{eq:6.5}
\begin{split}
&\int_t^\infty \left\|\mu\cos(t'\sqrt{P})\psi_m(P^{1/2})\mu f\right\|_{L^2}^2dt'+
\int_t^\infty \left\|\mu\nabla^{\ell}P^{-1/2}\sin(t'\sqrt{P})\psi_m(P^{1/2})\mu f\right\|_{L^2}^2dt'\\
&\le  C^{2m+2}(m!)^2 (\ln m)^2 t^{-2m}\|f\|_{L^2}^2,\qquad\forall f\in L^2,
\end{split}
\end{equation}
\begin{equation}\label{eq:6.6}
\begin{split}
&\int_t^\infty \left\|\mu P^{1/2}\sin(t'\sqrt{P})\psi_m(P^{1/2})\mu f\right\|_{L^2}^2dt'+
\int_t^\infty \left\|\mu\nabla^{\ell}\cos(t'\sqrt{P})\psi_m(P^{1/2})\mu f\right\|_{L^2}^2dt'\\
&\le  C^{2m+2}(m!)^2(\ln m)^2 t^{-2m}\|f\|_{H^1}^2,\qquad \forall f\in H^1,
\end{split}
\end{equation}
where $\mu(x)=e^{-c\langle x\rangle/2}$, $\ell\in\{0,1\}$, 
 and $C>0$ is a constant independent of $m$, $t$ and $f$. 
 
 Furthermore, the estimates (\ref{eq:6.5}) and (\ref{eq:6.6}) hold with $\psi_m$ replaced by $\psi_m^\sharp$ 
 with an additional factor $(\ln m)^{2m+2}$ in the right-hand sides.
 
 If the dimension $d$ is odd and the condition (\ref{eq:1.6})
 is assumed, then the estimates (\ref{eq:6.5}) and (\ref{eq:6.6}) hold for all integers $m\ge 0$ with $\psi_m$ replaced by $1$. 
\end{prop}

\begin{proof} 
We first prove

\begin{lemma} \label{6.3}
Given any $g\in D(P^{1/2})$,
\begin{equation}\label{eq:6.7}
\|\nabla g\|_{L^2}\lesssim \|g\|_{L^2}+\|P^{1/2}g\|_{L^2},
\end{equation}
\begin{equation}\label{eq:6.8}
\|P^{1/2}g\|_{L^2}\lesssim \|g\|_{L^2}+\|\nabla g\|_{L^2}.
\end{equation}
\end{lemma}

{\it Proof.} If $g \in D(P)$, 
\begin{equation*}
\|P^{1/2}g\|_{L^2}^2=\langle Pg,g\rangle_{L^2}=\|(i\nabla+b)g\|_{L^2}^2+\langle Vg,g\rangle_{L^2}\ge 
\frac{1}{2}\|\nabla g\|_{L^2}^2-O(1)\|g\|_{L^2}^2,
\end{equation*}
which implies (\ref{eq:6.7}) for $g \in D(P)$. Observe that for the case a), we used the estimate (\ref{semiboundedness}) from Appendix \ref{self-adjointness mag Schro appendix} with $\epsilon = 1/2$. For the case b) we used Green's formula. 
Similarly,
\begin{equation*}
\|P^{1/2}g\|_{L^2}^2=\|(i\nabla+b)g\|_{L^2}^2+\langle Vg,g\rangle_{L^2}\le 
\|\nabla g\|_{L^2}^2+O(1)\|g\|_{L^2}^2,
\end{equation*}
which is (\ref{eq:6.8}) for $g \in D(P)$.

Having showed (\ref{eq:6.7}) and (\ref{eq:6.8}) for $g \in D(P)$, they follow for any $g \in D(P^{1/2})$. This is because $D(P)$ is dense in $D(P^{1/2})$ with respect to the norm $g \mapsto ( \|g\|^2_{L^2} + \|P^{1/2} g \|^2_{L^2})^{1/2}$.
\end{proof}

We will also need the following

\begin{lemma} \label{6.4}
There is a constant $0<\gamma<1$ such that for all $0\le t\le\gamma$ we have the estimates
\begin{equation}\label{eq:6.9}
\left\|\mu^{-1}\cos(t\sqrt{P})\mu f\right\|_{L^2}
+\left\|\mu^{-1}P^{-1/2}\sin(t\sqrt{P})\mu f\right\|_{L^2}\lesssim \|f\|_{L^2}, \quad\forall f\in L^2,
\end{equation}
\begin{equation}\label{eq:6.10}
\left\|\mu^{-1}P^{1/2}\sin(t\sqrt{P})\mu f\right\|_{L^2}\lesssim \|f\|_{H^1},\quad\forall f\in H^1.
\end{equation}
\end{lemma}

\begin{proof} 
Let $u(\cdot,t) \in C^2(\R; L^2(\Omega)) \cap C^1(\R; D(P^{1/2}))$, $u(\cdot,t)\in D(P)$ be a solution of the equation $(\partial_t^2+P)u=0$.
Let $\eta\in C^2(\overline\Omega)$ be a bounded real-valued function with bounded derivatives, independent of the variable $t$. Set
\begin{equation*}
\mathcal{E}(t)=\left\|\eta u(t)\right\|_{L^2}^2+\left\|\eta\partial_tu(t)\right\|_{L^2}^2+\left\|\eta(i\nabla+b)u(t)\right\|_{L^2}^2.
\end{equation*}
We have the identity
\begin{equation}\label{eq:6.11}
\frac{d\mathcal{E}(t)}{dt}=\mathcal{E}_1(t)+\mathcal{E}_2(t),
\end{equation}
where 
\begin{equation*}
\mathcal{E}_1(t)=2{\rm Re}\langle\eta\partial_t u(t),\eta u(t)\rangle_{L^2}
\end{equation*}
and
\begin{equation*}
\begin{split}
\mathcal{E}_2(t)&=2{\rm Re}\langle\eta\partial_t^2 u(t),\eta \partial_tu(t)\rangle_{L^2}+2{\rm Re}\langle\eta(i\nabla+b)\partial_t u(t),\eta(i\nabla+b)u(t)\rangle_{L^2}\\
&=-2{\rm Re}\langle\eta^2 Pu(t),\partial_t u(t)\rangle_{L^2}+2{\rm Re}\langle\eta^2(i\nabla+b)u(t),(i\nabla+b)\partial_tu(t)\rangle_{L^2}\\
&=-2{\rm Re}\langle P\eta^2u(t),\partial_tu(t)\rangle_{L^2}+2{\rm Re}\langle[P,\eta^2]u(t),\partial_tu(t)\rangle_{L^2}\\
&+2{\rm Re}\langle\eta^2(i\nabla+b)u(t),(i\nabla+b)\partial_tu(t)\rangle_{L^2}\\
&=-2{\rm Re}\langle (i\nabla+b)\eta^2u(t),(i\nabla+b)\partial_tu(t)\rangle_{L^2}\\
&-2{\rm Re}\langle(V\eta^2-[P,\eta^2])u(t),\partial_tu(t)\rangle_{L^2}\\
&+2{\rm Re}\langle\eta^2(i\nabla+b)u(t),(i\nabla+b)\partial_tu(t)\rangle_{L^2}\\ 
&=-2{\rm Re}\langle [i\nabla,\eta^2]u(t),(i\nabla+b)\partial_tu(t)\rangle_{L^2}\\
&-2{\rm Re}\langle(V\eta^2-[P,\eta^2])u(t),\partial_tu(t)\rangle_{L^2}\\
&=2{\rm Re}\langle \mathcal{M}(\eta)u(t),\partial_tu(t)\rangle_{L^2},
\end{split}
\end{equation*}
where
\begin{equation*}
\begin{split}
\mathcal{M}(\eta) &= [P, \eta^2] - V\eta^2 - (i\nabla + b)\cdot[i\nabla, \eta^2] \\
&= [- \Delta, \eta^2] + 2ib \cdot \nabla \eta^2 - V\eta^2 -(i\nabla+b)\cdot[i\nabla,\eta^2].
\end{split}
\end{equation*}
For the fourth equality in the above calculation, we used Green's formula in the case b).

Let 
$\chi\in C_0^\infty(\R^d;[0,1])$ be such that $\chi(x)=1$ for $|x|\le a$, $\chi(x)=0$ for $|x|\ge 2a$, where $a > 0$ is fixed sufficiently large so that $\overline{\mathcal{O}} \subset \{|x| < a \} $ Given
$k\in\mathbb{N}$, set $\mu_k(x)=e^{-\frac{c}{2}\langle x\rangle\chi(x/k)}$. Clearly, we have $\mu_k(x)^{-1}\le\mu(x)^{-1}$
and $|\partial_x^\alpha(\mu_k(x)^{-1})|\lesssim\mu_k(x)^{-1}$ for $|\alpha|\le 1$ uniformly in $k$. 
We are going to use the above identities with
$\eta=\mu_k^{-1}$. Observe that
\begin{equation*}
\left|\mathcal{M}(\mu_k^{-1})u\right|\lesssim \mu_k^{-2}(|u|+|\nabla u|)
\end{equation*}
uniformly in $k$, which implies
\begin{equation}\label{eq:6.12}
|\mathcal{E}_j(t)|\lesssim \mathcal{E}(t),\quad j=1,2,
\end{equation}
uniformly in $k$. By (\ref{eq:6.11}) and (\ref{eq:6.12}) we obtain
\begin{equation}\label{eq:6.13}
\mathcal{E}(t)\le\mathcal{E}(0)+C\int_0^t\mathcal{E}(t')dt'
\end{equation}
with a constant $C>0$ independent of $k$. Integrating (\ref{eq:6.13}) leads to the inequality
\begin{equation*}
\int_0^\gamma\mathcal{E}(t)dt\le\mathcal{E}(0)+C\gamma\int_0^\gamma\mathcal{E}(t)dt
\end{equation*}
for any $0<\gamma\le 1$. Taking $\gamma\le (2C)^{-1}$, we obtain 
\begin{equation*}
\int_0^\gamma\mathcal{E}(t)dt\le 2\mathcal{E}(0),
\end{equation*}
which combined with (\ref{eq:6.13}) yield
\begin{equation}\label{eq:6.14}
\mathcal{E}(t)\le C\mathcal{E}(0)
\end{equation}
for $0\le t\le\gamma$ with a new constant $C>0$ independent of $k$. Clearly, (\ref{eq:6.14}) implies
\begin{equation}\label{eq:6.15}
\sum_{j=0}^1\left\|\mu_k^{-1}\partial_t^ju(\cdot,t)\right\|^2_{L^2}\le C\mathcal{E}(0).
\end{equation}
We now apply (\ref{eq:6.15}) to the function
\begin{equation*}
u=P^{-1/2}\sin(t\sqrt{P})\mu f,\qquad f\in D(P).
\end{equation*}
Since $u|_{t=0}=0$, we have
\begin{equation}\label{eq:6.16}
\mathcal{E}(0)=\left\|\mu_k^{-1}\mu f\right\|^2_{L^2}\le \left\|f\right\|^2_{L^2}.
\end{equation}
By (\ref{eq:6.15}), (\ref{eq:6.16}) and Fatou's lemma,
\begin{equation*}
\begin{split}
&\left\|\mu^{-1}\cos(t\sqrt{P})\mu f\right\|_{L^2}^2
+\left\|\mu^{-1}P^{-1/2}\sin(t\sqrt{P})\mu f\right\|_{L^2}^2\\
&\le\liminf_{k\to\infty}\left\|\mu_k^{-1}\cos(t\sqrt{P})\mu f\right\|_{L^2}^2
+\liminf_{k\to\infty}\left\|\mu_k^{-1}P^{-1/2}\sin(t\sqrt{P})\mu f\right\|_{L^2}^2\\
&\le C\left\|f\right\|^2_{L^2},
\end{split}
\end{equation*}
which proves (\ref{eq:6.9}) for $f \in D(P)$. But then (\ref{eq:6.9}) holds for any $f \in L^2(\Omega)$ since $D(P)$ is dense in $L^2(\Omega)$. 

 To prove (\ref{eq:6.10}) we apply (\ref{eq:6.15}) to the function
\begin{equation*}
u=\cos(t\sqrt{P})\mu f,\qquad f\in D(P).
\end{equation*}
Since $\partial_tu|_{t=0}=0$, we have
\begin{equation}\label{eq:6.17}
\mathcal{E}(0)=\left\|\mu_k^{-1}\mu f\right\|^2_{L^2}+\left\|\mu_k^{-1}(i\nabla+b)\mu f\right\|^2_{L^2}\lesssim \left\|f\right\|^2_{H^1}
\end{equation}
uniformly in $k$. Now (\ref{eq:6.10}) for $f \in D(P)$ follows from (\ref{eq:6.15}), (\ref{eq:6.17}) and Fatou's lemma.

Having showed (\ref{eq:6.10}) for $f \in D(P)$, it holds for any $f \in H^1(\Omega)$ by (\ref{eq:6.7}) and the fact that $D(P)$ is dense in $D(P^{1/2})$ with respect to the norm $g \mapsto (\|g \|^2_{L^2} + \|P^{1/2}g\|^2_{L^2})^{1/2}$.  \eproof

Let $\phi\in C^\infty(\R)$ be such that $\phi(t)=0$ for $t\le\gamma/3$ and $\phi(t)=1$ for $t\ge\gamma/2$. Let $u(\cdot, t) \in C^2(\R ; L^2(\Omega)) \cap C^1(\R; D(P^{1/2})$, $u(\cdot, t) \in D(P)$ be a solution of the equation $(\partial_t^2+P)u(t)=0$. Then the function 
$\phi u$ satisfies the equation
\begin{equation*}
\left(\partial_t^2+P\right)(\phi u)(t)=v(t),
\end{equation*}
where 
\begin{equation*}
v(t)=\phi''(t)u(t)+2\phi'(t)\partial_tu(t).
\end{equation*}
By Duhamel's formula we get
\begin{equation}\label{eq:6.18}
(\phi u)(t)=\int_0^t\sin\left((t-t')\sqrt{P}\right)P^{-1/2}v(t')dt'.
\end{equation}
On the other hand, we have the formula
\begin{equation}\label{eq:6.19}
(P-(\lambda-i\varepsilon)^2)^{-1}=\int_0^\infty e^{-it(\lambda-i\varepsilon)}\sin\left(t\sqrt{P}\right)P^{-1/2}dt,\quad\lambda\in\R,\, 0<\varepsilon<1.
\end{equation}
It follows from (\ref{eq:6.18}) and (\ref{eq:6.19}) that the Fourier transform of the function 
$e^{-\varepsilon t}\partial_t^j(\phi u)$, $\varepsilon > 0$, $j=0,1$, satisfies
\begin{equation}\label{eq:6.20}
\widehat{e^{-\varepsilon t}\partial_t^j(\phi u)} =i^j(\lambda-i\varepsilon)^j
(P-(\lambda-i\varepsilon)^2)^{-1}\widehat{v}(\lambda-i\varepsilon), \qquad \lambda \in \R,\, \varepsilon > 0.
\end{equation}
Note that since $v(t)$ is compactly supported in $t$, it's Fourier transform $\widehat{v}$ is an entire function.

We apply (\ref{eq:6.20}) to the function
\begin{equation*}
u(t)=\sin(t\sqrt{P})P^{-1/2}\psi_m(P^{1/2})\mu f,\qquad f\in D(P),
\end{equation*}
In this situation,
\begin{equation*}
v(t)=\psi_m(P^{1/2})\mathcal{V}(t),
\end{equation*}
\begin{equation*}
\mathcal{V}(t) \defeq \phi''(t)P^{-1/2}\sin(t\sqrt{P})\mu f+2\phi'(t)\cos(t\sqrt{P})\mu f.
\end{equation*}  
By (\ref{eq:6.20}) and the identity
\begin{equation*}
\psi_m(P^{1/2})-\psi_m(\lambda)=(P-\lambda^2)\Psi_m(\lambda,P^{1/2})
\end{equation*}
 we get, with $j=0,1$, 
\begin{equation}\label{eq:6.21}
\begin{split}
&\widehat{e^{-\varepsilon t}\partial_t^j(\phi u)}(\lambda)\\
&=i^j(\lambda-i\varepsilon)^j(P-(\lambda-i\varepsilon)^2)^{-1}
\psi_m(P^{1/2})\widehat{\mathcal{V}}(\lambda-i\varepsilon)\\
&=i^j(\lambda-i\varepsilon)^j(P-(\lambda-i\varepsilon)^2)^{-1}\psi_m(\lambda)\widehat{\mathcal{V}}(\lambda-i\varepsilon)\\
&+i^j(\lambda-i\varepsilon)^j(P-(\lambda-i\varepsilon)^2)^{-1}(P-\lambda^2)\Psi_m(\lambda,P^{1/2})
\widehat{\mathcal{V}}(\lambda-i\varepsilon)\\
&=i^j(\lambda-i\varepsilon)^j(P-(\lambda-i\varepsilon)^2)^{-1}\psi_m(\lambda)\widehat{\mathcal{V}}(\lambda-i\varepsilon)\\
&+i^j(\lambda-i\varepsilon)^j\Psi_m(\lambda,P^{1/2})\widehat{\mathcal{V}}(\lambda-i\varepsilon)\\
&-(2i\varepsilon\lambda+\varepsilon^2)i^j(\lambda-i\varepsilon)^j(P-(\lambda-i\varepsilon)^2)^{-1}\Psi_m(\lambda,P^{1/2})
\widehat{\mathcal{V}}(\lambda-i\varepsilon).
\end{split}
\end{equation}
If $j = 1$, we multiply the left-hand side of (\ref{eq:6.21}) by $\mu$, while if $j = 0$, we let the operator 
$\mu\nabla^{\ell}$, $\ell=0,1$, act on 
the left-hand side of (\ref{eq:6.21}). We would like to make disappear the last term in the right-hand side of (\ref{eq:6.21})
by taking the limit $\varepsilon\to 0$. To this end we need the following lemma, the proof of which is given in the next section.

\begin{lemma} \label{6.5} 
For $m$ large and $\ell\in\{0,1\}$, there exists $C > 0$ so that
\begin{equation}\label{eq:6.22}
\left\| \mu\nabla^{\ell}(P-(\lambda-i\varepsilon)^2)^{-1}\Psi_m(\lambda,P^{1/2})\mu \right\| \le C. 
\end{equation}
for all $\lambda \in \R$ and $0 < \varepsilon < 1$. The constant $C>0$ may depend on $m$ and $\ell$ but is independent of  $\lambda$ and $\varepsilon$. The estimate \eqref{eq:6.22} also holds with $\Psi_m$ replaced by $\Psi^{\sharp}_m$.
\end{lemma}

It follows from Lemma \ref{6.4} that the $L^2$ norm of the function $\mu^{-1}\widehat{\mathcal{V}}(\lambda-i\varepsilon)$ 
is bounded uniformly in $\varepsilon$. Therefore, by Lemma \ref{6.5} we conclude that the $L^2$ norm
of the last term in the right-hand side of (\ref{eq:6.21}) is $O(\varepsilon)$ and hence 
tends to zero as $\varepsilon\to 0$. Thus, from (\ref{eq:6.21}) we get the identities
\begin{equation}\label{eq:6.23}
\widehat{\mu\partial_t(\phi u)}(\lambda)=i\lambda\mu(P-(\lambda-i0)^2)^{-1}\mu
\psi_m(\lambda)\mu^{-1}\widehat{\mathcal{V}}(\lambda)+i\lambda\mu\Psi_m(\lambda,P^{1/2})\widehat{\mathcal{V}}(\lambda),
\end{equation}
\begin{equation}\label{eq:6.24}
\widehat{\mu\nabla^{\ell}\phi u}(\lambda)=\mu\nabla^{\ell}(P-(\lambda-i0)^2)^{-1}\mu
\psi_m(\lambda)\mu^{-1}\widehat{\mathcal{V}}(\lambda)
+\mu\nabla^{\ell}\Psi_m(\lambda,P^{1/2})\widehat{\mathcal{V}}(\lambda),
\end{equation}
for $\lambda\in\R$.
Hence, if $\ell+j\le 1$, given any integer $0\le k\le m$, using the Leibniz formula, we obtain 
\begin{equation}\label{eq:6.25}
\begin{split}
&\widehat{t^k\mu\partial_t^j\nabla^{\ell}\phi u}(\lambda)=(-i\partial_\lambda)^k
\left(\mu(i\lambda)^j\nabla^{\ell}(P-(\lambda-i0)^2)^{-1}\mu
\psi_m(\lambda)\mu^{-1}\widehat{\mathcal{V}}(\lambda)\right)\\
&+\mu(-i\partial_\lambda)^k\left((i\lambda)^j\nabla^{\ell}\Psi_m(\lambda,P^{1/2})
\widehat{\mathcal{V}}(\lambda)\right)\\
&=\sum_{\nu=0}^k\frac{k!}{\nu!(k-\nu)!}(-i\partial_\lambda)^\nu
\left(\mu(i\lambda)^j\nabla^{\ell}(P-(\lambda-i0)^2)^{-1}\mu
\psi_m(\lambda)\right)\mu^{-1}\widehat{t^{k-\nu}\mathcal{V}}(\lambda)\\
&+\mu\sum_{\nu=0}^k\frac{k!}{\nu!(k-\nu)!}(-i\partial_\lambda)^\nu
\left((i\lambda)^j\nabla^{\ell}\Psi_m(\lambda,P^{1/2})\right)\widehat{t^{k-\nu}\mathcal{V}}(\lambda).
\end{split}
\end{equation}
It follows from the estimate (\ref{eq:3.8}) in the case a) and (\ref{eq:4.18}) in the case b), together with (\ref{eq:6.2}),
\begin{equation}\label{eq:6.26}
\left\|\partial_\lambda^\nu\left(\mu\nabla^{\ell}(P-(\lambda-i0)^2)^{-1}\mu
\psi_m(\lambda)\right)\right\|+
\left\|\partial_\lambda^\nu\left(\mu\lambda(P-(\lambda-i0)^2)^{-1}\mu
\psi_m(\lambda)\right)\right\|\le C^{\nu+1}\nu!.
\end{equation}
By (\ref{eq:6.3}) and (\ref{eq:6.7}) we also have 
\begin{equation}\label{eq:6.27}
\left\|\nabla^{\ell}\partial_\lambda^\nu\left(\Psi_m(\lambda,P^{1/2})\right)\right\|\le 
\sum_{j=0}^1\left\|P^{j/2}\partial_\lambda^\nu\left(\Psi_m(\lambda,P^{1/2})\right)\right\|
\le C^{\nu+1}\nu!\ln m,
\end{equation}
\begin{equation}\label{eq:6.28}
\left\|\partial_\lambda^\nu\left(\lambda\Psi_m(\lambda,P^{1/2})\right)\right\|
\le C^{\nu+1}\nu!.
\end{equation}
By (\ref{eq:6.25}) through (\ref{eq:6.28}),
\begin{equation}\label{eq:6.29}
\left\|\widehat{t^k\mu\partial_t\phi u}(\lambda)\right\|_{L^2}+\left\|\widehat{t^k\mu\nabla^{\ell}\phi u}(\lambda)\right\|_{L^2}
\le C^{k+1}k!\ln m\sum_{\nu=0}^k\left\|\mu^{-1}\widehat{t^{k-\nu}\mathcal{V}}(\lambda)\right\|_{L^2}.
\end{equation}
Now let $\mathcal{J}$ denote either $\nabla^{\ell}$ or $\partial_t$. 
By Plancherel's identity and (\ref{eq:6.29}) together with (\ref{eq:6.9}), we obtain
\begin{equation}\label{eq:6.30}
\begin{split}
&\int_{-\infty}^\infty t^{2k}\left\|\mu(\phi\mathcal{J}u)(t)\right\|_{L^2}^2dt
=C\int_{-\infty}^\infty\left\|\widehat{t^k\mu\phi\mathcal{J}u}(\lambda)\right\|_{L^2}^2d\lambda\\
&\le C^{2k+2}(k!)^2(\ln m)^2\sum_{\nu=0}^k\int_{-\infty}^\infty\left\|\mu^{-1}
\widehat{t^{k-\nu}\mathcal{V}}(\lambda)\right\|_{L^2}^2d\lambda\\
&\le C^{2k+2}(k!)^2(\ln m)^2\sum_{\nu=0}^k\int_0^\gamma t^{2k-2\nu}
\left\|\mu^{-1}\mathcal{V}(t)\right\|_{L^2}^2dt\\
&\le C^{2k+2}(k!)^2(\ln m)^2\int_0^\gamma
\left\|\mu^{-1}\mathcal{V}(t)\right\|_{L^2}^2dt\\
&\le C^{2k+2}(k!)^2(\ln m)^2\|f\|_{L^2}^2,
\end{split}
\end{equation}
where $C>0$ denotes a constant that changes from line to line. 
Since
\begin{equation*}
t^{2k}\int_t^\infty \left\|\mu\mathcal{J}u(t')\right\|_{L^2}^2dt'\le \int_{-\infty}^\infty (t')^{2k}\left\|\mu(\phi 
\mathcal{J}u)(t')\right\|_{L^2}^2dt',\quad t>1,
\end{equation*}
the estimate (\ref{eq:6.5}) for $f \in D(P)$ follows from (\ref{eq:6.30}) with $k=m$. But in turn we get (\ref{eq:6.5}) for any $f \in L^2(\Omega)$ by Fatou's lemma and the fact that $D(P)$ is dense in $L^2(\Omega)$.  

To get (\ref{eq:6.6}), we apply the same strategy to the function
\begin{equation*}
u(t)=\cos(t\sqrt{P})\psi_m(P^{1/2})\mu f,\qquad f\in D(P).
\end{equation*}
In this case we have 
\begin{equation*}
\mathcal{V}(t)=\phi''(t)\cos(t\sqrt{P})\mu f+2\phi'(t)P^{1/2}\sin(t\sqrt{P})\mu f.
\end{equation*}
We use (\ref{eq:6.10}) to conclude that (\ref{eq:6.30}) holds for $f \in D(P)$, with $\|f\|_{L^2}$ in the right-hand side
replaced by $\|f\|_{H^1}$. In the same way as above we arrive at (\ref{eq:6.6}) for $f \in D(P)$. But then we conclude (\ref{eq:6.6}) for any $f \in H^1(\Omega)$ using Fatou's Lemma, (\ref{eq:6.7}) and the fact that $D(P)$ is dense in $D(P^{1/2})$ with respect to the norm $f \mapsto (\|f \|^2_{L^2} + \|P^{1/2} f\|^2_{L^2})^{1/2}$.

Clearly, the estimates (\ref{eq:6.26}), (\ref{eq:6.27}) and (\ref{eq:6.28}) hold with $\psi_m$ and $\Psi_m$ replaced by $\psi_m^\sharp$ 
and $\Psi_m^\sharp$ 
 with an additional factor $(\ln m)^{\nu+1}$ in the right-hand sides, which implies the estimates (\ref{eq:6.5}) and (\ref{eq:6.6})
 in this case.

In odd dimensions, under the condition (\ref{eq:1.6}),
the above analysis works with $\psi_m\equiv 1$ because so does the resolvent estimate (\ref{eq:6.26}).
\end{proof}

\begin{proof}[Proof of Theorem \ref{1.1}] 
We will derive the estimate (\ref{eq:1.7}) from (\ref{eq:6.5}). We apply the identity (\ref{eq:6.11}) with $\eta=\mu$ to the function 
\begin{equation*}
u(t) = \sin(t \sqrt{P}) P^{-1/2} \psi_m(P^{1/2}) \mu f,\quad f \in D(P).
\end{equation*}
 We get
\begin{equation}\label{eq:6.31}
\begin{split}
&\frac{d}{dt}\left(\left\|\mu \partial_tu(t)\right\|^2_{L^2}+\left\|\mu(i\nabla+b)u(t)\right\|^2_{L^2}
+\left\|\mu u(t)\right\|^2_{L^2}\right)\\
&=2{\rm \Re}\langle \mathcal{N}(\mu)u(t),\mu \partial_tu(t)\rangle_{L^2} + 2{\rm Re}\langle\mu\partial_t u(t),\mu u(t)\rangle_{L^2}\\
&\le 2\left\|\mu \partial_tu(t)\right\|^2_{L^2}+\left\|\mu u(t)\right\|^2_{L^2}+\left\|\mathcal{N}(\mu)u(t)\right\|^2_{L^2},
\end{split}
\end{equation}
where
\begin{equation*}
\mathcal{N}(\mu)=\mu^{-1}\left([-\Delta, \mu^2] + 2ib \cdot \nabla \mu^2 - V \mu^2 - (i\nabla+b)\cdot[i\nabla,\mu^2]\right) =
\sum_{\ell=0}^1O_{\ell}(\mu)\nabla^{\ell}.
\end{equation*}
By (\ref{eq:6.31}), for all $T>t>1$, 
\begin{equation}\label{eq:6.32}
\begin{split}
&\left\|\mu \partial_tu(t)\right\|_{L^2}^2+\left\|\mu(i\nabla+b)u(t)\right\|_{L^2}^2+\left\|\mu u(t)\right\|_{L^2}^2\\
&\lesssim \left\|\mu \partial_tu(T)\right\|_{L^2}^2+\left\|\mu(i\nabla+b)u(T)\right\|_{L^2}^2+\left\|\mu u(T)\right\|_{L^2}^2\\
&+\int_t^T\left\|\mu\partial_t u(t')\right\|_{L^2}^2dt'+\sum_{\ell=0}^1\int_t^T\left\|\mu\nabla^{\ell}u(t')\right\|_{L^2}^2dt'.
\end{split}
\end{equation}
On the other hand, it follows from (\ref{eq:6.30}) with $k=0$ that there exists a sequence $T_j\to\infty$ such that
\begin{equation}\label{eq:6.33}
\lim_{T_j\to\infty} \left(\left\|\mu \partial_tu(T_j)\right\|_{L^2}^2+\left\|\mu(i\nabla+b)u(T_j)\right\|_{L^2}^2
+\left\|\mu u(T_j)\right\|_{L^2}^2\right)=0.
\end{equation}
Therefore, using (\ref{eq:6.32}) with $T=T_j$ and taking the limit as $T_j\to\infty$, in view of (\ref{eq:6.33}), we obtain
\begin{equation}\label{eq:6.34}
\begin{split}
&\left\|\mu \partial_tu(t)\right\|_{L^2}^2+\left\|\mu(i\nabla+b)u(t)\right\|_{L^2}^2+\left\|\mu u(t)\right\|_{L^2}^2\\
&\lesssim\int_t^\infty\left\|\mu\partial_t u(t')\right\|_{L^2}^2dt'+\sum_{\ell=0}^1
\int_t^\infty\left\|\mu\nabla^{\ell}u(t')\right\|_{L^2}^2dt'.
\end{split}
\end{equation}
By (\ref{eq:6.5}) and (\ref{eq:6.34}),
\begin{equation}\label{eq:6.35}
\begin{split}
&\left\|\mu\partial_t u(t)\right\|_{L^2}+\left\|\mu(i\nabla+b)u(t)\right\|_{L^2}+\left\|\mu u(t)\right\|_{L^2}\\
&\le C^{m+1}m!(\ln m)t^{-m}\|f\|_{L^2}\le C\ln m(Cemt^{-1})^me^{-m}\|f\|_{L^2}\le C (\ln m)e^{-m}\|f\|_{L^2}
\end{split}
\end{equation}
if $Cemt^{-1}\le 1$. We now let $m$ be the biggest integer $\le t(Ce)^{-1}$. 
Then 
\begin{equation*}
(\ln m) e^{-m}\le e \ln (t(Ce)^{-1}) e^{-t(Ce)^{-1}} \le C_1 e^{-c_1t}.
\end{equation*}
Taking $\psi_{\delta,t}(\sigma)=\psi_m(\sigma)$, (\ref{eq:6.35}) shows the desired bound holds for elements $f \in D(P)$. But then (\ref{eq:1.7}) immediately follows because $D(P)$ is dense in $L^2(\Omega)$.

To get (\ref{eq:1.8}), we apply the above analysis
to the function 
\begin{equation*}
u(t) = \cos(t \sqrt{P}) \psi_m(P^{1/2})\mu f,\quad f \in D(P),
\end{equation*}
 and use the estimate (\ref{eq:6.6}) instead of (\ref{eq:6.5}) to conclude that the estimate (\ref{eq:6.35}) holds with $\|f\|_{L^2}$ in the right-hand side replaced by $\|f\|_{H^1}$. Then (\ref{eq:1.8}) follows from (\ref{eq:6.35}), (\ref{eq:6.7}) and the fact that $D(P)$ is dense in $D(P^{1/2})$ with respect to the norm $f \mapsto (\|f\|^2_{L^2} + \|P^{1/2} f\|^2_{L^2})^{1/2}$. 
 
 If we replace $\psi_m$ by $\psi_m^\sharp$, the estimate (\ref{eq:6.35}) takes the form
 \begin{equation}\label{eq:6.36}
\begin{split}
&\left\|\mu\partial_t u(t)\right\|_{L^2}+\left\|\mu(i\nabla+b)u(t)\right\|_{L^2}+\left\|\mu u(t)\right\|_{L^2}\\
&\le (C\ln m)^{m+1}m!t^{-m}\|f\|_{L^2}\le C(Cem(\ln m)t^{-1})^me^{-m}\|f\|_{L^2}\le Ce^{-m}\|f\|_{L^2}
\end{split}
\end{equation}
if $Cem\ln mt^{-1}\le 1$. We now let $m$ be the biggest integer such that $m\ln m\le t(Ce)^{-1}$. Clearly, for $t\gg 1$
we have $m\sim ct/\ln t$ with some constant $c>0$. Therefore, in this case (\ref{eq:1.7}) with
$\psi_{\delta,t}^\sharp(\sigma)=\psi_m^\sharp(\sigma)$ follows from (\ref{eq:6.36}).
The proof of (\ref{eq:1.8}) is similar.

In odd dimensions, under the condition (\ref{eq:1.6}),
the above analysis works with $\psi_m\equiv 1$ because so do the estimates (\ref{eq:6.5}) and (\ref{eq:6.6}).
\end{proof}

\section{Proof of Lemma \ref{6.5}}

We prove the result for $\Psi_m$. The proof of $\Psi^\sharp_m$ is identical, with $\delta_1$ below replaced by $ \delta$. 

For $0\le |\lambda| \le\delta/2$,
\begin{equation*}
\Psi_m(\lambda,P^{1/2})=\psi_m(P^{1/2})(P-\lambda^2)^{-1}
\end{equation*}
and $|x-\lambda|\ge\delta/2$ if $x\in{\rm supp}\,\psi_m$. 
Hence in this case we have
\begin{equation*}
\begin{split}
&\left\|\mu\nabla^{\ell}(P-(\lambda-i\varepsilon)^2)^{-1}\Psi_m(\lambda,P^{1/2})\mu\right\|\\
&\lesssim\sum_{j=0}^1\left\|P^{j/2}(P-(\lambda-i\varepsilon)^2)^{-1}\psi_m(P^{1/2})(P-\lambda^2)^{-1}\right\|\\
&\lesssim \sup_{x\in{\rm supp}\,\psi_m}(|x|+1)|x^2-(\lambda-i\varepsilon)^2|^{-1}|x^2-\lambda^2|^{-1}\lesssim 1
\end{split}
\end{equation*}
uniformly in $\varepsilon$ and $\lambda$. Note that to get the second inequality we used \eqref{eq:6.7}.
 Let now $|\lambda| \ge\delta/2$. Let $\chi_1,\, \chi_2,\, \chi_3 \in C^\infty(\R^+)$ be such that 
$\chi_1+\chi_2+\chi_3\equiv 1$ on $\R^+$, 
$\chi_1(\lambda')=1$ for 
$\lambda'\le \delta/3$, $\chi_1(\lambda')=0$ for $\lambda'\ge \delta/2$,
$\chi_3(\lambda')=0$ for 
$\lambda'\le 3\delta_1$, $\chi_3(\lambda')=1$ for $\lambda'\ge 4\delta_1$. Then
\begin{equation*}
\chi_1(P^{1/2})\Psi_m(\lambda,P^{1/2})=-\psi_m(\lambda)\chi_1(P^{1/2})(P-\lambda^2)^{-1}
\end{equation*}
and $|x-\lambda|\ge\delta/2$ if $x\in{\rm supp}\,\chi_1$ and $\lambda\in{\rm supp}\,\psi_m$. Hence 
\begin{equation*}
\begin{split}
&\left\|\mu\nabla^{\ell}(P-(\lambda-i\varepsilon)^2)^{-1}\chi_1(P^{1/2})\Psi_m(\lambda,P^{1/2})\mu\right\|\\
&\lesssim\sum_{j=0}^1\left\|P^{j/2}(P-(\lambda-i\varepsilon)^2)^{-1} \psi_m(\lambda) \chi_1(P^{1/2})(P-\lambda^2)^{-1}\right\|\\
&\lesssim \sup_{x\in{\rm supp}\,\chi_1,\,\lambda\in{\rm supp}\,\psi_m}(|x|+1)|x^2-(\lambda-i\varepsilon)^2|^{-1}|x^2-\lambda^2|^{-1}\lesssim 1
\end{split}
\end{equation*}
uniformly in $\varepsilon$ and $\lambda$. Furthermore, we have \begin{equation*}
\chi_3(P^{1/2})\Psi_m(\lambda,P^{1/2})=(1-\psi_m)(\lambda)\chi_3(P^{1/2})(P-\lambda^2)^{-1}
\end{equation*}
and $|x-\lambda|\ge\delta_1$ if $x\in{\rm supp}\,\chi_3$ and $\lambda\in{\rm supp}\,(1-\psi_m)$. In the same way as above, we get
\begin{equation*}
\left\|\mu\nabla^{\ell}(P-(\lambda-i\varepsilon)^2)^{-1}\chi_3(P^{1/2})\Psi_m(\lambda,P^{1/2})\mu\right\|\lesssim 1
\end{equation*}
uniformly in $\varepsilon$ and $\lambda$. It remains to show that 
\begin{equation}\label{eq:7.1}
\left\|\mu\nabla^{\ell}(P-(\lambda-i\varepsilon)^2)^{-1}\chi_2(P^{1/2})\Psi_m(\lambda,P^{1/2})\mu\right\|\lesssim 1
\end{equation}
uniformly in $\varepsilon$ and $\lambda$. Clearly, the function 
\begin{equation*}
\varphi(x)=\chi_2(x^{1/2})\Psi_m(\lambda,x^{1/2})
\end{equation*}
belongs to $C_0^\infty(\R)$ and $|\partial_x^n\varphi(x)|\lesssim 1$ for $n\le 2$. 
Then there is an almost analytic extension, $\widetilde\varphi$, of $\varphi$ on $\C$ such that
$\widetilde\varphi|_{\R}=\varphi$ and 
\begin{equation}\label{eq:7.2}
\left|\overline{\partial}\widetilde\varphi(z)\right|\le C_N|{\rm Im}\,z|^N\sum_{n=0}^N\sup_x|\partial_x^n\varphi(x)|,\quad\forall N\ge 0,
\end{equation}
where the constant $C_N$ does not depend on the function $\varphi$; $\widetilde\varphi$ is supported in 
a complex neighbourhood of supp$\,\varphi$. In our case 
$\widetilde\varphi$ is supported in a complex neighbourhood of supp$\,\chi_2(x^{1/2})$.
We are going to use the Helffer-Sj\"ostrand
formula 
\begin{equation}\label{eq:7.3}
\varphi(P)=\frac{1}{\pi}\int\overline{\partial}\widetilde\varphi(z)(P-z)^{-1}dx dy,\quad z=x+iy.
\end{equation}
Thus the operator in (\ref{eq:7.1}) can be written in the form
\begin{equation}\label{eq:7.4}
\frac{1}{\pi}\int\overline{\partial}\widetilde\varphi(z)
\mu\nabla^{\ell}(P-(\lambda-i\varepsilon)^2)^{-1}(P-z)^{-1}\mu dx dy.
\end{equation}
From the resolvent identity
\begin{equation*}
\begin{split}
&(\lambda^2 -\varepsilon^2-z)(P-(\lambda-i\varepsilon)^2)^{-1}(P-z)^{-1}\\
&=2i\varepsilon\lambda(P-(\lambda-i\varepsilon)^2)^{-1}(P-z)^{-1}\\
&+(P-(\lambda-i\varepsilon)^2)^{-1}-(P-z)^{-1},
\end{split}
\end{equation*}
we get
\begin{equation*}
\begin{split}
&|\lambda^2-\varepsilon^2-z|\left\|\mu\nabla^{\ell}(P-(\lambda-i\varepsilon)^2)^{-1}(P-z)^{-1}\mu\right\|\\
&\le 2\varepsilon|\lambda|\left\|\nabla^{\ell}(P-(\lambda-i\varepsilon)^2)^{-1}\right\|\left\|(P-z)^{-1}\right\|\\
&+\left\|\mu\nabla^{\ell}(P-(\lambda-i\varepsilon)^2)^{-1}\mu\right\|+
\left\|(P-z)^{-1}\right\|\\
&\lesssim|{\rm Im}\,z|^{-1}\varepsilon|\lambda|\sum_{j=0}^1\left\|P^{j/2}(P-(\lambda-i\varepsilon)^2)^{-1}\right\|\\
&+ 1+|{\rm Im}\,z|^{-1},
\end{split}
\end{equation*}
where we have used the resolvent estimate (\ref{eq:3.2}) in the case a) and (\ref{eq:4.2}) in the case b). Moreover,
\begin{equation} \label{eq:7.5}
\varepsilon |\lambda| \|(P-(\lambda-i\varepsilon)^2)^{-1}\| \lesssim 1, 
\end{equation}
and
\begin{equation} \label{eq:7.6}
\varepsilon |\lambda| \|P^{1/2}(P-(\lambda-i\varepsilon)^2)^{-1}\| \lesssim |\lambda| + 1.
\end{equation}
Indeed, \eqref{eq:7.6} is a consequence of \eqref{eq:7.5}, the identity
\begin{equation*}
P(P - (\lambda - i\varepsilon)^2)^{-1} = I + (\lambda - i \varepsilon)^2 (P - (\lambda - i\varepsilon)^2)^{-1}
\end{equation*}
and the estimate 
\begin{equation*}
\|P^{1/2}u\|^2_{L^2} =  \langle Pu, u \rangle_{L^2} \le \|Pu\|_{L^2} \|u\|_{L^2}, \qquad u \in D(P). 
\end{equation*}
Combining the previous inequalities implies
\begin{equation*}
|\lambda^2-\varepsilon^2-z|\left\|\mu\nabla^{\ell}(P-(\lambda-i\varepsilon)^2)^{-1}(P-z)^{-1}\mu\right\| \lesssim |\lambda| + 1 + |{\rm Im}\,z|^{-1}.
\end{equation*}
On the other hand, 
\begin{equation*}
|\lambda^2-\varepsilon^2-z|=((\lambda^2 - \varepsilon^2 - {\rm Re}\,z)^2 +  |{\rm Im}\,z|^2)^{1/2}\ge |{\rm Im}\,z|,
\end{equation*}
while for large $|\lambda|$ and $z\in{\rm supp}\,\widetilde{\varphi}$ we have
\begin{equation*}
|\lambda^2-\varepsilon^2-z|\ge |\lambda|^2/2.
\end{equation*}
 Therefore, we have on the support of $\widetilde{\varphi}$,
\begin{equation}\label{eq:7.7}
\left\|\mu\nabla^{\ell}(P-(\lambda-i\varepsilon)^2)^{-1}(P-z)^{-1}\mu\right\|\lesssim  |{\rm Im}\,z|^{-2}
\end{equation}
uniformly in $\varepsilon$ and $\lambda$. It follows from (\ref{eq:7.7}) together with the formula (\ref{eq:7.3}) and (\ref{eq:7.2})
with $N=2$ that the operator (\ref{eq:7.4}) is bounded uniformly in $\varepsilon$ and $\lambda$, which in turn implies (\ref{eq:7.1}).

\appendix

\section{Self-adjointness of the magnetic Schr\"odinger operator on $L^2(\R^d)$} \label{self-adjointness mag Schro appendix}

In this appendix we discuss self-adjointness of \eqref{eq:1.1} when $\Omega = \R^d$, $b \in L^\infty(\R^d ; \R^d)$ is not identically zero, 
and $V \in L^\infty(\R^d ; \R)$. In this case the self-adjoint realization of \eqref{eq:1.1} we use throughout the paper is constructed via a sesquilinear form as follows. On $H^1(\R^d) \times H^1(\R^d)$ put
\begin{equation} \label{form for mag Schro}
q (u, v) \defeq  \int_{\R^d} \nabla \overline{u} \cdot \nabla v  - i \overline{u}b \cdot \nabla v + i vb \cdot \nabla \overline{u} + (V + |b|^2)\overline{u} v dx.
\end{equation}
For any $\epsilon \ge 0$, by Cauchy-Schwarz and Young's inequality,
\begin{equation*}
\begin{split}
\big|\int_{\R^d}  i ub \cdot \nabla \overline{u} - i\overline{u}b \cdot \nabla u dx\big| & = 2 \big| \Im \int_{\R^d}  \overline{u}b \cdot \nabla u dx\big|  \\
& \le (1 - \epsilon)  \| \nabla u \|^2_{L^2} +  \frac{1}{1 - \epsilon}  \| bu  \|^2_{L^2},  
\end{split}
\end{equation*} 
whence,
\begin{equation} \label{semiboundedness}
q (u, u) + \| u\|_{L^2}^2 \ge \epsilon  \| \nabla u \|^2_{L^2} + \int_{\R^d} \big(V + 1 - \frac{\epsilon}{1 - \epsilon} \| b \|^2_{L^\infty(\R^d; \R^d)} \big) |u|^2 dx, \qquad \epsilon \ge 0.
\end{equation}
The last estimate shows that \eqref{form for mag Schro} is semibounded and closed in the sense of \cite[Section 2.3]{kn:Te}. Moreover, if $V \ge 0$, setting $\epsilon = 0$ yields $q(u,u) \ge 0$.

By \cite[Theorem 2.14]{kn:Te}, there exists a unique, densely defined self-adjoint operator $P$ whose quadratic form domain is $H^1(\R^n)$, and whose associated sesquilinear form coincides with $q$ on $H^1(\R^n)$. The domain of $P$ is 
\begin{equation*}
\begin{gathered}
D(P) = \{ u \in H^1(\R^d) : \text{there is } \tilde{u} \in L^2(\R^d) \text{ such that }  q(u, v) = \langle \tilde{u}, v\rangle_{L^2} \text{ for all } v \in H^1(\R^d)\}, \\
 Pu = \tilde{u}.
\end{gathered}
\end{equation*}
The equality $q(u, v) = \langle \tilde{u}, v\rangle_{L^2}$ for $u \in D(P)$ and $v \in H^1(\R^d)$ shows that, in the sense of distributions on $\R^d$, 
\begin{equation*}
Pu = -\Delta u + i \nabla \cdot (u b) + ib \cdot \nabla u + (V + |b|^2)u.
\end{equation*}
 Here, if $(\cdot, \cdot)$ denotes distributional pairing, we define the divergence of a distribution $u$ by $(\nabla \cdot u, v) \defeq -(u, \nabla \cdot v)$.    

In Section \ref{resolv bds mag Schro section} we make use of several mapping properties of $(P- z)^{-1}$ for $z$ in the resolvent set $\rho(P)$ of $P$, which we now formulate. Recall that the negative index Sobolev space $H^{-1}(\R^d)$ is isometrically isomorphic to the dual space of $H^1(\R^d)$ under the mapping $H^{-1}(\R^d) \ni u \mapsto \langle u , \cdot \rangle_{L^2}$, and that
\begin{equation*}
\| v\|_{H^{-1}} = \sup_{0 \neq u \in H^1} \frac{ \langle v, u \rangle_{L^2}}{\| u\|_{H^1}}. 
\end{equation*}
We show that for any $z \in \rho(P)$, $(P- z)^{-1}$ maps boundedly from $H^{-1}(\R^d)$ to $H^1(\R^d)$. In particular, there exists $C_z > 0$ so that 
\begin{equation} \label{H minus 1 to H1 bd}
\|(P- z)^{-1} u\|_{H^1} \le C_z \| u\|_{H^{-1}}, \qquad u \in C^\infty_0(\R^d). 
\end{equation}
We reuse the constant $C_z$ below. Its precise value changes from line to line, but it stays independent of $u$. 

By fixing $0 < \epsilon \ll 1$ in  \eqref{semiboundedness}, we get $C > 0$ independent of $u \in L^2(\R^d)$ such that
\begin{equation} \label{H1 to L2 bd}
\begin{split}
\| (P - z)^{-1} u \|^2_{H^1} &\le C (\| (P - z)^{-1} u  \|^2_{L^2} + q((P - z)^{-1} u , (P - z)^{-1} u ) \\
&= C (\| (P - z)^{-1} u  \|^2_{L^2} + \langle P(P - z)^{-1} u , (P - z)^{-1} u \rangle_{L^2}) \\
&= C (\| (P - z)^{-1} u  \|^2_{L^2} + \langle u , (P - z)^{-1} u \rangle_{L^2} + \overline{z} \| (P - z)^{-1} u\|^2_{L^2}). 
\end{split}
\end{equation}
Since $\| (P - z)^{-1} u \|_{L^2} \le C_z \|u\|_{L^2}$ this implies  
\begin{equation*}
\|(P- z)^{-1} u\|_{H^1} \le C_z \| u\|_{L^2}, \qquad u \in L^2(\R^d). 
\end{equation*}
Thus, for any $u \in C^\infty_0(\R^d)$ and $v \in L^2(\R^d)$,  
\begin{equation*}
|\langle (P- z)^{-1} u, v \rangle_{L^2}| = | \langle  u,  (P- \overline{z})^{-1}v \rangle_{L^2}| \le \|u \|_{H^{-1}} \|(P-z)^{-1}v \|_{H^1} \le C_z \| u \|_{H^{-1}} \| v \|_{L^2}. 
\end{equation*}
Therefore we conclude, 
\begin{equation} \label{H minus 1 to L2 bd}
\|(P- z)^{-1} u\|_{L^2} \le C_z \| u\|_{H^{-1}}, \qquad u \in C^\infty_0(\R^d). 
\end{equation}
Now \eqref{H minus 1 to H1 bd} follows from  \eqref{H1 to L2 bd} and \eqref{H minus 1 to L2 bd}. Indeed, for $u \in C^\infty_0(\R^d)$,

\begin{equation*}
\begin{split}
\| (P &- z)^{-1} u \|^2_{H^1}\\
 & \le  C (\| (P - z)^{-1} u  \|^2_{L^2} + \langle u , (P - z)^{-1} u \rangle_{L^2} + \overline{z} \|(P - z)^{-1} u\|^2_{L^2}) \\
& \le C_z (\| (P - z)^{-1} u  \|_{L^2} + \| u \|_{H^{-1}}) \| (P-z)^{-1} u\|_{H^1} \\
& \le C_z \|u \|_{H^{-1}}  \| (P-z)^{-1} u\|_{H^1} .
\end{split}
\end{equation*}

We use \eqref{H minus 1 to H1 bd} to verify a resolvent identity that we apply in Section \ref{resolv bds mag Schro section}. Let $P_0 = - \Delta$ denote the free Laplacian on $\R^d$. We show that for $u \in H^2(\R^d)$ and $z \in \rho(P)$,
\begin{equation} \label{key resolv id for mag Schro}
(P - z)^{-1} (\widetilde{V} + i\nabla \cdot b + ib \cdot \nabla ) u = u - (P - z)^{-1}(P_0 - z)u, 
\end{equation}
where $\widetilde{V} = V + |b|^2$ and we note that the divergence is a bounded operator $L^2(\R^d; \C^d) \to H^{-1}(\R^d)$. To show \eqref{key resolv id for mag Schro}, let $u_k \in C^\infty_0(\R^d ; \C^d)$ converge to $bu$ in $L^2(\R^d; \C^d)$, in which case by \eqref{H minus 1 to H1 bd} we have that $(P - z)^{-1} i\nabla \cdot u_{k} $ converges to $(P - z)^{-1} i\nabla \cdot bu$ in $L^2(\R^d)$. Then for any $v \in L^2(\R^d)$, 
\begin{equation*}
\begin{split}
\langle v, &(P - z)^{-1} (\widetilde{V} + i\nabla \cdot b + ib \cdot \nabla ) u \rangle_{L^2}  \\
&= \langle  -ib(P- \overline{z})^{-1} v,  \nabla u \rangle_{L^2(\R^d; \C^d)}+ \langle  \widetilde{V} 
(P- \overline{z})^{-1} v, u  \rangle_{L^2}  + \lim_{k \to \infty} \langle v, (P - z)^{-1} i \nabla \cdot u_k  \rangle_{L^2} \\
&= \langle  (P- \overline{z})^{-1} v, -\Delta u  \rangle_{L^2} -  \langle  ib(P- \overline{z})^{-1} v, 
 \nabla u \rangle_{L^2(\R^d; \C^d)}+ \langle ( \widetilde{V} + i b \cdot \nabla - \overline{z})
  (P- \overline{z})^{-1} v, u  \rangle_{L^2}  \\
&- \langle   v, (P- z)^{-1} (P_0 -z)   u  \rangle_{L^2}
\end{split}
\end{equation*}
Now \eqref{key resolv id for mag Schro} follows since $(-\Delta + i \nabla \cdot b + ib \cdot \nabla + \widetilde{V} - \overline{z}) 
(P- \overline{z})^{-1} v = v$ in the sense of distributions.

 \section{Analytic functions in a strip}
 
 Let the function $f(\lambda)$ be analytic in $\{\lambda\in\mathbb{C}: A<{\rm Re}\,\lambda,\,|{\rm Im}\,\lambda|<\gamma\}$, 
 where $\gamma>0$ is some constant, while $A$ is either a constant or $A=-\infty$. Let also $f$ satisfy in this region the bound
 \begin{equation}\label{eq:B.1}
|f(\lambda)|\le M
\end{equation}
with some constant $M>0$. Let $\gamma_1$ be any constant such that $0<\gamma_1<\gamma$. 
If $A$ is a constant we take any constant $A_1$ such that $A_1>A$. If $A=-\infty$ we take $A_1=-\infty$.
For such functions we will prove the following

 \begin{lemma} \label{B.1}
There exists a constant $C>0$ such that for $A_1\le{\rm Re}\,\lambda,\,|{\rm Im}\,\lambda|\le\gamma_1$ we have the bounds
\begin{equation}\label{eq:B.2}
|\partial_\lambda^kf(\lambda)|\le C^{k+1}k!
\end{equation}
for every integer $k\ge 0$, and
\begin{equation}\label{eq:B.3}
|f(\lambda)-f({\rm Re}\,\lambda)|\le C|{\rm Im}\,\lambda|.
\end{equation}
\end{lemma}

{\it Proof.} The bound (\ref{eq:B.2}) follows from (\ref{eq:B.1}) and the Cauchy formula
\begin{equation}\label{eq:B.4}
\partial_\lambda^kf(\lambda)=\frac{k!}{2\pi i}\int_{|z-\lambda|=\sigma}\frac{f(z)}{(z-\lambda)^{k+1}}dz
\end{equation}
for every integer $k\ge 0$, where $\sigma$ is a constant such that $0<\sigma<\gamma-\gamma_1$. If $A$ and $A_1$ are constants we also require that
$\sigma<A_1-A$. Furthermore, we have
\begin{equation*}
f(\lambda)-f({\rm Re}\,\lambda)=i{\rm Im}\,\lambda f'({\rm Re}\,\lambda+it)
\end{equation*}
with some real $t$ such that $|t|\le|{\rm Im}\,\lambda|$, where $f'$ denotes the first derivative of 
$f$. Therefore, (\ref{eq:B.3}) follows from (\ref{eq:B.2}) with $k=1$. 
\eproof

\section{Resolvent bounds for the free resolvent} \label{free resolv appendix}

 The following estimates for the free resolvent are well-known and therefore we omit the proof.
 
  \begin{lemma} \label{C.1}
 Let $d\ge 2$, $s>1/2$, and let $\alpha$ and $\beta$ be multi-indices such that $|\alpha|+|\beta|\le 2$. 
 Then, given any $\delta>0$, we have the bound
 \begin{equation}\label{eq:C.1}
 \left\|\langle x\rangle^{-s}\partial_x^\alpha(P_0-\lambda^2\pm i\varepsilon)^{-1}\partial_x^\beta\langle x\rangle^{-s}\right\|
\le C\lambda^{|\alpha|+|\beta|-1},\quad \lambda\ge\delta,
 \end{equation}
 uniformly in $\varepsilon$. If $|\alpha|+|\beta|\ge 1$, we have the bound
 \begin{equation}\label{eq:C.2}
 \left\|\langle x\rangle^{-s}\partial_x^\alpha(P_0-\lambda^2\pm i\varepsilon)^{-1}\partial_x^\beta\langle x\rangle^{-s}\right\|
\le C,\quad 0<\lambda\le\delta,
 \end{equation}
 uniformly in $\varepsilon$. If $d\ge 3$ and $s>1$ we have the bound
 \begin{equation}\label{eq:C.3}
 \left\|\langle x\rangle^{-s}(P_0-\lambda^2\pm i\varepsilon)^{-1}\langle x\rangle^{-s}\right\|
\le C,\quad 0<\lambda\le\delta,
 \end{equation}
 uniformly in $\varepsilon$. 
 \end{lemma}
 
 Next lemma is well-known when the function $\mu$ is compactly supported. We show that it still holds with
 $\mu=e^{-c\langle x\rangle/2}$, $c>0$.

\begin{lemma} \label{C.2}
Let $\alpha$ and $\beta$ be multi-indices such that $|\alpha|+|\beta|\le 2$. Then, 
there exists a constant $\gamma_0>0$ such that the operator-valued function  
\begin{equation}\label{eq:C.4}
\mu\partial_x^\alpha(P_0-\lambda^2)^{-1}\partial_x^\beta\mu:L^2\to L^2
\end{equation}
extends analytically from $\mathbb{C}^-$ to $\mathcal{L}_{\gamma_0}$
 and satisfies the bound 
\begin{equation}\label{eq:C.5}
\left\|\mu\partial_x^\alpha(P_0-\lambda^2)^{-1}\partial_x^\beta\mu\right\|\le C(|\lambda|+1)^{|\alpha|+|\beta|-1}
\end{equation}
for $\lambda\in \mathcal{L}_{\gamma_0}$, $|\lambda|\ge\delta$, $\delta>0$ being arbitrary, 
 with a constant $C$ depending on $\delta$. We also have the bound
\begin{equation}\label{eq:C.6}
\left\|\mu\partial_x^\alpha(P_0-\lambda^2)^{-1}\partial_x^\beta\mu-\mu\partial_x^\alpha
(P_0-({\rm Re}\,\lambda)^2)^{-1}\partial_x^\beta\mu\right\|\le C|{\rm Im}\,\lambda|(|\lambda|+1)^{|\alpha|+|\beta|-1}
\end{equation}
for $\lambda\in \mathcal{L}_{\gamma'_0}$, $|\lambda|\ge\delta$, where $0<\gamma'_0<\gamma_0$ is a constant.
When $d\ge 3$ is odd (\ref{eq:C.5}) holds for all $\lambda\in \mathcal{L}_{\gamma_0}$ and (\ref{eq:C.6}) holds for all $\lambda\in \mathcal{L}_{\gamma'_0}$. 
\end{lemma}

{\it Proof.} Note first that (\ref{eq:C.6}) follows from (\ref{eq:C.5}) and (\ref{eq:B.3}). Since the operator 
$\partial_x^\beta$ commutes with the free resolvent, it suffices to prove the lemma with $\beta=0$ and $|\alpha|\le 2$.

It is well-known that the kernel $K(x,y;\lambda)$ of the free resolvent 
\begin{equation*}
R_0(\lambda)=(P_0-\lambda^2)^{-1},\quad {\rm Im}\,\lambda<0,
\end{equation*}
can be expressed in terms of the Hankel functions by the formula
\begin{equation*}
K(x,y;\lambda)=i2^{-2}(2\pi)^{-\frac{d-2}{2}}\lambda^{\frac{d-2}{2}}|x-y|^{-\frac{d-2}{2}}H^-_{\frac{d-2}{2}}(\lambda|x-y|).
\end{equation*}
It is also well-known that $z^{\frac{d-2}{2}}H^-_{\frac{d-2}{2}}(z)$ extends analytically from $\mathbb{C}^-$ to
the complex plane $\mathbb{C}$ if $d$ is odd and to the Riemann surface of the logarithm if $d$ is even
and satisfies the bounds 
\begin{equation}\label{eq:C.7}
\left|H^-_{\frac{d-2}{2}}(z)\right|\lesssim 
\left\{
\begin{array}{lll}
|\log|z||+1&\mbox{for}& |z|\le 1,\,d=2,\\
 |z|^{-\frac{d-2}{2}}&\mbox{for}& |z|\le 1,\,d\ge 3,\\
 |z|^{-1/2}e^{{\rm Im}\,z}&\mbox{for}& |z|\ge 1,\,d\ge 2.
\end{array}
\right.
\end{equation}
Hence the kernel $K$ extends analytically in $\lambda$ from $\mathbb{C}^-$ to
the complex plane $\mathbb{C}$ if $d$ is odd and to the Riemann surface of the logarithm if $d$ is even and  
satisfies the bound
\begin{equation}\label{eq:C.8}
|K(x,y;\lambda)|\lesssim G_d(|x-y|)+|\lambda|^{(d-3)/2}|x-y|^{-(d-1)/2}e^{{\rm Im}\,\lambda|x-y|},
\end{equation}
where 
\begin{equation*}
G_d(\sigma)=
\left\{
\begin{array}{lll}
|\log\sigma|+|\log|\lambda||+1&\mbox{if}& d=2,\\
 \sigma^{-d+2}&\mbox{if}& d\ge 3.
\end{array}
\right.
\end{equation*}
Fix a constant $0<\gamma_0<c/2$. Since
\begin{equation*}
\mu(x)\mu(y)\le e^{-c|x-y|/2},
\end{equation*}
we deduce from (\ref{eq:C.8}),
\begin{equation}\label{eq:C.9}
|\mu(x)K(x,y;\lambda)\mu(y)|\lesssim \left(G_d(|x-y|)
+|\lambda|^{(d-3)/2}|x-y|^{-(d-1)/2}\right)e^{-(c/2-\gamma_0)|x-y|}
\end{equation}
for ${\rm Im}\,\lambda\le \gamma_0$. It follows from (\ref{eq:C.9}) and Schur's lemma that the operator 
$\mu R_0(\lambda)\mu$,  
is bounded on $L^2$ for ${\rm Im}\,\lambda\le\gamma_0$ with norm $O\left(1+|\lambda|^{(d-3)/2}\right)$. 
Therefore the operator 
$\mu R_0(\lambda)\mu: L^2\to L^2$
extends analytically from $\mathbb{C}^-$ to $\mathcal{L}_{\gamma_0}$. 
This also implies the bound (\ref{eq:C.5}) with $\alpha=\beta=0$ for $|\lambda|\le 1$ if $d\ge 3$. 

To prove the bound (\ref{eq:C.5}) in the other cases we will follow \cite{kn:B2} where (\ref{eq:C.5}) with $\alpha=\beta=0$
is proved for compactly supported $\mu$ (see Proposition 2.1 of \cite{kn:B2}). When $\mu$ decays exponentially the proof
 is the same but we
will sketch the main points for the sake of completeness. It is based on the formula
\begin{equation}\label{eq:C.10}
K(x,y;\lambda)-K(x,y;-\lambda)=i2^{-1}(2\pi)^{1-d}\lambda^{d-2}\int_{\mathbb{S}^{d-1}}e^{i\lambda\langle x-y,w\rangle}dw,
\end{equation}
where $\mathbb{S}^{d-1}$ denotes the unit sphere in $\mathbb{R}^d$. From (\ref{eq:C.10}) we get the formula
\begin{equation}\label{eq:C.11}
\mu\partial_x^\alpha R_0(\lambda)\mu-\mu\partial_x^\alpha R_0(-\lambda)\mu =i2^{-1}(2\pi)^{1-d}\lambda^{d-2+|\alpha|}\mathcal{A}_\mu^{(\alpha)}(\lambda)
\mathcal{A}_\mu^{(0)}(\overline\lambda)^*,
\end{equation}
for any muli-index $\alpha$, where
\begin{equation*}
\mathcal{A}_\mu^{(\alpha)}(\lambda):L^2(\mathbb{S}^{d-1})\to L^2(\mathbb{R}^d)
\end{equation*}
is the operator with kernel
\begin{align*}
A_\mu^{(\alpha)}(x,w)=i^{|\alpha|}w^{\alpha}\mu(x)e^{i\lambda\langle x,w\rangle},\quad x\in\mathbb{R}^d,\,w\in \mathbb{S}^{d-1}.
\end{align*}
Our goal is to prove the bound
\begin{equation}\label{eq:C.12}
\left\|\mu\partial_x^\alpha R_0(\lambda)\mu-\mu\partial_x^\alpha R_0(-\lambda)\mu \right\|\lesssim |\lambda|^{-1+|\alpha|}
\end{equation}
for all $\lambda\in\mathbb{C}$, $\lambda\neq 0$, such that $|{\rm Im}\,\lambda|\le\gamma_0$.
In view of (\ref{eq:C.11}), it suffices to prove the bound
\begin{equation}\label{eq:C.13}
\left\|\mathcal{A}_\mu^{(\alpha)}(\lambda)\right\|_{L^2(\mathbb{S}^{d-1})\to L^2(\mathbb{R}^d)}\lesssim |\lambda|^{-\frac{d-1}{2}}.
\end{equation}
On the other hand, in view of Plancherel's identity, the norm in (\ref{eq:C.13}) is equivalent to the norm of the operator
\begin{equation*}
\mathcal{F}\mathcal{A}_\mu^{(\alpha)}(\lambda):L^2(\mathbb{S}^{d-1})\to L^2(\mathbb{R}^d),
\end{equation*}
where $\mathcal{F}$ is the Fourier transform. Since the kernel of this operator is equal to 
$i^{|\alpha|}w^{\alpha}(\mathcal{F}\mu)(\xi-i\lambda w)$,
by Schur's lemma it suffices to show that
\begin{equation}\label{eq:C.14}
\int_{\mathbb{R}^{d}}|(\mathcal{F}\mu)(\xi-i\lambda w)|d\xi\lesssim 1,\quad 
\int_{\mathbb{S}^{d-1}}|(\mathcal{F}\mu)(\xi-i\lambda w)|dw\lesssim |\lambda|^{-d+1}
\end{equation}
for all $\lambda\in\mathbb{C}$, $\lambda\neq 0$, such that $|{\rm Im}\,\lambda|\le\gamma_0$.
To this end, we will use that $(\mathcal{F}\mu)(\xi)$ extends to all $\xi\in\mathbb{C}^d$ such that 
$|{\rm Im}\,\xi|\le\gamma_0$ and satisfies the bounds
\begin{equation*}
\left|\xi^\beta(\mathcal{F}\mu)(\xi)\right|\lesssim \int_{\mathbb{R}^{d}}\left|\partial_x^\beta\mu(x)\right|e^{|{\rm Im}\,\xi||x|}dx
\lesssim \int_{\mathbb{R}^{d}}\mu(x)e^{|{\rm Im}\,\xi||x|}dx\lesssim\int_{\mathbb{R}^{d}}e^{-(c/2-\gamma_0)|x|}dx\lesssim 1
\end{equation*}
for all multi-indices $\beta$. 
Thus we obtain that given any integer $M\ge 0$ there is a constant $C_M>0$ such that
\begin{equation}\label{eq:C.15}
\left|(\mathcal{F}\mu)(\xi)\right|\le C_M(|\xi|+1)^{-M}
\end{equation}
for all $\xi\in\mathbb{C}^d$ such that $|{\rm Im}\,\xi|\le\gamma_0$. We now apply (\ref{eq:C.15}) with $\xi-i\lambda w$, 
$\xi\in\mathbb{R}^d$. Then the bounds (\ref{eq:C.14}) follow from (\ref{eq:C.15}) in the same way as in Section 2
of \cite{kn:B2}.

It is easy to see now that the estimate (\ref{eq:C.12}) implies (\ref{eq:C.5}). Indeed, since 
the bound (\ref{eq:C.5}) is trivial on ${\rm Im}\,\lambda=-\gamma_0$, by (\ref{eq:C.12}) with $|\alpha|\le 2$
we conclude that it also holds on ${\rm Im}\,\lambda=\gamma_0$. Then the Phragm\'en-Lindel\"of principle implies that 
(\ref{eq:C.5}) holds for $|{\rm Im}\,\lambda|\le\gamma_0$.
\eproof

\section{Poincar\'e inequality} \label{Poincare appendix}
 
 \begin{lemma} \label{D.1}
 Let $d\ge 3$. Then, given a function $b\in L^\infty(\mathbb{R}^d,\mathbb{R}^d)$, we have the inequality
 \begin{equation}\label{eq:D.1}
 \left\||x|^{-1}f\right\|_{L^2(\mathbb{R}^d)}\lesssim \left\|(i\nabla+b)f\right\|_{L^2(\mathbb{R}^d)}
 \end{equation}
 for all $f\in H^1(\mathbb{R}^d)$. If $\mathcal{O}\subset\mathbb{R}^d$, $d\ge 3$, is a bounded domain with smooth boundary such that 
$\Omega=\mathbb{R}^d\setminus\mathcal{O}$ is connected, then we have the inequality
 \begin{equation}\label{eq:D.2}
 \left\||x|^{-1}f\right\|_{L^2(\Omega)}\lesssim \left\|\nabla f\right\|_{L^2(\Omega)}
 \end{equation}
 for all $f\in H_0^1(\Omega)$.
 \end{lemma} 
 
 {\it Proof.} Clearly, it suffices to prove (\ref{eq:D.1}) for all functions $f\in C_0^1(\mathbb{R}^d)$.
 Let $(r,w)\in (0,\infty)\times \mathbb{S}^{d-1}$ be the polar coordinates and set
 \begin{equation*}
 u(r,w)=f(rw)e^{-i\int_0^rw\cdot b(\sigma w)d\sigma}.
 \end{equation*}
  We have
 \begin{equation*}
 \begin{split}
 \int_0^\infty r^{d-3}|u(r,w)|^2dr&=(d-2)^{-1}\int_0^\infty |u(r,w)|^2 (r^{d-2})'dr\\
 &=-2(d-2)^{-1}{\rm Re}\,\int_0^\infty u'(r,w)
 \overline{u(r,w)}r^{d-2}dr,
 \end{split}
 \end{equation*}
 where the prime notation denotes the first derivative of a function with respect to $r$. Hence
 \begin{equation*}
 \int_0^\infty r^{d-3}|u(r,w)|^2dr\lesssim 
 \left(\int_0^\infty r^{d-1}|u'(r,w)|^2dr\right)^{1/2} \left(\int_0^\infty r^{d-3}|u(r,w)|^2dr\right)^{1/2}
\end{equation*}
 which implies
 \begin{equation*}
 \int_0^\infty r^{d-3}|u(r,w)|^2dr\lesssim \int_0^\infty r^{d-1}|u'(r,w)|^2dr.
 \end{equation*}
 Integrating this inequlity with respect to $w$ leads to the estimate
 \begin{equation}\label{eq:D.3}
 \left\|r^{-1}u\right\|_{L^2(\mathbb{R}^d)}\lesssim \left\|\partial_r u\right\|_{L^2(\mathbb{R}^d)}.
 \end{equation}
 Observe now that
 \begin{equation}\label{eq:D.4}
 \left|\partial_r u\right|=\left|(i\partial_r+w\cdot b(rw))f\right|=\left|w\cdot(i\nabla+ b(rw))f\right|
 \le \left|(i\nabla+ b)f\right|.
 \end{equation}
 Clearly, (\ref{eq:D.1}) follows from (\ref{eq:D.3}) and (\ref{eq:D.4}). 
 
The inequality (\ref{eq:D.2}) follows from (\ref{eq:D.1}) in the following manner. Given $f \in H^1(\Omega)$ with $f = 0$ on $\partial \Omega$, by \cite[Theorem 2, section 5.5]{kn:EV}, there exists a sequence $f_k \in C^\infty_0(\Omega) \subseteq C^\infty_0(\R^d)$ converging to $f$ in $H^1$-norm. By taking a subsequence if necessary, which we still denote by $f_k$, we can suppose the $f_k$ converge pointwise almost everywhere to $f$ with respect to the Lebesgue measure. Then by Fatou's lemma and \eqref{eq:D.1},
 \begin{equation*}
 \| |x|^{-1} f\|^2_{L^2(\Omega)} \le \liminf_{k \to \infty} \||x|^{-1} f_k \|_{L^2(\Omega)} \lesssim \liminf_{k \to \infty}  \| \nabla f_k \|_{L^2(\Omega)} =   \| \nabla f \|_{L^2(\Omega)} .
 \end{equation*}
 \eproof 
 
\section{Construction of the function $\rho_m$}

We will prove the following

\begin{lemma} \label{E.1}
Given an integer $m\gg 1$ there is a real-valued function $\rho_m\in C_0^\infty(\mathbb{R})$ such that $\rho_m\ge 0$, 
$\int_{-\infty}^\infty \rho_m(\sigma)d\sigma=1$,
${\rm supp}\,\rho_m\subset[0,\ln m+1]$ and
\begin{equation}\label{eq:E.1}
\left|\partial_\lambda^k\rho_m(\lambda)\right|\le 3^{k+1}k!,\quad 0\le k\le m,\quad\forall \lambda\in\mathbb{R}.
 \end{equation}
\end{lemma}

{\it Proof.} It follows from the analysis in \cite[Chapter 1.3]{kn:H}. Here we sketch the main points. Let $H_1$ be 
the characteristic function of the interval $[0,1]$, that is, $H_1(\lambda)=1$ when $0<\lambda<1$, $H_1(\lambda)=0$ otherwise.
Given any $a>0$, set $H_a(\lambda)=a^{-1}H_1(\lambda/a)$. Clearly, $\int_{-\infty}^\infty H_a(\sigma)d\sigma=1$.
Let $a_k=(k+1)^{-1}$, $k=0,1,...$. Define $u_k$ to be the product of convolutions
\begin{equation*}
u_m=H_{a_0}\ast \cdots \ast H_{a_{m+1}}.
\end{equation*}
Then $u_m\in C^m(\mathbb{R})$ is supported in $[0,b_m]$, where
\begin{equation*}
b_m \defeq \sum_{k=0}^{m+1}a_k=\sum_{k=1}^{m+2}k^{-1}=\ln m+\gamma+O(m^{-1})<\ln m+\frac{3}{5}
\end{equation*}
for large $m$, where $0<\gamma<0.6$ is Euler's constant (see the proof of \cite[Theorem 1.3.5]{kn:H}). 
Moreover, the derivatives of $u_m$ satisfy the bounds
(see $(1.3.12)'$ of \cite{kn:H})
\begin{equation}\label{eq:E.2}
\left|\partial_\lambda^ku_m(\lambda)\right|\le 2^k(k+1)!\le 3^{k+1}k!,\quad 0\le k\le m,\quad\forall \lambda\in\mathbb{R}.
 \end{equation}
Clearly, we also have $\int_{-\infty}^\infty u_m(\sigma)d\sigma=1$. Let now $\phi\in C_0^\infty(\mathbb{R})$ be independent of $m$,
supported in $[0,1/5]$ and such that $\phi\ge 0$, 
$\int_{-\infty}^\infty \phi(\sigma)d\sigma=1$. Then the function $\rho_m=u_m\ast\phi$ satisfies the conditions of the lemma.
In particular the bounds (\ref{eq:E.1}) follow from (\ref{eq:E.2}).
\eproof

We can derive from the above lemma the following

\begin{lemma} \label{E.2}
Given an integer $m\gg 1$ there is a real-valued function $\rho_m^\sharp\in C_0^\infty(\mathbb{R})$ such that $\rho_m^\sharp\ge 0$, 
$\int_{-\infty}^\infty\rho_m^\sharp(\sigma)d\sigma=1$,
${\rm supp}\,\rho_m^\sharp\subset[0,1]$ and
\begin{equation}\label{eq:E.3}
\left|\partial_\lambda^k\rho_m^\sharp(\lambda)\right|\le (3\ln m)^{k+1}k!,\quad 0\le k\le m,\quad\forall \lambda\in\mathbb{R}.
 \end{equation}
\end{lemma}

{\it Proof.} Let $\rho_m$ be the function from Lemma \ref{E.1} and set 
\begin{equation*}
\rho_m^\sharp(\lambda)=(\ln m+1)\rho_m((\ln m+1)\lambda).
\end{equation*}
Clearly, the bounds (\ref{eq:E.3}) follow from (\ref{eq:E.1}).
\eproof

\end{document}